\documentclass[a4paper,fleqn,final]{cas-sc}

\usepackage{silence}
\usepackage[numbers,sort&compress]{natbib}

\usepackage{amsmath,amssymb,amsthm}
\usepackage{mathtools}          
\usepackage{bm}

\usepackage{graphicx}
\usepackage{float}              
\restylefloat{figure}           
\usepackage{caption}
\usepackage{subcaption}         
\usepackage{placeins}           

\graphicspath{{1st Paper/update paper/}}

\renewcommand{\topfraction}{0.9}         
\renewcommand{\bottomfraction}{0.8}      
\renewcommand{\textfraction}{0.07}       
\renewcommand{\floatpagefraction}{0.5}                                           
\usepackage{algorithm}
\usepackage{algpseudocode}

\usepackage{booktabs}

\usepackage{xcolor}
\usepackage{tikz}
\usetikzlibrary{calc,arrows,arrows.meta,shapes,positioning,fit}

\definecolor{figpaper} {RGB}{250,248,243}
\definecolor{slate}    {RGB}{ 38, 70, 83}
\definecolor{slatefill}{RGB}{213,228,231}
\definecolor{ochre}    {RGB}{174,131, 44}
\definecolor{ochrefill}{RGB}{245,234,206}
\definecolor{plum}     {RGB}{106, 61, 99}
\definecolor{plumfill} {RGB}{233,217,230}
\definecolor{rust}     {RGB}{172, 84, 56}
\definecolor{rustfill} {RGB}{247,224,213}
\definecolor{moss}     {RGB}{ 55, 88, 62}
\definecolor{mossfill} {RGB}{218,231,217}

\numberwithin{equation}{section}

\theoremstyle{plain}
\newtheorem{theorem}{Theorem}[section]
\newtheorem{lemma}[theorem]{Lemma}
\newtheorem{proposition}[theorem]{Proposition}
\theoremstyle{definition}
\newtheorem{assumption}[theorem]{Assumption}

\theoremstyle{remark}
\newtheorem{remark}[theorem]{Remark}

\ExplSyntaxOn
\cs_gset:Npn \__first_footerline:
  {
    \group_begin:
    \small
    \sffamily
    \__short_authors:
    \group_end:
  }
\ExplSyntaxOff

\makeatletter
\def\ps@first{%
  \let\@oddhead\@empty
  \let\@evenhead\@empty
  \def\@oddfoot{\hfill\thepage\hfill}%
  \let\@evenfoot\@oddfoot
}
\def\ps@pprintTitle{%
  \let\@oddhead\@empty
  \let\@evenhead\@empty
  \def\@oddfoot{\hfill\thepage\hfill}%
  \let\@evenfoot\@oddfoot
}
\makeatother
\begin{document}
\let\WriteBookmarks\relax       

% Disable bookmark writing
\let\WriteBookmarks\relax

% Improve float placement
\renewcommand{\floatpagefraction}{0.8}
\renewcommand{\textfraction}{0.1}

% relax LaTeX's conservative float rules
\renewcommand{\topfraction}{0.9}      
\renewcommand{\bottomfraction}{0.8}
\renewcommand{\textfraction}{0.07}    
\renewcommand{\floatpagefraction}{0.75}  
\setcounter{topnumber}{3}
\setcounter{totalnumber}{5}

% Running title and author names
\shorttitle{Residual-based adaptive refinement in VC-PINNs}
\shortauthors{S. Giri, M. Malik and R. Yadav}

% Title
\title{\LARGE{Residual-Based Adaptive Refinement in Hybrid\texorpdfstring{\\}{ } Variational-Collocation PINNs for Singular Elliptic PDEs\texorpdfstring{\\}{ } with Mixed Dirichlet-Neumann Boundary Conditions}}

% Authors
% First Author
\author[1]
{Somnath Giri}
[orcid=0009-0004-9829-3805]
\ead{girisomnath70@gmail.com}

% Corresponding Author
\author[1]
{Muslim Malik}
[orcid=0000-0003-0055-7581]
\cormark[1]
\ead{muslim@iitmandi.ac.in}

% Third Author
\author[1]
{Rishikesh Yadav}
[orcid=0000-0003-3419-9356]
\ead{rishikesh@iitmandi.ac.in}

% Author Affiliation
\affiliation[1]
{
    organization={School of Mathematical \& Statistical Sciences, Indian Institute of Technology Mandi},
    city={Mandi},
    postcode={175005},
    state={Himachal Pradesh},
    country={India}
}

% Corresponding Author Footnote
\cortext[1]{Corresponding author}

% Abstract
\begin{abstract}
Solutions of second-order elliptic PDE problems on nonconvex or cracked domains with mixed Dirichlet-Neumann conditions typically lie outside $H^2(\Omega)$. Near re-entrant corners, crack tips and points where the boundary condition changes type, the gradient grows like $r^{\lambda-1}$ with $\lambda<1$. We study residual-based adaptive refinement in hybrid variational-collocation PINNs. A multiscale Fourier-feature network is trained on a composite loss. The loss combines a variational residual tested against Legendre bubble functions, a strong-form collocation term and penalties for the Dirichlet and Neumann data. The variational term needs only first derivatives of the network, whereas the collocation term requires second derivatives and provides a residual that can be evaluated at any point. After each training stage, candidate points selected by this residual are added to the collocation set. The method is tested on five two dimensional benchmarks, namely a Mode~III crack, a cross-shaped domain with four re-entrant corners with unequal amplitude, a reaction-diffusion problem with four localized peaks, a monotone semilinear problem with a boundary layer and the Motz problem. Compared with uniform point insertion under the same collocation and iteration budget, residual-based refinement reduced the relative $L^2$ error by up to a factor of $32$, the relative $L^\infty$ error by up to a factor of $38$ and the relative $H^1$ semi-norm error by up to a factor of $15$.
\end{abstract}

\hfuzz=120pt
%% Keywords %%
\maketitle
\hfuzz=0.1pt

%% Introductions %%
\section{Introduction}\label{Sec1_intro}
Many physical phenomena, including steady heat transfer, mass transport, electrostatic fields, linear elasticity and incompressible potential flows, are governed by second-order elliptic boundary value problems~\cite{Ciarlet2002,BrennerScott2008}. In many applications, the primary quantities of interest are gradient, such as heat fluxes, electric fields and stress intensity factors. For such quantities, the accuracy of a numerical method in approximating \(\nabla u\) is often as important as its accuracy in approximating \(u\)~\cite{LeVeque2007}. Mixed Dirichlet-Neumann boundary conditions frequently arise in elliptic problems, with the Neumann component presenting the greater difficulty. As it prescribes normal derivatives, its accurate enforcement requires resolving gradients that may be singular near crack faces and re-entrant corners. The change of boundary condition type at a Dirichlet-Neumann junction introduces a singularity even in the absence of geometric singularities~\cite{Grisvard2011}. The classical Motz problem illustrates this effect on a straight boundary with smooth data~\cite{Motz1947}.

Near a corner of interior angle \(\omega\), Kondrat'ev-Grisvard theory~\cite{Kondratev1967,Grisvard2011} provides the decomposition $u=u_{\mathrm{reg}}+\sum_j c_j \chi(r)\, r^{\lambda_j}\phi_j(\theta),$ where \(u_{\mathrm{reg}}\in H^2(\Omega)\) and the singular exponents \(\lambda_j\) depend on both the interior angle and the two adjacent boundary conditions and equal $\frac{j\pi}{\omega}$ when both conditions are of the same type. With \(\lambda=\lambda_1\) denoting the leading singular exponent, \(u\in H^{s}(\Omega)\) for every \(s<1+\lambda\) and \(|\nabla u|\sim r^{\lambda-1}\) is unbounded whenever \(\lambda<1\), as at re-entrant corners (\(\omega>\pi\)) and at Dirichlet-Neumann junctions. This loss of regularity is intrinsic to the underlying problem and cannot be eliminated by the choice of approximation space. A uniform discretization fine enough to resolve the singularity over resolves the rest of the domain. Adaptive and \(hp\)-adaptive finite element methods instead use a posteriori error estimators~\cite{Verfurth2013} to concentrate resolution near corners and crack tips. They recover optimal~\cite{BinevDahmenDeVore2004,Cascon2008} or even exponential~\cite{SzaboBabuska2021,MitchellMcClain2011} convergence rates, at the cost of mesh generation and local refinement.

PINNs enforce the governing PDE by minimizing the strong-form residual together with boundary condition penalties, with derivatives computed by automatic differentiation~\cite{Raissi2019,Lagaris1998,Baydin2018}. Their meshfree formulation enables the direct treatment of complex geometries~\cite{Karniadakis2021,NoorMalik2025}. However, three factors limit the accuracy of PINNs for the class of problems considered here. First, the strong formulation requires the evaluation of second-order derivatives. Second, the optimization process exhibits a bias toward smooth, low-frequency solution components, leading to poor resolution of singularities and sharp layers~\cite{Rahaman2019,WangWangPerdikaris2021,Krishnapriyan2021}. Third, Neumann boundary conditions are typically enforced either through penalty terms, whose weights must be balanced against the residual term~\cite{Wang2021} or through specially designed trial spaces~\cite{McFallMahan2009,SukumarSrivastava2022,BeidokhtiMalek2009}. For corner singularities and mixed boundary conditions, singularity enriched PINNs~\cite{HuJinZhou2024} add the leading singular functions, with trainable coefficients to the network.

In variational PINNs (VPINNs)~\cite{Kharazmi2019}, the strong-form residual is replaced by a variational residual of Petrov-Galerkin type, with the neural network representing the trial solution and local polynomial basis functions defining the test space~\cite{Kharazmi2021,GriffithsLorenz1978}. Integration by parts transfers derivatives to the test functions, making only first-order derivatives of the network necessary. The mathematical theory of VPINNs discretizations is now well developed. It covers a priori energy norm estimates under inf-sup stability conditions, the effects of quadrature accuracy and test-space enrichment~\cite{BerroneCanutoPintore2022a} and reliable and efficient a posteriori error estimation~\cite{BerroneCanutoPintore2022b}. Further advances include basis independent residual minimization through the Riesz representation of the discrete dual norm~\cite{Rojas2024} and its spectral realization in the Deep Fourier Residual framework~\cite{Taylor2023,Taylor2024}.

Adaptive strategies for VPINNs and related weak-form neural solvers typically act on the test space or on an underlying mesh. For example, the meshfree VPINN (MF-VPINN) generates additional test functions in regions identified by a posteriori error indicators~\cite{BerronePintore2024}, the robust VPINN (RVPINN) adaptively enriches the test space using information derived from the mismatch between the discrete and exact Riesz representatives~\cite{Udomworarat2026} and finite element-interpolated neural networks with adaptive refinement modify the underlying mesh used to represent the network approximation~\cite{Badia2025}. However, the variational residual does not provide a pointwise measure of error and therefore cannot be evaluated or ranked at arbitrary locations \(\bm{x}\in\Omega\). Consequently, pure VPINNs formulations admit no point level adaptivity and the residual-based sampling strategies developed for strong-form PINNs~\cite{Lu2021,Wu2023,MaoMeng2023,Tang2023,Gao2023} cannot be applied directly.

A hybrid loss restores this capability. Specifically, augmenting the variational residual with the strong-form collocation term

$$
\mathcal{L}_{f}
=
\frac{1}{N_f}
\sum_{k=1}^{N_f}
\bigl|R(\hat{u})(\bm{x}_{k})\bigr|^{2},
$$

where \(R(\hat{u})\) is the strong-form residual~\eqref{eq:strongresidualfunction} of the network approximation \(\hat{u}\) and \(\bm{x}_{1},\dots,\bm{x}_{N_f}\) are the collocation points, provides a pointwise residual that can be evaluated throughout \(\Omega\). This residual may then be computed on a dense candidate pool and ranked to guide residual-based adaptive refinement (RAR). In this way, adaptivity acts directly on the collocation support, whereas variational adaptivity acts through the test space, the two mechanisms are complementary and address distinct aspects of the approximation. At the same time, the variational component depends only on first-order derivatives of \(\hat{u}\), thereby avoiding the \(H^{2}(\Omega)\) regularity requirements that are violated near crack tips, re-entrant corners and other singular features. The proposed method is neither a pure weak-form method nor is it entirely free of second-order derivatives, as both the collocation term \(\mathcal{L}_{f}\) and the RAR indicator require the evaluation of \(\Delta\hat{u}\). Moreover, the Neumann boundary condition is not imposed naturally, the Legendre bubble functions that serve as test functions vanish on element boundaries, so the boundary functional drops out of the discrete variational residual. The prescribed flux is therefore enforced through an explicit penalty term, natural imposition would require a modified test space on elements adjacent to the Neumann boundary~\cite{INIVPINN2026}. In addition, no adaptive \(h\)- or \(p\)-refinement is performed, both the mesh partition and the local polynomial degree remain fixed throughout training. The underlying static discretization is that of \(hp\)-VPINNs~\cite{Kharazmi2021} and we call the resulting framework a hybrid variational-collocation PINNs (VC-PINNs), or VC-PINNs-RAR when it is combined with RAR. Collocation also enters collocation-based robust VPINNs (CRVPINNs)~\cite{Los2025}, where the robust variational loss is discretized by a point collocation scheme with a sparse Gram matrix. Here, by contrast, the collocation points carry a separate strong-form term beside a quadrature-based variational residual.

The hybrid formulation is chosen as the five benchmark problems share two difficulties that neither component of the loss addresses on its own. First, in the Mode~III crack, the cross-shaped domain, the semilinear L-shaped problem and the Motz problem, the solution has a crack-tip, re-entrant-corner or Dirichlet-Neumann junction singularity, so that \(u\notin H^{2}(\Omega)\) and \(|\nabla u|\sim r^{\lambda-1}\). For such solutions the variational residual, which involves only \(\nabla\hat{u}\), remains well defined, whereas a purely strong-form loss controls second derivatives that are not square integrable near the singular point. Second, the error in all five problems is localized, near the singular points, along the boundary layer of the semilinear problem and around the steep peaks of the reaction-diffusion problem. A fixed set of collocation points resolves such regions only if it is dense everywhere. The pointwise strong-form residual of the hybrid loss, by contrast, can be ranked on a candidate pool, so that RAR places new points where the error concentrates. In the cross-shaped domain this ranking also allows the inserted points to be shared among four singular corners rather than one. The two components therefore play distinct roles. The variational residual provides a global constraint that remains meaningful under reduced regularity and the collocation term provides the pointwise indicator that drives refinement. The five benchmark problems are classical test problems for these features~\cite{Grisvard2011,Motz1947,Mitchell2013}.

The contributions of this work are as follows. First, we observe that a strong-form collocation term kept beside the variational residual is what allows point insertion to act on training, its residual can be evaluated and ranked at any candidate point, whereas each entry of the variational residual belongs to a test function, not to a point. Second, we adapt RAR~\cite{Lu2021} to this hybrid loss for elliptic problems with mixed Dirichlet-Neumann conditions. Where needed, candidates within a guard radius of a singular point are excluded, as there the pointwise residual is not a reliable refinement indicator. Third, we compare every RAR run with a uniform insertion run on five two dimensional benchmark problems. The two runs of each comparison, the uniform arm and the RAR arm, share the architecture, quadrature rule, collocation schedule and optimizer schedule, point placement is the only difference. The relative $L^{2}$, $L^{\infty}$ and $H^{1}$ semi-norm errors are reported problem by problem together with the training budgets. 

The paper is organized as follows. Section~\ref{Sec2_preliminaries} states the model problem, its weak formulation and the regularity results used later and it separates the element indexed variational residual from the pointwise strong-form residual, a distinction on which the rest of the paper builds. Section~\ref{Sec3_hybrid-VC-PINNs} develops the hybrid discretization, the composite loss, the RAR procedure and the training algorithm, together with the budget under which the two arms are compared. Section~\ref{Sec4_airchitecture} assembles these into the architecture of the framework. Five numerical examples are presented in section~\ref{Sec5_neumerical_examples}, followed by concluding remarks in section~\ref{Sec6_conclusion}.

%% Preliminaries
\section{Preliminaries}\label{Sec2_preliminaries}
This section fixes the model problem and its weak form, records the regularity imposed by the geometry and the boundary conditions and distinguishes the two residuals on which the method of section~\ref{Sec3_hybrid-VC-PINNs} is built.

\subsection{Model problem and weak formulation}\label{Sec2.1_model_weak}
Let $\Omega\subset\mathbb{R}^{2}$ be a bounded, connected open set whose boundary is a finite union of closed straight segments, with vertex set $\mathcal{V}$ and interior angles $\omega_{v}\in(0,2\pi]$. The value $\omega_{v}=2\pi$ occurs only at the tip of a slit, the one case in which $\Omega$ is not Lipschitz, in the terminology of~\cite{Grisvard2011}, $\Omega$ is a polygon with cracks. The two faces of a slit are treated as distinct boundary portions with opposite outward normals, so that a function in $H^{1}(\Omega)$ has two independent traces there.

The boundary is split into disjoint, relatively open sets $\Gamma_{D}$ and $\Gamma_{N}$ with $\partial\Omega=\overline{\Gamma_{D}\cup\Gamma_{N}}$ and either (D) $|\Gamma_{D}|>0$ or (N) $\Gamma_{D}=\emptyset$, where $|\cdot|$ denotes one-dimensional measure. Case~(N) occurs only in section~\ref{reaction_diffusion_PDE4}. With $(\cdot,\cdot)_{\Omega}$ the $L^{2}(\Omega)$ inner product, let
\begin{equation}
V=\{v\in H^{1}(\Omega):v|_{\Gamma_{D}}=0\},\qquad H^{1}_{g}=\{w\in H^{1}(\Omega):w|_{\Gamma_{D}}=g\},
\label{eq:spaces}
\end{equation}
so that $V=H^{1}_{g}=H^{1}(\Omega)$ in case~(N) and $H^{1}_{g}=u_{g}+V$ for any $u_{g}\in H^{1}(\Omega)$ with trace $g$ on $\Gamma_{D}$. The dual space of $V$ is denoted by $V'$ and $\langle\cdot,\cdot\rangle$ denotes the duality pairing between $V'$ and $V$.

\begin{assumption}[Data]\label{ass1_data}
$\bm{A}=\kappa\bm{I}$ with constant $\kappa>0$, $\kappa=1$ for the crack, cross and Motz problems and $\kappa=\varepsilon^{2}$ for the reaction-diffusion and semilinear problems, with $\varepsilon=0.01$ and $\varepsilon=0.05$, respectively. The data satisfy $f\in L^{2}(\Omega)$ and $h\in L^{2}(\Gamma_{N})$ and $g$ is the trace on $\Gamma_{D}$ of a function in $H^{1}(\Omega)$. The reaction term $\mathcal{N}:\mathbb{R}\to\mathbb{R}$ is continuous with polynomial growth, so $\mathcal{N}(w)\in L^{2}(\Omega)$ for every $w\in H^{1}(\Omega)$ by the embedding $H^{1}(\Omega)\hookrightarrow L^{q}(\Omega)$, $q<\infty$, which holds since $\Omega$ coincides, up to a null set, with a finite union of Lipschitz rectangles. Three cases occur: $\mathcal{N}\equiv0$ for the crack, cross and Motz problems, $\mathcal{N}(t)=t^{3}-t$ for the reaction-diffusion problem and $\mathcal{N}(t)=t^{3}$ for the semilinear problem.
\end{assumption}

Square integrability of $h$, rather than boundedness, is required as, in the cross problem, the prescribed flux behaves like $r^{-1/3}$ at the re-entrant corners. With $\bm{n}$ the outward unit normal, the model problem is
\begin{align}
-\nabla\!\cdot\!(\bm{A}\nabla u)+\mathcal{N}(u)&=f
  &&\text{in }\Omega,\label{eq:pde}\\
u&=g &&\text{on }\Gamma_{D},\label{eq:dir}\\
(\bm{A}\nabla u)\cdot\bm{n}&=h &&\text{on }\Gamma_{N}.\label{eq:neu}
\end{align}
Equation~\eqref{eq:pde} holds in the sense of distributions and pointwise only where $u\in H^{2}_{\mathrm{loc}}$, this fails at the singular points of section~\ref{Sce2.2_regularity}. Testing \eqref{eq:pde} with $v\in V$ and applying Green's formula together with \eqref{eq:neu} gives the weak problem: find $u\in H^{1}_{g}$ such that
\begin{equation}
a(u,v)+\bigl(\mathcal{N}(u),v\bigr)_{\Omega}=\ell(v)
\qquad\forall\,v\in V,
\label{eq:weak}
\end{equation}
where
\begin{equation}
a(u,v)=\int_{\Omega}\bm{A}\nabla u\cdot\nabla v\,d\bm{x},
\qquad
\ell(v)=\int_{\Omega}f\,v\,d\bm{x}+\int_{\Gamma_{N}}h\,v\,dS.
\label{eq:bilinear}
\end{equation}
Problem~\eqref{eq:weak} involves only first derivatives of $u$. The Dirichlet condition is essential, the Neumann condition is natural and enters through $\ell$. Section~\ref{Sec3.2_testspace_quadratures} shows that the discrete test space annihilates this term and section~\ref{sec:penalties} restores it by a penalty.

\begin{proposition}[Well-posedness]\label{prop1_wellposed}
Let Assumption~\ref{ass1_data} hold.
\begin{enumerate}
\item[(i)] If $\mathcal{N}\equiv0$ and case~(D) holds, $a$ is bounded and coercive on $V$ and \eqref{eq:weak} has a unique solution by the Lax-Milgram theorem~\cite{BrennerScott2008}.
\item[(ii)] If $\mathcal{N}$ is nondecreasing and case~(D) holds, the operator $\mathcal{G}:V\to V'$, $\langle\mathcal{G}(z),v\rangle =a(u_{g}+z,v)+(\mathcal{N}(u_{g}+z),v)_{\Omega}-\ell(v)$,
is bounded, hemicontinuous, coercive and strictly monotone, with
\begin{equation}
\langle\mathcal{G}(z_{1})-\mathcal{G}(z_{2}),z_{1}-z_{2}\rangle
\ge\kappa\,\|\nabla(z_{1}-z_{2})\|^{2}_{L^{2}(\Omega)} .
\label{eq:monotone}
\end{equation}
By the Browder-Minty theorem~\cite{Zeidler2013}, \eqref{eq:weak} has a unique solution, independent of the lift $u_{g}$.
\item[(iii)] If $\mathcal{N}(t)=t^{3}-t$ and case~(N) holds, $\mathcal{N}$ is decreasing on $(-\frac{1}{\sqrt{3}},\frac{1}{\sqrt{3}})$ and for $f\equiv h\equiv0$ the constants $u\equiv0$ and $u\equiv\pm1$ all solve \eqref{eq:weak}. Uniqueness fails and the boundary data do not select a branch.
\end{enumerate}
\end{proposition}

Coercivity in (i) and (ii) rests on the Poincar\'e-Friedrichs inequality on $V$. On the slit domain it follows from $|\Gamma_{D}|>0$, the connectedness of $\Omega$ and the compactness of $H^{1}(\Omega)\hookrightarrow L^{2}(\Omega)$. The compactness holds as $\Omega$ coincides, up to a null set, with a finite union of Lipschitz rectangles.

\begin{remark}[Branch selection]\label{remarks1_anchors}
In view of proposition~\ref{prop1_wellposed}(iii), the reaction-diffusion problem is supplemented by twenty interior values of the exact solution, placed at the four peak centers and at four neighbors of each. They enter the loss~\eqref{eq:loss} as an additional quadratic term. They select the branch and also supply the locations of the features that the adaptive scheme resolve, hence no conclusion about feature discovery is drawn from that example. The semilinear problem has a unique solution by proposition~\ref{prop1_wellposed}(ii) and nothing beyond $f$, $g$, $h$ and $\Omega$ is supplied to the solver.
\end{remark}

\subsection{Regularity}\label{Sce2.2_regularity}
Call $p\in\partial\Omega\setminus\mathcal{V}$ a Dirichlet-Neumann junction if $p\in\overline{\Gamma_{D}}\cap\overline{\Gamma_{N}}$, let $\mathcal{J}$ be the finite set of such points and define the singular set $\mathcal{S}=\mathcal{V}\cup\mathcal{J}$. Its subset $\mathcal{S}^{*}=\{p\in\mathcal{S}:\lambda_{1}(p)<1\}$, with $\lambda_{1}(p)$ the leading exponent at $p$ given by~\eqref{eq:eigenpairs}, collects the points at which the expansion~\eqref{eq:kondratev} below can contain singular terms.
On every compact subset of $\overline{\Omega}\setminus\mathcal{S}$, $u\in H^{2}(\Omega)$ and $u$ is $C^{\infty}$ there for smooth data \citep{Grisvard2011,BrennerScott2008}. Near $p\in\mathcal{S}$ with interior angle $\omega$ and local polar coordinates $(r,\theta)$,
\begin{equation}
u=u_{\mathrm{reg}}+\sum_{0<\lambda_{j}<1}c_{j}\,\chi(r)\,
r^{\lambda_{j}}\phi_{j}(\theta),\qquad u_{\mathrm{reg}}\in H^{2}(\Omega),
\label{eq:kondratev}
\end{equation}
where $\chi$ is a smooth cut off and the coefficients $c_{j}$ are not known a priori. The pairs $(\lambda_{j},\phi_{j})$ solve $\phi''+\lambda^{2}\phi=0$ on $(0,\omega)$ under the boundary conditions of the two rays meeting at $p$ \citep{Grisvard2011},
\begin{equation}
\lambda_{j}=\frac{j\pi}{\omega}\ \ \text{(D,D), (N,N)},\qquad
\lambda_{j}=\frac{(2j-1)\pi}{2\omega}\ \ \text{(D,N), (N,D)},
\qquad j\ge1,
\label{eq:eigenpairs}
\end{equation}
with $\phi_{j}=\sin(\lambda_{j}\theta)$ if the ray $\theta=0$ is Dirichlet and $\phi_{j}=\cos(\lambda_{j}\theta)$ if it is Neumann. A change of boundary condition type shifts the spectrum by half an index, so $\lambda_{1}=\frac{1}{2}$ already at $\omega=\pi$; this is the mechanism of the Motz problem~\cite{Motz1947}. Terms with $\lambda_{j}\in\mathbb{N}$ whose angular part makes $r^{\lambda_{j}}\phi_{j}$ a harmonic polynomial are smooth and belong to $u_{\mathrm{reg}}$. For the two nonlinear problems, $-\kappa\Delta u=f-\mathcal{N}(u)\in L^{2}(\Omega)$, so \eqref{eq:kondratev} holds with the same $\lambda_{j}$ and different $c_{j}$.

Let $\lambda=\min\limits_{p\in\mathcal{S}}\lambda_{1}(p)$ and suppose $\lambda<1$.
Since, for any $r_{0}>0$, $\int_{0}^{r_{0}}r^{2(\lambda-s)}\,r\,dr<\infty$ exactly when
$s<1+\lambda$,
\begin{equation}
u\in H^{s}(\Omega)\ \ \forall\,s<1+\lambda,
\qquad\text{if } c_{1}\neq0 \text{ and}\ \
u\notin H^{1+\lambda}(\Omega),\ \ |\nabla u|\sim r^{\lambda-1}
\ \text{as } r\to0 .
\label{eq:reg}
\end{equation}
In that case $\nabla u\in L^{2}(\Omega)$ is unbounded and $u\notin H^{2}(\Omega)$.
Table~\ref{tab:exponents} lists the leading exponents of the five benchmark problems. For the four benchmarks with $\lambda<1$, the second derivatives required by~\eqref{eq:strongresidualfunction} are not in $L^{2}(\Omega)$. In the reaction-diffusion problem every vertex is a right angle with Neumann data on both rays, so $\lambda_{j}=2j$, no singular term arises and $u\in H^{2}(\Omega)$. Its difficulty lies in resolving steep peaks, not in regularity. The semilinear problem also carries a boundary layer of width $\frac{\varepsilon}{\sqrt{3}}$ along $\Gamma_{D}$, analyzed in section~\ref{semilinear_PDE5}.

\begin{table}[t]
\caption{Leading singular exponents from \eqref{eq:eigenpairs}. In the cross problem the corner singularity enters through the flux data $h$, the exponent is the same for (D,D) and (N,N).}
\label{tab:exponents}
\centering\small
\begin{tabular}{@{}llccc@{}}
\toprule
Problem & Points of $\mathcal{S}$ that govern $\lambda$ & Rays & $\omega$ & $\lambda_{1}$\\
\midrule
Crack (\S\ref{creak_slit_PDE1}) & crack tip & (N,N) & $2\pi$ & $1/2$\\
Cross (\S\ref{plus_shaped_PDE2}) & four re-entrant corners & (N,N) & $3\pi/2$ & $2/3$\\
Reaction-diffusion (\S\ref{reaction_diffusion_PDE4}) & vertices & (N,N) & $\pi/2$ & $2$\\
Semilinear (\S\ref{semilinear_PDE5}) & re-entrant corner & (D,D) & $3\pi/2$ & $2/3$\\
Motz (\S\ref{motz_PDE6}) & junction at $(0,0)$ & (N,D) & $\pi$ & $1/2$\\
\bottomrule
\end{tabular}
\end{table}

\subsection{Two notions of residual}\label{Sec2.3_residuals}
The \emph{variational residual} of $w\in H^{1}(\Omega)$ is
\begin{equation}
\mathcal{R}(w;v)=a(w,v)+\bigl(\mathcal{N}(w),v\bigr)_{\Omega}-\ell(v),
\qquad v\in V,
\label{eq:varresidualfunctional}
\end{equation}
so $\mathcal{R}(w;\cdot)\in V'$. It involves only first derivatives and is therefore well defined under~\eqref{eq:reg}. A VPINNs~\citep{Kharazmi2019,Kharazmi2021} is a Petrov-Galerkin method that takes its trial functions from a class of neural networks and its test functions from a finite-dimensional space $V_{K}\subset V$. It drives $\mathcal{R}$ to zero on a basis of $V_{K}$. The space $V_{K}$ used here is constructed in section~\ref{Sec3.2_testspace_quadratures}.

The \emph{strong-form residual} is
\begin{equation}
R(w)(\bm{x})=-\nabla\!\cdot\!\bigl(\bm{A}\nabla w(\bm{x})\bigr)
+\mathcal{N}\bigl(w(\bm{x})\bigr)-f(\bm{x}),
\label{eq:strongresidualfunction}
\end{equation}
and it requires second derivatives. As observed in~\cite{Kharazmi2019,Kharazmi2021}, the choice $v=\delta_{\bm{x}_{0}}$, the Dirac measure at a point $\bm{x}_{0}\in\Omega$, formally turns the variational residual into the strong-form residual at $\bm{x}_{0}$. For $w\in H^{1}(\Omega)$ this identification is not available. In two dimensions
$H^{1}(\Omega)\not\hookrightarrow C(\overline{\Omega})$, with $\log|\log r|$ as the standard witness, so point evaluation is unbounded on $H^{1}(\Omega)$, $\delta_{\bm{x}_{0}}\notin V'$ and
$\mathcal{R}(w;\delta_{\bm{x}_{0}})$ is undefined. Consequently, $\mathcal{R}$ can be tested yet cannot be evaluated, ranked or compared at a point. Conversely, $R(u)$ is not defined near any $p\in\mathcal{S}^{*}$, as $u\notin H^{2}(\Omega)$ there. In section~\ref{Sec3.3_collcation_&_composite_loss}, $R$ is applied to the network trial function
$\hat{u}(\cdot;\bm{\theta})$, which is twice differentiable away from $\mathcal{S}$, therefore $R(\hat{u})$ is a function of position that can be evaluated at any candidate point and it is the only loss component with this property.

\begin{remark}[The trial function is not globally smooth]\label{rem:notsmooth}
The input enrichment of section~\ref{Sec3.1_trial_enrich} contains features $r^{\lambda_{1}(p)}$ centered at points $p\in\mathcal{S}^{*}$. Away from these points $\hat{u}$ is $C^{\infty}$. At them,
$|D^{2}r^{\lambda_{1}}|\sim r^{\lambda_{1}-2}$, where $D^{2}$ denotes the Hessian, so $\hat{u}$ inherits the behavior~\eqref{eq:reg} it is built to represent. This has two consequences for~\eqref{eq:strongresidualfunction} evaluated at $\hat{u}$. First, the individual second derivatives grow like $r^{\lambda_{1}-2}$ while their
combination is the small quantity being measured, so near $p$ the computed value is eventually dominated by floating point cancellation. Second, the $r^{\lambda_{1}-2}$ envelope is a genuine property of the solution, so a ranking by $|R(\hat{u})|$ favors proximity to $p$ independently of the local error. Candidates within a guard radius of $\mathcal{S}^{*}$ are therefore excluded from refinement (section~\ref{sec:guard}).
\end{remark}

 % Hybrid Variational-Collocation 
\section{Hybrid variational-collocation discretization with residual-based adaptive refinement}\label{Sec3_hybrid-VC-PINNs}
This section discretizes the residual formulations of section~\ref{Sec2.3_residuals} and introduces the adaptive procedure built on them. The trial function (section~\ref{Sec3.1_trial_enrich}) is a residual $\tanh$ network on a multiscale Fourier-feature embedding. For the singular benchmarks it is preceded by an enrichment that encodes the known behavior at the singular points in $\mathcal{S}^{*}$. The variational residual (section~\ref{Sec3.2_testspace_quadratures}) is tested against Legendre bubble functions on a fixed rectangular partition and evaluated by Gauss-Legendre quadrature. As these functions vanish on every element boundary, including $\partial\Omega$, the Neumann data are imposed by a penalty in the composite loss (section~\ref{Sec3.3_collcation_&_composite_loss}). The loss also contains a collocation term built on the strong-form
residual. This term makes the method hybrid. Unlike the variational residual, the strong-form residual can be evaluated at arbitrary points, so it drives the insertion of collocation points by RAR (section~\ref{Sec3.4_rar_collocations}). Section~\ref{Sec3.5_algorithm_&_scope} states the complete algorithm, in which only the collocation set is adapted, together with the optimization settings and the training budgets.

\subsection{Trial function and input enrichment}\label{Sec3.1_trial_enrich}
As in PINNs and VPINNs~\cite{Raissi2019,Kharazmi2021}, the trial functions are neural networks. Here the trial function is the composition
\begin{equation}
\hat{u}(\cdot\,;\bm{\theta})=c_{u}\,\mathcal{F}_{\bm{\theta}}\circ\gamma\circ\Psi,
\qquad
\Psi:\overline{\Omega}\to\mathbb{R}^{n_{\mathrm{in}}},\quad
\gamma:\mathbb{R}^{n_{\mathrm{in}}}\to\mathbb{R}^{2Sn_{F}},\quad
\mathcal{F}_{\bm{\theta}}:\mathbb{R}^{2Sn_{F}}\to\mathbb{R},
\label{eq:enrichmap}
\end{equation}
of an enrichment map $\Psi$, a multiscale Fourier-feature embedding $\gamma$~\cite{Tancik2020,WangWangPerdikaris2021} and a residual $\tanh$ network $\mathcal{F}_{\bm{\theta}}$. The constant $c_{u}$ fixes the output scale. It equals $500$ for the Motz problem, where the network represents $u/500$ and $1$ otherwise.

\textbf{Enrichment:} The maps $\Psi$ are listed in table~\ref{tab:enrich}. Without enrichment,
$\Psi=\mathrm{id}$ and $n_{\mathrm{in}}=2$. Every fractional power is taken of a floored radius,
\begin{equation}
r_{\tau}=\bigl(\max\{r^{2},\tau\}\bigr)^{1/2},
\qquad \tau\le10^{-28},
\label{eq:floor}
\end{equation}
which keeps $\Psi$ finite and differentiable at the singular points themselves. This floor only safeguards the evaluation of $\Psi$. It is unrelated to and many orders of magnitude smaller than, the guard radius $r_{\mathrm{g}}$ of~\eqref{eq:guard}.

\begin{table}[t]
\caption{The enrichment map $\Psi$ of~\eqref{eq:enrichmap}. Here $r_{i}$ is the distance to the $i$-th re-entrant corner. For the crack, $(r,\theta)$ are polar coordinates centered at the tip, with $\theta\in(0,2\pi)$ and the branch cut on the slit. For the Motz problem, $(r,\theta)$ are centered at the Dirichlet-Neumann junction, with $\theta\in[0,\pi]$. In the semilinear row, $r$ is the distance to the re-entrant corner at the origin. Every radius is evaluated through~\eqref{eq:floor}.}
\label{tab:enrich}
\centering\small
\begin{tabular}{@{}lll@{}}
\toprule
Problem & $\Psi(\bm{x})$ & $n_{\mathrm{in}}$\\
\midrule
Crack (\S\ref{creak_slit_PDE1})
 & $\bigl(x,\,y,\,r^{1/2}\cos\tfrac{\theta}{2},\,r^{1/2}\sin\tfrac{\theta}{2}\bigr)$ & 4\\[2pt]
Cross (\S\ref{plus_shaped_PDE2})
 & $\bigl(x,\,y,\,r_{1}^{2/3},\,r_{2}^{2/3},\,r_{3}^{2/3},\,r_{4}^{2/3}\bigr)$ & 6\\[2pt]
Reaction-diffusion (\S\ref{reaction_diffusion_PDE4}) & $(x,\,y)$ & 2\\[2pt]
Semilinear (\S\ref{semilinear_PDE5}) & $\bigl(x,\,y,\,r^{2/3}\bigr)$ & 3\\[2pt]
Motz (\S\ref{motz_PDE6}) & $\bigl(x,\,y,\,r^{1/2}\cos\tfrac{\theta}{2}\bigr)$ & 3\\
\bottomrule
\end{tabular}
\end{table}
\textbf{Embedding and network:} Fix scales $\sigma_{1},\dots,\sigma_{S}>0$ and matrices
$\bm{B}^{(s)}\in\mathbb{R}^{n_{\mathrm{in}}\times n_{F}}$ whose entries are independent draws from the normal distribution with mean $0$ and variance $\sigma_{s}^{2}$. They are sampled once and then held fixed, they are not trained. With $\sin$ and $\cos$ acting componentwise,
\begin{equation}
\gamma(\bm{z})=\Bigl(
\sin\!\bigl(\bm{B}^{(1)\top}\bm{z}\bigr)^{\!\top},
\cos\!\bigl(\bm{B}^{(1)\top}\bm{z}\bigr)^{\!\top},\dots,
\sin\!\bigl(\bm{B}^{(S)\top}\bm{z}\bigr)^{\!\top},
\cos\!\bigl(\bm{B}^{(S)\top}\bm{z}\bigr)^{\!\top}\Bigr)^{\!\top}.
\label{eq:fourier}
\end{equation}
In every experiment $n_{F}=64$ and $S=3$, giving $\gamma(\bm{z})\in\mathbb{R}^{384}$. The network has $L$ hidden layers of uniform width $n_{H}$. With $\bm{h}^{(0)}=\gamma(\Psi(\bm{x}))$,
\begin{equation}
\bm{h}^{(1)}=\tanh\!\bigl(\bm{W}^{(1)}\bm{h}^{(0)}+\bm{b}^{(1)}\bigr),\qquad
\bm{h}^{(i)}=\bm{h}^{(i-1)}+\tanh\!\bigl(\bm{W}^{(i)}\bm{h}^{(i-1)}+\bm{b}^{(i)}\bigr),
\quad i=2,\dots,L,
\label{eq:mlp}
\end{equation}
\begin{equation}
\mathcal{F}_{\bm{\theta}}(\bm{h}^{(0)})=\bm{W}^{(L+1)}\bm{h}^{(L)}+b^{(L+1)},
\label{eq:output}
\end{equation}
where $\bm{W}^{(1)}\in\mathbb{R}^{n_{H}\times2Sn_{F}}$, $\bm{W}^{(i)}\in\mathbb{R}^{n_{H}\times n_{H}}$ for $2\le i\le L$, $\bm{b}^{(i)}\in\mathbb{R}^{n_{H}}$ for $1\le i\le L$, $\bm{W}^{(L+1)}\in\mathbb{R}^{1\times n_{H}}$ and $b^{(L+1)}\in\mathbb{R}$. The trainable parameters are $\bm{\theta}=\{\bm{W}^{(i)},\bm{b}^{(i)}\}_{i=1}^{L}\cup\{\bm{W}^{(L+1)},b^{(L+1)}\}$. The skip connection acts only between hidden layers of equal width. Weights are initialized by the Xavier normal rule~\cite{Tancik2020} and biases by zero. All computations are in double precision. A single fixed seed governs $\{\bm{B}^{(s)}\}$, the initialization
of $\bm{\theta}$ and every random point set of sections~\ref{Sec3.3_collcation_&_composite_loss}
and~\ref{Sec3.4_rar_collocations}.

Since $\gamma$ and $\mathcal{F}_{\bm{\theta}}$ are $C^{\infty}$, the regularity of $\hat{u}$ is that of $\Psi$. Hence $\hat{u}\in C^{\infty}(\overline{\Omega}\setminus\mathcal{S}^{*})$ and near a singular point $p\in\mathcal{S}^{*}$ the trial function carries the algebraic behavior $r^{\lambda_{1}(p)}$ of~\eqref{eq:reg}. The consequences for the strong-form residual are recorded in remark~\ref{rem:notsmooth}.

\begin{remark}[Status of the enrichment]\label{rem:crackenrich}
The crack pair $(s_{1},s_{2})=(r^{1/2}\cos\tfrac{\theta}{2},\,r^{1/2}\sin\tfrac{\theta}{2})$
serves two purposes, which should be kept apart. First, $s_{2}$ removes a topological obstruction. The exact solution takes different values on the two faces of the slit, whereas any function continuous in $(x,y)$ takes arbitrarily close values at points on opposite faces. As $\theta\to0^{+}$ and $\theta\to2\pi^{-}$, $s_{2}\to0^{+}$ and $s_{2}\to0^{-}$, so the map $\bm{x}\mapsto(s_{1},s_{2})$ is continuous on $\Omega$ and separates the faces. This is the discrete counterpart of the two independent traces of section~\ref{Sec2.1_model_weak}. The function $s_{2}$ is not a mode of the (N,N) wedge at $\omega=2\pi$, whose angular functions are cosines by~\eqref{eq:eigenpairs}, it is supplied for the reason just given. Second, $s_{1}$ is the leading singular mode at the tip and for the exact solution of section~\ref{creak_slit_PDE1} it coincides with $u$. The enrichment is therefore not a neutral change of coordinates rather, it is a singular function enrichment of the trial space, as in the singular function method~\cite{FixGulatiWakoff1973}, extended finite element methods~\cite{Moes1999} and enriched neural solvers for crack problems~\cite{Gu2023,Zhu2025}. It differs from those constructions, including the singularity enriched PINNs of~\cite{HuJinZhou2024}, in one respect: the mode is not added as a separate term with its own coefficient instead, it is passed through~\eqref{eq:fourier} and~\eqref{eq:mlp}, so $\hat{u}$ depends on it nonlinearly.

The corner features $r^{2/3}$ have the same status. The Motz feature supplies only the leading mode of~\eqref{eq:kondratev}, the higher modes must be learned. Section~\ref{Sec5_neumerical_examples} measures the effect of adaptive point placement \emph{given} the enrichment, which is identical in the two compared arms.
\end{remark}

\subsection{Test space, variational residual and quadrature}\label{Sec3.2_testspace_quadratures}
\textbf{Mesh and reference element:} Let $\mathcal{E}=\{\Omega_{e}\}_{e=1}^{N_{e}}$ be a partition of $\Omega$ into $N_{e}$ non-overlapping, axis-aligned open rectangles
$\Omega_{e}=(x^{e}_{a},x^{e}_{b})\times(y^{e}_{a},y^{e}_{b})$ with
$\overline{\Omega}=\bigcup_{e=1}^{N_{e}}\overline{\Omega}_{e}$. Let
$\hat{\Omega}=[-1,1]^{2}$ denote the reference element with coordinates
$(\xi,\eta)$. Each element is the image of $\hat{\Omega}$ under the affine map
\begin{equation}
F_{e}(\xi,\eta)=\bigl(x^{e}_{a}+J^{e}_{x}(1+\xi),\;y^{e}_{a}+J^{e}_{y}(1+\eta)\bigr),
\qquad
J^{e}_{x}=\tfrac{1}{2}(x^{e}_{b}-x^{e}_{a}),\quad
J^{e}_{y}=\tfrac{1}{2}(y^{e}_{b}-y^{e}_{a}),
\label{eq:affinemap}
\end{equation}
so that $F_{e}(\hat{\Omega})=\overline{\Omega}_{e}$. The Jacobian matrix of
$F_{e}$ is $\operatorname{diag}(J^{e}_{x},J^{e}_{y})$ where $J^{e}_{x}$ and
$J^{e}_{y}$ are half the side lengths of $\Omega_{e}$ and $d\bm{x}=J^{e}_{x}J^{e}_{y}\,d\xi\,d\eta$.

The partition is chosen so that no point of $\mathcal{S}^{*}$ lies in an element interior. The re-entrant corner of the semilinear problem and the Motz junction are element vertices, the four re-entrant corners of the cross are vertices of the two side elements and lie inside the vertical edges of the central bar. For the crack, the slit $\Sigma=\{(x,0):0\le x\le1\}$ of section~\ref{creak_slit_PDE1} lies on an element interface. Hence no element straddles the slit and the polar angle $\theta$ in the exact solution is smooth on every element. We use $N_{e}=3$ for the cross and semilinear problems and $N_{e}=4$ for the crack, reaction-diffusion and Motz problems.

\textbf{Test functions:} On $[-1,1]$ we use the Legendre bubble functions of~\citep{Kharazmi2021},
\begin{equation}
\varphi_{i}(\xi)=P_{i+1}(\xi)-P_{i-1}(\xi),\qquad i=1,\dots,K,
\label{eq:bubble}
\end{equation}
where $P_{j}$ is the Legendre polynomial of degree $j$, thus $\varphi_{i}$ is a
polynomial of degree $i+1$. Since $P_{j}(1)=1$, $P_{j}(-1)=(-1)^{j}$ and
$P_{i+1}'-P_{i-1}'=(2i+1)P_{i}$,
\begin{equation}
\varphi_{i}(\pm1)=0,\qquad \varphi_{i}'=(2i+1)P_{i}.
\label{eq:bubblezero}
\end{equation}
The derivatives of the test functions are therefore available in closed form and the basis is never differentiated numerically. On the reference element we define the tensor-product functions
\begin{equation}
\hat{v}_{ij}(\xi,\eta)=\varphi_{i}(\xi)\,\varphi_{j}(\eta),\qquad 1\le i,j\le K,
\label{eq:reftest}
\end{equation}
and transfer them to $\Omega_{e}$, extended by zero outside the element,
\begin{equation}
v^{e}_{ij}(\bm{x})=
\begin{cases}
\hat{v}_{ij}\bigl(F_{e}^{-1}(\bm{x})\bigr), & \bm{x}\in\Omega_{e},\\[2pt]
0, & \bm{x}\in\Omega\setminus\Omega_{e}.
\end{cases}
\label{eq:localtest}
\end{equation}
The test space is spanned by all such functions,
\begin{equation}
V_{K}=\operatorname{span}\bigl\{v^{e}_{ij}\;:\;1\le e\le N_{e},\ 1\le i,j\le K\bigr\}.
\label{eq:testspace}
\end{equation}
Since distinct elements have disjoint interiors, the functions $v^{e}_{ij}$ are
linearly independent and $\dim V_{K}=N_{e}K^{2}$. Equivalently, $V_{K}$
consists of the functions whose restriction to each $\Omega_{e}$, pulled back
by $F_{e}$, is a polynomial of degree at most $K+1$ in each variable vanishing
on $\partial\hat{\Omega}$. In every experiment $K=8$, which gives $K^{2}=64$
test functions per element and $N_{e}K^{2}\in\{192,256\}$ in total.

\begin{lemma}[Admissibility]\label{lem:conforming}
$V_{K}\subset H^{1}_{0}(\Omega)\subset V$.
\end{lemma}

\begin{proof}
By~\eqref{eq:bubblezero}, $\hat{v}_{ij}$ vanishes on $\partial\hat{\Omega}$, so
$v^{e}_{ij}$ vanishes on $\partial\Omega_{e}$. Its zero extension is therefore
continuous and piecewise polynomial, hence lies in $H^{1}(\Omega)$ with zero
trace on $\partial\Omega$. The inclusion $H^{1}_{0}(\Omega)\subset V$ holds
by~\eqref{eq:spaces} in both case~(D) and case~(N).
\end{proof}

\textbf{Variational residual:} By Lemma~\ref{lem:conforming}, the variational
residual~\eqref{eq:varresidualfunctional} can be tested with $V_{K}$ directly,
without interface or boundary terms and requires only
$\hat{u}\in H^{1}(\Omega)$.
For $1\le e\le N_{e}$
and $1\le i,j\le K$,
\begin{equation}
\mathcal{R}^{e}_{ij}(\bm{\theta})=\mathcal{R}\bigl(\hat{u};v^{e}_{ij}\bigr)
=\int_{\Omega_{e}}\Bigl(\bm{A}\nabla\hat{u}\cdot\nabla v^{e}_{ij}
+\mathcal{N}(\hat{u})\,v^{e}_{ij}-f\,v^{e}_{ij}\Bigr)\,d\bm{x}.
\label{eq:varres}
\end{equation}
The integral is over $\Omega_{e}$ alone, since $v^{e}_{ij}$ vanishes outside
it. The nonlinear term is present only in the reaction-diffusion and semilinear problems.

\begin{remark}[Loss of the natural Neumann term]\label{rem:noneumann}
Every $v\in V_{K}$ vanishes on $\Gamma_{N}$. The boundary functional
of~\eqref{eq:bilinear},
\begin{equation}
B^{e}_{ij} =\int_{\Gamma_{N}}h\,v^{e}_{ij}\,dS,
\label{eq:bdryfunctional}
\end{equation}
is therefore identically zero and the Neumann data do not enter the discrete
variational residual. This is a property of the bubble
basis~\eqref{eq:bubble}, not of the weak form~\eqref{eq:weak}. A basis that is
active on $\Gamma_{N}$, such as the odd sine and cosine modes
of~\citep{INIVPINN2026}, retains the term. We keep the bubble basis and impose
$h$ through the penalty~\eqref{eq:LN}, which is then the only route by which
$h$ enters the loss.
\end{remark}

\textbf{Quadrature:} Let $\{(\xi_{p},w_{p})\}_{p=1}^{Q}$ be the nodes and weights of the $Q$-point Gauss-Legendre rule on $[-1,1]$. Set $w_{pq}=w_{p}w_{q}$ and
$\bm{x}^{e}_{pq}=F_{e}(\xi_{p},\xi_{q})$ and define the set of quadrature points
\begin{equation}
\mathcal{T}_{Q}=\bigl\{\bm{x}^{e}_{pq}:1\le e\le N_{e},\ 1\le p,q\le Q\bigr\},
\qquad|\mathcal{T}_{Q}|=N_{e}Q^{2}.
\label{eq:quadset}
\end{equation}
By the chain rule and~\eqref{eq:affinemap}, for $\bm{x}\in\Omega_{e}$ with
$(\xi,\eta)=F_{e}^{-1}(\bm{x})$,
\begin{equation}
\partial_{x}v^{e}_{ij}(\bm{x})=\frac{\varphi_{i}'(\xi)\,\varphi_{j}(\eta)}{J^{e}_{x}},
\qquad
\partial_{y}v^{e}_{ij}(\bm{x})=\frac{\varphi_{i}(\xi)\,\varphi_{j}'(\eta)}{J^{e}_{y}}.
\end{equation}
With $\bm{A}=\kappa\bm{I}$ (assumption~\ref{ass1_data}), one Jacobian factor cancels in each stiffness term and~\eqref{eq:varres} is approximated, as
in~\citep{Kharazmi2021}, by
\begin{align}
\widetilde{\mathcal{R}}^{e}_{ij}(\bm{\theta})
&=\kappa J^{e}_{y}\sum_{p,q=1}^{Q}w_{pq}\,\partial_{x}\hat{u}(\bm{x}^{e}_{pq})\,
  \varphi_{i}'(\xi_{p})\varphi_{j}(\xi_{q})
 +\kappa J^{e}_{x}\sum_{p,q=1}^{Q}w_{pq}\,\partial_{y}\hat{u}(\bm{x}^{e}_{pq})\,
  \varphi_{i}(\xi_{p})\varphi_{j}'(\xi_{q})\nonumber\\
&\quad+J^{e}_{x}J^{e}_{y}\sum_{p,q=1}^{Q}w_{pq}
  \Bigl[\mathcal{N}\bigl(\hat{u}(\bm{x}^{e}_{pq})\bigr)-f(\bm{x}^{e}_{pq})\Bigr]
  \varphi_{i}(\xi_{p})\varphi_{j}(\xi_{q}).
\label{eq:discres}
\end{align}
Only first derivatives of $\hat{u}$ enter~\eqref{eq:discres}, they are obtained by automatic differentiation at the points of $\mathcal{T}_{Q}$. For the three linear benchmarks, $f\equiv0$ and $\mathcal{N}\equiv0$, so the last line vanishes.

The quadrature rule uses $Q=24$ for the crack, cross and Motz problems, $Q=48$ for the semilinear problem, whose boundary layer of width $\frac{\varepsilon}{\sqrt{3}}$ must be integrated accurately and $Q=80$ for the reaction-diffusion problem. There the elements have unit side length and the peaks have standard deviation $\sigma_{\mathrm{p}}=0.05$, so an interval of length $3\sigma_{\mathrm{p}}$ centered at any of the four peaks contains roughly eight to ten nodes per direction. The rule integrates exactly the polynomial factors $\varphi_{i}\varphi_{j}$, whose degree in each variable is at most $2K+2\le2Q-1$. The full integrands contain $\hat{u}$ and are not integrated exactly, the resulting quadrature error is the variational crime analyzed for VPINNs in~\citep{BerroneCanutoPintore2022a}.

\subsection{Collocation term, constraints and composite loss}\label{Sec3.3_collcation_&_composite_loss}
The family $\{\widetilde{\mathcal{R}}^{e}_{ij}\}$ is indexed by $(e,i,j)$. It is the action of an element of $V_{K}'$, not a function of position. Consistent with section~\ref{Sec2.3_residuals}, there is no map $\bm{x}\mapsto\widetilde{\mathcal{R}}(\bm{x})$ to sample. A pure VPINNs therefore admits refinement of the test space (more elements, higher degree, additional patches) however, it has no notion of inserting a \emph{point} at which the equation is poorly satisfied. This is why adaptive weak-form methods refine test spaces~\citep{BerronePintore2024,Taylor2024,Udomworarat2026} or meshes~\citep{Badia2025} rather than point sets.

We therefore retain the strong-form residual~\eqref{eq:strongresidualfunction} on a set of collocation points (the penalizing points of~\citep{Kharazmi2019}), $\mathcal{T}_{f}=\{\bm{x}^{f}_{k}\}_{k=1}^{N_{f}}\subset\Omega$,
\begin{equation}
\mathcal{L}_{f}(\bm{\theta};\mathcal{T}_{f})
=\frac{1}{N_{f}}\sum_{k=1}^{N_{f}}\bigl|R(\hat{u})(\bm{x}^{f}_{k})\bigr|^{2},
\qquad R(\hat{u})=-\kappa\Delta\hat{u}+\mathcal{N}(\hat{u})-f.
\label{eq:Lf}
\end{equation}
The second derivatives are obtained by nested automatic differentiation, so the method is not second-derivative free. The term plays two roles. It penalizes pointwise violation of~\eqref{eq:pde}, supplementing the element averaged information in~\eqref{eq:discres}. As it is built on a function of $\bm{x}$, it makes point level adaptivity possible, without $\mathcal{L}_{f}$, the indicator of section~\ref{Sec3.4_rar_collocations} would not exist.

It also changes the size of the discrete system. The variational term supplies
$N_{e}K^{2}\in\{192,256\}$ scalar conditions against $8.9\times10^{4}$ to $2.0\times10^{5}$ trainable parameters. With the collocation term the count becomes $N_{e}K^{2}+N_{f}$, which at the final round lies between $4.2\times10^{3}$ (crack and Motz) and $1.3\times10^{4}$ (reaction-diffusion). The system remains underdetermined in every case.

\textbf{Boundary and data constraints:}\label{sec:penalties}
With $\mathcal{T}_{D}=\{\bm{x}^{D}_{i}\}_{i=1}^{N_{D}}\subset\Gamma_{D}$ and
$\mathcal{T}_{N}=\{\bm{x}^{N}_{j}\}_{j=1}^{N_{N}}\subset\Gamma_{N}$, define
\begin{align}
\mathcal{L}_{D}(\bm{\theta};\mathcal{T}_{D})&=\frac{1}{N_{D}}\sum_{i=1}^{N_{D}}
\bigl|\hat{u}(\bm{x}^{D}_{i};\bm{\theta})-g(\bm{x}^{D}_{i})\bigr|^{2},
\label{eq:LD}\\
\mathcal{L}_{N}(\bm{\theta};\mathcal{T}_{N})&=\frac{1}{N_{N}}\sum_{j=1}^{N_{N}}
\bigl|\kappa\,\partial_{\bm{n}}\hat{u}(\bm{x}^{N}_{j};\bm{\theta})-h(\bm{x}^{N}_{j})\bigr|^{2}.
\label{eq:LN}
\end{align}
These enforce~\eqref{eq:dir} and restore~\eqref{eq:neu}, which was lost in remark~\ref{rem:noneumann}. Both sets are drawn once, uniformly along each boundary segment (150-300 points per segment, depending on the problem) and are never resampled. Since~\eqref{eq:LN} penalizes the flux, its effective weight on
$\partial_{\bm{n}}\hat{u}$ is $w_{N}\kappa^{2}$. For the semilinear problem, where $h\equiv0$ and $\kappa=\varepsilon^{2}$, this is small, about
$3\times10^{-4}$.

For the reaction-diffusion problem, $\Gamma_{D}=\emptyset$ and $\mathcal{L}_{D}$ is vacuous. In its place, with $N_{A}=20$ interior anchors $\mathcal{T}_{A}=\{\bm{x}^{A}_{k}\}_{k=1}^{N_{A}}\subset\Omega$,
\begin{equation}
\mathcal{L}_{A}(\bm{\theta};\mathcal{T}_{A})=\frac{1}{N_{A}}\sum_{k=1}^{N_{A}}
\bigl|\hat{u}(\bm{x}^{A}_{k};\bm{\theta})-u(\bm{x}^{A}_{k})\bigr|^{2},
\label{eq:LA}
\end{equation}
where $u$ is the exact solution of section~\ref{reaction_diffusion_PDE4}. The anchors sit at the four peak centers and at the four points one peak width $\sigma_{\mathrm{p}}$ from each center along the coordinate directions. They select the branch (remark~\ref{remarks1_anchors}) and $\mathcal{L}_{A}$ enters~\eqref{eq:loss} with weight $w_{D}$.

\textbf{Composite loss:}\label{sec:loss}
The training objective is
\begin{align}
\mathcal{L}(\bm{\theta})&=w_{v}\mathcal{L}_{v}+w_{D}\mathcal{L}_{D}
+w_{N}\mathcal{L}_{N}+w_{f}\mathcal{L}_{f},
\label{eq:loss}\\
\mathcal{L}_{v}(\bm{\theta};\mathcal{T}_{Q})&=\frac{1}{N_{e}K^{2}}
\sum_{e=1}^{N_{e}}\sum_{i,j=1}^{K}\bigl|\widetilde{\mathcal{R}}^{e}_{ij}(\bm{\theta})\bigr|^{2},
\label{eq:Lv}
\end{align}
with $\mathcal{L}_{D}$ replaced by $\mathcal{L}_{A}$ for the reaction-diffusion problem. Each term is a mean rather than a sum, so the weights keep their meaning as $N_{f}$ grows during refinement. The weights are $(w_{v},w_{D},w_{N},w_{f})=(1,50,50,10)$ for all problems except reaction-diffusion, where they are $(1,100,50,5)$. They were fixed by preliminary experimentation, were not tuned per run and are identical in the two compared arms.

\subsection{Residual-based adaptive refinement}\label{Sec3.4_rar_collocations}\label{sec:adaptive}
\textbf{Indicator:}
After each training stage, a candidate pool $\mathcal{C}\subset\Omega$ with $|\mathcal{C}|\gg N_{f}$ is drawn uniformly at random and the indicator
\begin{equation}
\rho(\bm{x};\bm{\theta})=\bigl|R(\hat{u})(\bm{x})\bigr|
=\bigl|-\kappa\Delta\hat{u}(\bm{x};\bm{\theta})
+\mathcal{N}\bigl(\hat{u}(\bm{x};\bm{\theta})\bigr)-f(\bm{x})\bigr|,
\qquad\bm{x}\in\mathcal{C},
\label{eq:indicator}
\end{equation}
is evaluated on it. The indicator requires no derivatives beyond those in~\eqref{eq:Lf}. It is not an error estimator and no reliability or efficiency property is claimed for it. Its justification is that it can be evaluated at an arbitrary point, which~\eqref{eq:discres} cannot. Near $\mathcal{S}^{*}$, by remark~\ref{rem:notsmooth}, large values of $\rho$ reflect both approximation error and the $r^{\lambda_{1}-2}$ growth of the second derivatives and $\rho$ alone does not separate the two.

\textbf{Guard:}\label{sec:guard} Where this growth is severe, candidates close to $\mathcal{S}^{*}$ are excluded,
\begin{equation}
\mathcal{C}_{\mathrm{adm}}=\bigl\{\bm{x}\in\mathcal{C}:
\operatorname{dist}(\bm{x},\mathcal{S}^{*})\ge r_{\mathrm{g}}\bigr\}.
\label{eq:guard}
\end{equation}
For the crack, $r_{\mathrm{g}}=2\times10^{-3}$. At this radius
$|\partial_{xx}\hat{u}|,|\partial_{yy}\hat{u}|\sim r^{-3/2}\approx1.1\times10^{4}$,
so the round off error in their sum is of order $10^{-12}$. Inside the disk, the computed $\rho$ would rank points by round off error.

For the Motz problem the growth is genuine rather than numerical. On a trained network, the median of $\rho$ rises monotonically from $4.5\times10^{-4}$ at $r=0.5$ to $56$ at $r=10^{-2}$. Ranking against this five decade spread sends the whole insertion budget to the smallest admissible radius. The pool is therefore restricted by $r_{\mathrm{g}}=5\times10^{-2}$, placed where the median indicator (about $0.5$) is comparable to the 90th percentile at larger radii, this compresses the spread to about three decades. All other interior samples of the Motz problem, namely the initial set and the insertions of the uniform arm, satisfy only $r\ge2\times10^{-3}$. The two Motz arms therefore differ in the region from which insertions may be drawn.

For the semilinear problem, $r_{\mathrm{g}}=2\times10^{-3}$ is applied to every interior sample in both arms. No guard is applied to the cross problem, whose second derivatives grow only like $r^{-4/3}$, nor to the reaction-diffusion problem, whose solution is smooth. Where no guard is used, $\mathcal{C}_{\mathrm{adm}}=\mathcal{C}$.

\textbf{Selection:} Let $\bm{x}_{(1)},\dots,\bm{x}_{(n)}$, with $n=|\mathcal{C}_{\mathrm{adm}}|$, enumerate $\mathcal{C}_{\mathrm{adm}}$ in
nonincreasing order of $\rho$. Given the insertion count $m$, set
\begin{equation}
\mathcal{M}=\{\bm{x}_{(1)},\dots,\bm{x}_{(m)}\},\qquad
\mathcal{T}_{f}\leftarrow\mathcal{T}_{f}\cup\mathcal{M}.
\label{eq:select}
\end{equation}
The pool is redrawn in every round and the procedure is repeated for $n_{\mathrm{RAR}}$ rounds.

\textbf{Radius-normalized selection for the crack:} At the crack tip, $\rho$ grows approximately like $r^{-2.3}$ over four decades on a trained network. In a preliminary run, the ranking~\eqref{eq:select} placed $92\%$ of the first round insertions in $r<0.1$, in tight clusters and the relative $L^{2}$ error increased. For this problem the ranking is therefore replaced by three steps.

First (\emph{radial normalization}), with $r(\bm{x})$ the distance from $\bm{x}$ to the crack tip, $\mathcal{C}_{\mathrm{adm}}$ is divided into $n_{b}=12$ radial bands in $r$ with logarithmically spaced edges between $r_{\mathrm{g}}$ and $1.5$. Each indicator value is normalized by the median over its own band,
\begin{equation}
\hat{\rho}(\bm{x})=\rho(\bm{x})\big/\operatorname{med}\{\rho(\bm{y}):
\bm{y}\in\mathcal{C}_{\mathrm{adm}},\ \bm{y}\text{ in the band of }\bm{x}\}.
\label{eq:bandscore}
\end{equation}

Second (\emph{density sampling}), $\min\{n,\max\{8m,m+1000\}\}$ candidates are drawn without replacement
with probability proportional to $\hat{\rho}(\bm{x})/\overline{\hat{\rho}}+1$, where $\overline{\hat{\rho}}$ is the mean of $\hat{\rho}$ over $\mathcal{C}_{\mathrm{adm}}$. This is the residual-based adaptive distribution of~\citep{Wu2023} with exponent $1$ and offset $1$. Third (\emph{minimum-separation filter}), the drawn candidates are visited in decreasing order of $\hat{\rho}$. A candidate is accepted if it lies at distance at least $0.35\max\{r(\bm{x}),r_{\mathrm{g}}\}$ from every point already accepted. If fewer than $m$ points are accepted, the remaining drawn candidates fill the set.

\textbf{Uniform arm:} The uniform arm, against which the RAR arm is compared, replaces the selection step by $m$ points drawn uniformly from the same admissible region and leaves everything else unchanged; the indicator is not evaluated. The only exception to identical admissible regions is the Motz problem, as described under Guard.

\subsection{Algorithm, scope and training budget}\label{Sec3.5_algorithm_&_scope}
\begin{algorithm}[t]
\caption{VC-PINNs with residual-based adaptive refinement (VC-PINNs-RAR)}
\label{alg:main}
\begin{algorithmic}[1]
\Require data $f,g,h$; network \eqref{eq:enrichmap}-\eqref{eq:output};
partition $\mathcal{E}$ and test space $V_{K}$ \eqref{eq:testspace};
quadrature set $\mathcal{T}_{Q}$; point sets
$\mathcal{T}_{D},\mathcal{T}_{N},\mathcal{T}_{f}$ (and $\mathcal{T}_{A}$);
weights $w_{v},w_{D},w_{N},w_{f}$; rounds $n_{\mathrm{RAR}}$; insertion count
$m$; pool size $|\mathcal{C}|$; guard radius $r_{\mathrm{g}}$; iteration
counts $N^{0}_{\mathrm{Adam}},N_{\mathrm{Adam}},N_{\mathrm{LBFGS}}$.
\State Initialize $\bm{\theta}$.
\State Train $N^{0}_{\mathrm{Adam}}$ Adam iterations on \eqref{eq:loss}, then
up to $N_{\mathrm{LBFGS}}$ L-BFGS iterations.
\For{$k=1,\dots,n_{\mathrm{RAR}}$}
  \State Draw $\mathcal{C}\subset\Omega$ and form $\mathcal{C}_{\mathrm{adm}}$ by \eqref{eq:guard}.
  \State Evaluate $\rho$ on $\mathcal{C}_{\mathrm{adm}}$ by \eqref{eq:indicator}.
  \State Select $\mathcal{M}$ by \eqref{eq:select} (crack: by \eqref{eq:bandscore} and the steps following it); set $\mathcal{T}_{f}\leftarrow\mathcal{T}_{f}\cup\mathcal{M}$.
  \State Train $N_{\mathrm{Adam}}$ Adam iterations, then up to $N_{\mathrm{LBFGS}}$ L-BFGS iterations, on the updated \eqref{eq:loss}.
\EndFor
\State \Return $\hat{u}(\cdot;\bm{\theta})$.
\end{algorithmic}
\end{algorithm}

Algorithm~\ref{alg:main} is written in execution order. In each round the pool is scored and $\mathcal{T}_{f}$ enlarged, using the parameters left by the previous stage. The algorithm acts on $\mathcal{T}_{f}$ and hence on $\mathcal{L}_{f}$ and on nothing else. The partition $\mathcal{E}$, the number $K$ of test functions per direction, the quadrature set $\mathcal{T}_{Q}$, the sets $\mathcal{T}_{D},\mathcal{T}_{N},\mathcal{T}_{A}$, the architecture and the
weights are fixed for the whole run. In particular, for given $\bm{\theta}$, $\mathcal{L}_{v}$ is the same function in the last round as in the first. The method performs neither $h$- nor $p$-refinement.

The loss~\eqref{eq:loss} is minimized by Adam~\cite{Kingma2015} followed by L-BFGS~\cite{LiuNocedal1989} with a strong Wolfe line search~\cite{NocedalWright2006} in every stage. Adam is restarted at each stage, with learning rate $10^{-3}$ in the first stage and $10^{-4}$ thereafter, exponential decay by a factor of $0.9$ per $5000$ iterations and gradient norm clipping at $1$. L-BFGS uses a history of $50$, a gradient tolerance of $10^{-14}$ and a change tolerance of $10^{-20}$ and is capped at $N_{\mathrm{LBFGS}}$ iterations per stage. Training ends when the prescribed stages are complete. There is no early stopping and all reported errors are those of the final iterate.

The two arms of every benchmark share the seed, the architecture, the weights, the partition, the test space $V_{K}$ and the quadrature rule. They also share the collocation schedule: $N_{f}^{0}$ initial points, $m$ insertions in each of $n_{\mathrm{RAR}}$ rounds and hence the same $N_{f}$ in every round. Finally, they share the optimizer schedule: $N^{0}_{\mathrm{Adam}}+n_{\mathrm{RAR}} N_{\mathrm{Adam}}$ Adam iterations and $1+n_{\mathrm{RAR}}$ L-BFGS stages. The arms therefore differ only in how the $m$ points of each round are chosen and for the Motz problem, in the admissible region. The problem specific values are given with each benchmark in section~\ref{Sec5_neumerical_examples}.

\section{Architecture of the VC-PINNs-RAR framework}\label{Sec4_airchitecture}
Figure~\ref{fig:vcpinn-rar-architecture} summarizes the proposed VC-PINNs-RAR framework. It assembles the four components of section~\ref{Sec3_hybrid-VC-PINNs} a neural trial function, a discrete variational residual tested against local polynomials on a partitioned domain, a composite loss carrying the mixed boundary data and the refinement loop. In the forward pass, the coordinate $\bm{x}\in\Omega$ is enriched, mapped by the multiscale Fourier-feature embedding and propagated to $\hat{u}(\bm{x};\bm{\theta})$. Automatic differentiation supplies $\nabla\hat{u}$ for the variational residual and $\Delta\hat{u}$ for the strong-form residual. The weighted sum of both residuals and the boundary penalties is minimized by Adam followed by L-BFGS. Between stages, the loop augments the interior collocation set and training continues from the current parameters.

The trial function $\hat{u}=c_{u}\,\mathcal{F}_{\bm{\theta}}\circ\gamma\circ\Psi$ composes a problem specific enrichment $\Psi$, the multiscale Fourier-feature embedding $\gamma$ of~\eqref{eq:fourier} and the residual $\tanh$ network $\mathcal{F}_{\bm{\theta}}$ of~\eqref{eq:mlp}-\eqref{eq:output}, whose hidden layers carry skip connections $\bm{h}^{(i)}=\bm{h}^{(i-1)}+\tanh(\cdot)$. The frequency matrices $\bm{B}^{(s)}$, whose entries are independent draws from the normal distribution with mean $0$ and variance $\sigma_{s}^{2}$, are sampled once at $S=3$ scales and then held fixed, so they belong to the architecture rather than to $\bm{\theta}$. The multiple scales reduce the spectral bias of the network~\cite{Rahaman2019,Tancik2020,WangWangPerdikaris2021} and sharpen its resolution of localized features. All problems are steady, so the input is the spatial coordinate alone. Where the corner behavior is known, $\Psi$ appends the leading singular mode: $r^{2/3}$ on the cross-shaped and L-shaped domains, $r^{1/2}\cos(\theta/2)$ for the Motz problem and the pair $\bigl(r^{1/2}\cos(\theta/2),\,r^{1/2}\sin(\theta/2)\bigr)$ on the slit domain. The pair also carries the jump across the crack faces, which the plain coordinates cannot represent.

Following the Petrov-Galerkin construction, the test functions are drawn from $V_{K}\subset V$ of \eqref{eq:testspace}. The domain is partitioned into nonoverlapping elements, $\overline{\Omega}=\bigcup_{e}\overline{\Omega}_{e}$ and on each element the test functions are the tensor products~\eqref{eq:localtest} of the one-dimensional Legendre bubble functions $\varphi_{1},\dots,\varphi_{K}$ of~\eqref{eq:bubble}. The partition, the number of test functions per direction $K=8$ and the quadrature order are fixed for the entire run. Since $\varphi_{i}(\pm1)=0$, every test function vanishes on $\partial\Omega_{e}$, so no interelement or boundary term survives and the natural boundary functional vanishes identically, $B^{e}_{ij}\equiv0$ on $V_{K}$ (remark~\ref{rem:noneumann}). Testing therefore leaves the three term residual~\eqref{eq:discres}, split as
\begin{equation}\label{eq:res-terms}
  \widetilde{\mathcal{R}}^{e}_{ij}(\bm{\theta})
  = \widetilde{\mathcal{R}}^{e,\mathrm{diff}}_{ij}(\bm{\theta}) + \widetilde{\mathcal{R}}^{e,\mathrm{reac}}_{ij}(\bm{\theta}) - \widetilde{\mathcal{R}}^{e,\mathrm{src}}_{ij},
\end{equation}
with $\widetilde{\mathcal{R}}^{e,\mathrm{diff}}_{ij}$ the diffusion contribution, $\widetilde{\mathcal{R}}^{e,\mathrm{reac}}_{ij}$ that of the nonlinear reaction term (absent when the problem is linear) and $\widetilde{\mathcal{R}}^{e,\mathrm{src}}_{ij}$ the source. Element integrals use the tensor-product Gauss-Legendre rule~\eqref{eq:quadset} with $Q>K$ nodes per direction: $Q=24$ by default, $Q=48$ where a boundary layer must be integrated and $Q=80$ for the localized peaks of the reaction-diffusion problem.

Neither boundary condition is natural in the discrete formulation, so both are imposed by explicit penalties in \eqref{eq:loss}: $\mathcal{L}_{D}$ on $\mathcal{T}_{D}$ and $\mathcal{L}_{N}$ on $\mathcal{T}_{N}$, the latter penalizing the flux $\kappa\,\partial_{\bm{n}}\hat{u}-h$. With $w_{v}$ for the variational residual and $w_{f}$ for the collocation residual, the formulation carries four weights, given in section~\ref{sec:loss}. The parameters are obtained by minimizing \eqref{eq:loss} with Adam followed by L-BFGS with a strong Wolfe line search, one such pair per training stage.

The quadrature set $\mathcal{T}_{Q}$ is fixed and the only object the refinement modifies is the interior collocation set $\mathcal{T}_{f}$. After each stage, the indicator $\rho=|R(\hat{u})|$ of~\eqref{eq:indicator} is evaluated on a candidate pool $\mathcal{C}\subset\Omega$ to identify where the governing equation is least well satisfied. The strong-form residual is used since it can be evaluated at an arbitrary point, whereas the variational residual is indexed by element and test mode moreover, points inserted into $\mathcal{T}_{f}$ act on training through $\mathcal{L}_{f}$. Where the singular growth of the second derivatives is severe, candidates inside the guard radius $r_{\mathrm{g}}$ of the singular point are discarded, as there the indicator grows toward the singular point regardless of the local error (remark~\ref{rem:notsmooth} and section~\ref{sec:guard}). The $m$ highest ranked admissible points, selected by~\eqref{eq:select}, are appended to $\mathcal{T}_{f}$ and training continues for $n_{\mathrm{RAR}}$ rounds (algorithm~\ref{alg:main}). The uniform arm differs only in this selection step: it draws the same number of points uniformly from the same admissible region and reaches the same final $N_{f}$ under the same optimizer schedule. The two exceptions, the admissible region of the Motz problem and the uniform arm of the reaction-diffusion problem, are stated in section~\ref{Sec3.5_algorithm_&_scope}. section~\ref{Sec5_neumerical_examples} compares the two arms.

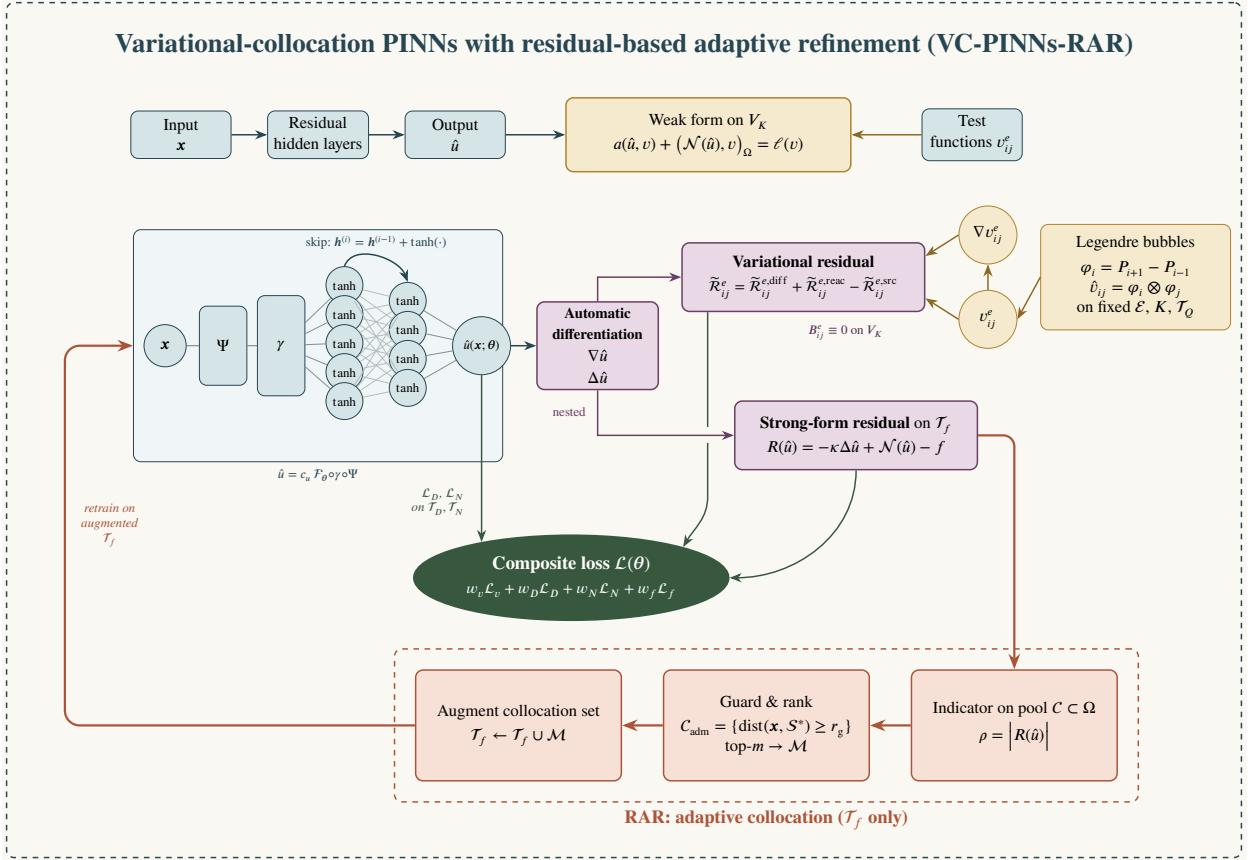
\begin{figure}[!t]
\centering
\resizebox{\textwidth}{!}{%
\begin{tikzpicture}[
    >=Stealth,
    font=\small,
    layerbox/.style   = {draw=slate, rounded corners, fill=slatefill,
                         minimum width=1.9cm, minimum height=0.9cm,
                         align=center},
    varblock/.style   = {draw=ochre, thick, rounded corners, fill=ochrefill,
                         minimum width=5.4cm, minimum height=1.4cm,
                         align=center},
    resblock/.style   = {draw=plum, thick, rounded corners, fill=plumfill,
                         minimum width=4.6cm, minimum height=1.3cm,
                         align=center},
    collblock/.style  = {draw=plum, thick, rounded corners, fill=plumfill,
                         minimum width=4.6cm, minimum height=1.3cm,
                         align=center},
    adbox/.style      = {draw=plum, thick, rounded corners, fill=plumfill,
                         minimum width=2.3cm, minimum height=1.5cm,
                         align=center},
    tfnode/.style     = {draw=ochre, circle, fill=ochrefill,
                         minimum size=11mm, inner sep=0pt},
    tfopt/.style      = {draw=ochre, rounded corners, fill=ochrefill,
                         minimum width=3.6cm, minimum height=2.0cm,
                         align=center},
    psibox/.style     = {draw=slate, rounded corners, fill=slatefill,
                         minimum width=0.9cm, minimum height=1.5cm,
                         align=center},
    fourier/.style    = {draw=slate, rounded corners, fill=slatefill,
                         minimum width=0.9cm, minimum height=1.9cm,
                         align=center},
    neuron/.style     = {circle, draw=slate, fill=slatefill,
                         minimum size=7mm, inner sep=0pt,
                         font=\scriptsize},
    inpneuron/.style  = {circle, draw=slate, fill=slatefill,
                         minimum size=8mm, inner sep=0pt},
    lossstyle/.style  = {ellipse, draw=moss, very thick, fill=moss,
                         text=white, minimum width=5.6cm,
                         minimum height=1.6cm, align=center,
                         font=\bfseries},
    rarblock/.style   = {draw=rust, thick, rounded corners, fill=rustfill,
                         minimum width=3.9cm, minimum height=2.1cm,
                         align=center},
    flow/.style       = {->, thick, slate},
    varflow/.style    = {->, thick, ochre!80!black},
    opflow/.style     = {->, thick, plum},
    lossflow/.style   = {->, thick, moss},
    rarflow/.style    = {->, very thick, rust},
    hop/.style        = {fill=figpaper, draw=none,
                         minimum width=0.2cm, minimum height=0.3cm,
                         inner sep=0pt}
]

% ---------------- background ----------------
\fill[figpaper] (-3.4,-7.8) rectangle (20.2,8.6);

% ---------------- title ----------------
\node[font=\Large\bfseries, slate] at (8.4,7.7)
    {Variational-collocation PINNs with residual-based adaptive
     refinement (VC-PINNs-RAR)};

\node[layerbox] (input)   at (0,6.0)   {Input\\$\bm{x}$};
\node[layerbox] (hidden)  at (2.6,6.0) {Residual\\hidden layers};
\node[layerbox] (output)  at (5.2,6.0) {Output\\$\hat{u}$};
\node[varblock] (varform) at (10.0,6.0)
    {Weak form on $V_K$\\[2pt]
     $a(\hat{u},v)+\big(\mathcal{N}(\hat{u}),v\big)_{\Omega}=\ell(v)$};
\node[layerbox] (testf)   at (15.0,6.0) {Test\\functions $v^{e}_{ij}$};

\draw[flow] (input)  -- (hidden);
\draw[flow] (hidden) -- (output);
\draw[flow] (output) -- (varform);
\draw[varflow] (testf) -- (varform);

\node[draw=slate, rounded corners, fill=slatefill!35,
      minimum width=7.0cm, minimum height=4.4cm] (nnbox) at (2.6,2.0) {};

\node[inpneuron] (x)     at (-0.3,2.0) {$\bm{x}$};
\node[psibox]    (psi)   at ( 0.8,2.0) {$\Psi$};
\node[fourier]   (gamma) at ( 1.9,2.0) {$\gamma$};

\foreach \i in {1,...,5}
    \node[neuron] (h1\i) at (3.1,3.7-\i*0.55) {$\tanh$};
\foreach \i in {1,...,4}
    \node[neuron] (h2\i) at (4.3,3.4-\i*0.55) {$\tanh$};

\node[neuron, minimum size=11mm, font=\scriptsize] (u) at (5.7,2.0)
    {$\hat{u}(\bm{x};\bm{\theta})$};

\draw[slate!70] (x) -- (psi);
\draw[slate!70] (psi) -- (gamma);
\foreach \i in {1,...,5} \draw[slate!60] (gamma) -- (h1\i);
\foreach \i in {1,...,5}
    \foreach \j in {1,...,4} \draw[slate!35] (h1\i) -- (h2\j);
\foreach \j in {1,...,4} \draw[slate!60] (h2\j) -- (u);

\draw[->, slate, thick]
    (h11.north) to[out=70,in=110] node[above, font=\scriptsize, slate]
    {skip: $\bm{h}^{(i)}=\bm{h}^{(i-1)}+\tanh(\cdot)$} (h21.north);

\node[font=\scriptsize, slate, align=center] at (2.6,-0.45)
    {$\hat{u}=c_{u}\,\mathcal{F}_{\bm{\theta}}\circ\gamma\circ\Psi$};

\node[adbox] (ad) at (7.9,2.0)
    {\textbf{\footnotesize Automatic}\\\textbf{\footnotesize differentiation}\\[1pt]
     $\nabla\hat{u}$\\[1pt] $\Delta\hat{u}$};
\draw[flow] (u) -- (ad);

\node[resblock] (res) at (11.8,3.3)
    {\textbf{Variational residual}\\[2pt]
     {\footnotesize$\widetilde{\mathcal{R}}^{e}_{ij}=\widetilde{\mathcal{R}}^{e,\mathrm{diff}}_{ij}+\widetilde{\mathcal{R}}^{e,\mathrm{reac}}_{ij}-\widetilde{\mathcal{R}}^{e,\mathrm{src}}_{ij}$}};
\node[font=\scriptsize, plum!80!black] at (12.6,2.25)
    {$B^{e}_{ij}\equiv0$ on $V_K$};

\draw[opflow] (ad.north) |- (res.west);

\node[collblock] (coll) at (12.8,0.3)
    {\textbf{Strong-form residual} on $\mathcal{T}_{f}$\\[2pt]
     $R(\hat{u})=-\kappa\Delta\hat{u}+\mathcal{N}(\hat{u})-f$};

\node[tfnode] (gradvk) at (15.3,4.1) {$\nabla v^{e}_{ij}$};
\node[tfnode] (vk)     at (15.3,2.5) {$v^{e}_{ij}$};
\node[tfopt]  (leg)    at (18.1,3.3)
    {Legendre bubbles\\[2pt]
     $\varphi_{i}=P_{i+1}-P_{i-1}$\\
     $\hat{v}_{ij}=\varphi_{i}\otimes\varphi_{j}$\\
     on fixed $\mathcal{E}$, $K$, $\mathcal{T}_{Q}$};

\draw[varflow] (leg.west) -- (vk.east);
\draw[varflow] (vk)  -- (gradvk);
\draw[varflow] (vk.west)     -- ([yshift=-4mm]res.east);
\draw[varflow] (gradvk.west) -- ([yshift= 4mm]res.east);

\node[lossstyle] (loss) at (7.4,-2.4)
    {Composite loss $\mathcal{L}(\bm{\theta})$\\[1pt]
     {\footnotesize$w_{v}\mathcal{L}_{v}+w_{D}\mathcal{L}_{D}
      +w_{N}\mathcal{L}_{N}+w_{f}\mathcal{L}_{f}$}};

\draw[lossflow] ([xshift=5mm]res.south west) coordinate (vtop)
    -- ++(0,-3.5) to[out=-90,in=60] (loss.north east);

\draw[lossflow]
    let \p1 = (u.south), \p2 = (loss.center),
        \p3 = (loss.east), \p4 = (loss.north),
        \n1 = {\x3-\x2}, \n2 = {\y4-\y2},
        \n3 = {\n2*sqrt(1-((\x1-\x2)/\n1)^2)}
    in (\p1) -- (\x1,\y2+\n3);
\node[font=\scriptsize\itshape, moss, align=right, anchor=east]
    at (5.5,-0.95)
    {$\mathcal{L}_{D}$, $\mathcal{L}_{N}$\\ on
     $\mathcal{T}_{D},\mathcal{T}_{N}$};

\draw[lossflow] (coll.south) to[out=-90,in=0] (loss.east);

\node[hop] at (vtop |- coll.west) {};
\draw[opflow] (ad.south) |- (coll.west);
\node[font=\scriptsize, plum, anchor=east] at (7.8,0.75) {nested};

% ================ RAR pipeline ================
\node[rarblock] (rar1) at (15.8,-5.2)
    {Indicator on pool $\mathcal{C}\subset\Omega$\\[2pt]
     $\rho=\big|R(\hat{u})\big|$};
\node[rarblock] (rar2) at (11.1,-5.2)
    {Guard \& rank\\[2pt]
     $\mathcal{C}_{\mathrm{adm}}=\{\operatorname{dist}(\bm{x},\mathcal{S}^{*})\ge r_{\mathrm{g}}\}$\\
     top-$m$ $\rightarrow\mathcal{M}$};
\node[rarblock] (rar3) at (6.4,-5.2)
    {Augment collocation set\\[2pt]
     $\mathcal{T}_{f}\leftarrow\mathcal{T}_{f}\cup\mathcal{M}$};

\draw[rarflow, rounded corners=6pt]
    (coll.east) -| (rar1.north);

\draw[rarflow] (rar1) -- (rar2);
\draw[rarflow] (rar2) -- (rar3);

\draw[rarflow, rounded corners=8pt]
    (rar3.west) -| (-2.2,2.0) -- (nnbox.west);
\node[font=\scriptsize\itshape, rust, align=center] at (-1.35,-1.4)
    {retrain on\\augmented\\ $\mathcal{T}_{f}$};

\node[draw=rust, dashed, thick, rounded corners,
      fit=(rar1)(rar2)(rar3), inner sep=11pt,
      label={[rust, font=\bfseries]below:%
             RAR: adaptive collocation ($\mathcal{T}_{f}$ only)}]
      (rarbox) {};

% ---------------- outer frame ----------------
\node[draw=slate, dashed, thick, rounded corners,
      fit={(-3.2,-7.6) (20.0,8.4)}, inner sep=3pt] {};
\end{tikzpicture}%
}
\caption{Architecture of the proposed framework. A multiscale Fourier-feature embedding $\gamma$, whose frequency matrices are drawn once and held fixed, acts on the enriched coordinate $\bm{z}=\Psi(\bm{x})$ and feeds a residual $\tanh$ network whose output, scaled by $c_{u}$, is the trial function $\hat{u}(\bm{x};\bm{\theta})$, $\Psi$ appends the leading singular mode where the corner behavior is known and is the identity otherwise. One application of automatic differentiation supplies $\nabla\hat{u}$ for the variational residual $\widetilde{\mathcal{R}}^{e}_{ij}$, assembled on the fixed quadrature set $\mathcal{T}_{Q}$ against the Legendre bubble test functions $v^{e}_{ij}$, a
second, nested application supplies $\Delta\hat{u}$ for the strong-form residual $R(\hat{u})$ on the collocation set $\mathcal{T}_{f}$. As every $v^{e}_{ij}$ vanishes on $\partial\Omega_{e}$, the natural boundary functional $B^{e}_{ij}$ vanishes identically and the Neumann data reach the loss only through the penalty $\mathcal{L}_{N}$. The composite loss is the weighted sum of the variational residual, the collocation term and the two boundary penalties. The RAR loop branches from the strong-form residual, the only loss component that can be evaluated at an arbitrary interior point. It scores a candidate pool $\mathcal{C}$, discards candidates inside the guard radius $r_{\mathrm{g}}$, inserts the $m$ highest ranked points into $\mathcal{T}_{f}$ and continues training from the current parameters. The element partition, the number $K$ of test functions per direction and the quadrature set $\mathcal{T}_{Q}$ are fixed throughout, so the scheme performs neither $h$- nor $p$-refinement, adaptivity acts on $\mathcal{T}_{f}$ alone.}
\label{fig:vcpinn-rar-architecture}
\end{figure}

% Neumerical Examples Block %
\section{Numerical examples}\label{Sec5_neumerical_examples}
We compare the VC-PINNs with uniform insertion (the uniform arm) and with residual-based adaptive refinement (VC-PINNs-RAR, the RAR arm) on five second-order elliptic benchmarks whose solutions are singular, nearly singular, or strongly localized. The first problem is the Laplace equation on a slit square with traction-free crack faces, whose tip has interior angle $2\pi$. The second is the Laplace equation on a cross-shaped domain with four re-entrant corner singularities and mixed Dirichlet-Neumann boundary conditions. The third is an Allen-Cahn type reaction-diffusion problem~\cite{AllenCahn1979} on the square with homogeneous Neumann data and four localized Gaussian peaks. The fourth is a monotone semilinear problem on an L-shaped domain, in which a re-entrant corner singularity exists with a thin boundary layer. The fifth is the classical Motz problem, whose singularity arises where the boundary condition changes type on a straight edge. Exact solutions, closed form or manufactured, are available for the first three problems, for the fourth and fifth, a high resolution finite element reference is used. In every benchmark problem the uniform collocation and RAR arms are budget matched, with the same final collocation count, the same total Adam iterations and the same number of L-BFGS stages, they differ in how the inserted points are chosen and for the Motz problem, in the region from which they are drawn.

All experiments are run in PyTorch~\cite{Paszke2019} in double precision on a single NVIDIA RTX~5080 GPU, with Adam followed by L-BFGS (section~\ref{Sec3.5_algorithm_&_scope}). The loss weights were $(w_v,w_D,w_N,w_f)=(1,50,50,10)$ for all benchmarks except the reaction-diffusion problem, where $(w_v,w_D,w_N,w_f)=(1,100,50,5)$ was used, there $w_D$ acts on interior anchor points rather than on a Dirichlet boundary.

% PDE 1
\subsection{Laplace equation on a slit domain with traction-free crack faces}\label{creak_slit_PDE1}
We consider the Laplace equation on the slit square
\begin{equation}
\Delta u = 0 \text{ in } \Omega = (-1,1)^2 \setminus \Sigma \text{ where } \Sigma = \{(x,0):0\le x\le 1\},
\end{equation}
with homogeneous Neumann data prescribed on the two faces of the slit and Dirichlet data taken from the exact solution on the outer boundary. This problem models antiplane shear (Mode~III) deformation of a cracked elastic body~\cite{Anderson2017}. In crack-tip polar coordinates with the branch cut aligned with the slit, $\theta\in(0,2\pi)$, the exact solution is
\begin{equation}
u(r,\theta)=r^{1/2}\cos\!\left(\frac{\theta}{2}\right).
\label{eq:crack_exact}
\end{equation}
The crack tip has interior angle $\omega=2\pi$, so the leading singular exponent is $\lambda_{1}=\frac{\pi}{\omega}=\tfrac12$ (table~\ref{tab:exponents}). Hence $u\in H^{s}(\Omega)$ for every $s<3/2$, however, $u\notin H^{2}(\Omega)$ and $|\nabla u|\sim r^{-1/2}$ as $r\to0$.

This problem presents two major challenges for neural-network solvers. First, smooth activation functions struggle to approximate the non smooth behavior $r^{1/2}$ near the crack tip, causing the approximation error to localize around the singularity while the optimization process remains dominated by the smooth bulk region. Second, $u$ jumps across the slit, taking the values $\pm\sqrt{x}$ on the two faces and no network that is continuous in $(x,y)$ can represent this jump, whatever its capacity (remark~\ref{rem:crackenrich}).

To address these issues, the computational domain is partitioned into four rectangular elements that meet on the coordinate axes, with the slit on an element interface, so that no element crosses the discontinuity. In addition, the harmonic crack-tip coordinates
\begin{equation}
s_1=r^{1/2}\cos\!\left(\frac{\theta}{2}\right),
\qquad
s_2=r^{1/2}\sin\!\left(\frac{\theta}{2}\right),
\end{equation}
are appended to the network input. These enriched features encode both the leading singular mode and the branch cut discontinuity that cannot be represented using $(x,y)$ alone.

The trial space is parameterized by a multiscale Fourier-feature embedding with $n_{F}=64$ frequencies at each of the scales $\{\sigma_{1},\sigma_{2},\sigma_{3}\}=\{1,5,15\}$, followed by six $\tanh$ hidden layers of width $100$, the last five with residual connections. The test space uses $K=8$ Legendre bubble functions and $Q=24$ Gauss-Legendre quadrature points per direction. The training objective combines the variational residual, boundary penalties and a strong-form collocation term,
\begin{equation}
\mathcal{L}
=
\mathcal{L}_v
+
w_D\mathcal{L}_D
+
w_N\mathcal{L}_N
+
w_f\mathcal{L}_f,
\end{equation}
with weights $(w_v,w_D,w_N,w_f) = (1,50,50,10).$ Training runs for thirteen stages, an initial stage followed by twelve rounds, with each stage combining an Adam phase and a subsequent L-BFGS phase. The collocation set starts with $1500$ points and gains $200$ points per round, ending at $|\mathcal{T}_{f}|=3900$ after $150{,}000$ Adam iterations in total.

The uniform and RAR arms differ only in how the newly inserted collocation points are selected. In the uniform arm, points are sampled uniformly throughout the admissible region. The RAR arm instead scores a candidate pool of $10^{5}$ points by the strong-form residual. Direct ranking by $|R|$ is ineffective. The exact solution is harmonic, so the residual equals, up to sign, the Laplacian of the approximation error, which grows approximately like $r^{-2.3}$ near the crack tip (section~\ref{Sec3.4_rar_collocations}). The largest residuals therefore mark proximity to the singularity rather than large local error.

To mitigate this effect, points are selected by the three step rule of section~\ref{Sec3.4_rar_collocations}, radial normalization of the residual by its median in twelve logarithmically spaced bands~\eqref{eq:bandscore}, density sampling and a minimum-separation filter, here with $r_{\mathrm{g}}=2\times10^{-3}$.

\begin{table}[htbp]
\centering
\caption{Crack problem on the slit square (Mode~III). All errors are reported in percent (\%) and bold marks the more accurate arm in each row. Both arms use the same collocation schedule ($|\mathcal{T}_{f}|=1500\to3900$), $150{,}000$ Adam iterations and thirteen L-BFGS stages, one training run per arm. Errors are evaluated on a uniform $301\times301$ grid on $[-1,1]^2$. The relative $L^\infty$ error is $\max|u-\hat u|/\max|u|$ over the grid, where $\max|u|=2^{1/4}\cos(\pi/8)\approx1.099$ is attained at the corners $(1,\pm1)$, the $H^1$ semi-norm excludes grid points with $r\le10^{-3}$.}
\label{tab:crack_results}
\begin{tabular}{lcc}
\toprule
Metric & VC-PINNs (uniform) & VC-PINNs-RAR \\
\midrule
Relative $L^{2}$ error           & $0.4138$ & $\mathbf{0.3545}$ \\
Relative $L^{\infty}$ error      & $2.462$  & $\mathbf{1.889}$  \\
Relative $H^{1}$ semi-norm error & $2.374$  & $\mathbf{1.442}$  \\
\bottomrule
\end{tabular}
\end{table}

Both arms resolve a solution with an unbounded gradient and a jump across the slit, with relative $L^{2}$ errors below $0.5\%$ (table~\ref{tab:crack_results}). This indicates that the element partition and the crack-tip enrichment together capture the singular structure of the solution; as $s_{1}$ coincides with the exact solution~\eqref{eq:crack_exact}, the remaining error reflects the discretization rather than the representation of the singularity. At identical collocation and iteration budgets, RAR reduces the relative $L^{2}$ error by $14\%$ (a factor of $1.17$), the relative $L^{\infty}$ error by $23\%$ (a factor of $1.30$) and the relative $H^{1}$ semi-norm error by $39\%$ (a factor of $1.65$). The largest improvement occurs in the $H^1$ semi-norm, which measures the error in the gradient, where the singularity resides.

The spatial error distributions show where these improvements occur. With uniform insertion, the error remains concentrated around the crack tip and extends into a broad lobe away from the singular region (Fig.~\ref{fig:PDE1_figures}C). RAR suppresses these secondary structures and produces a nearly radially symmetric error pattern around the crack tip (Fig.~\ref{fig:PDE1_figures}I). This is consistent with the adaptive strategy allocating collocation points to regions where the approximation error is most significant and thus making more effective use of the same collocation budget.

The adaptive collocation set is graded toward the crack tip while maintaining adequate coverage of the far field (Fig.~\ref{fig:PDE1_figures}F). This behavior closely resembles mesh grading in adaptive finite element methods. In the loss histories (Fig.~\ref{fig:PDE1_figures}J and K), the jump in $\mathcal{L}_{f}$ at an insertion is larger in the RAR arm in most rounds; its inserted points are chosen for their large residual.

\begin{figure}[!htbp]
\centering
%---------------- Row 1 ----------------%
\begin{subfigure}{0.32\textwidth}
    \textbf{A}\\[0.5mm]
    \includegraphics[width=\linewidth]{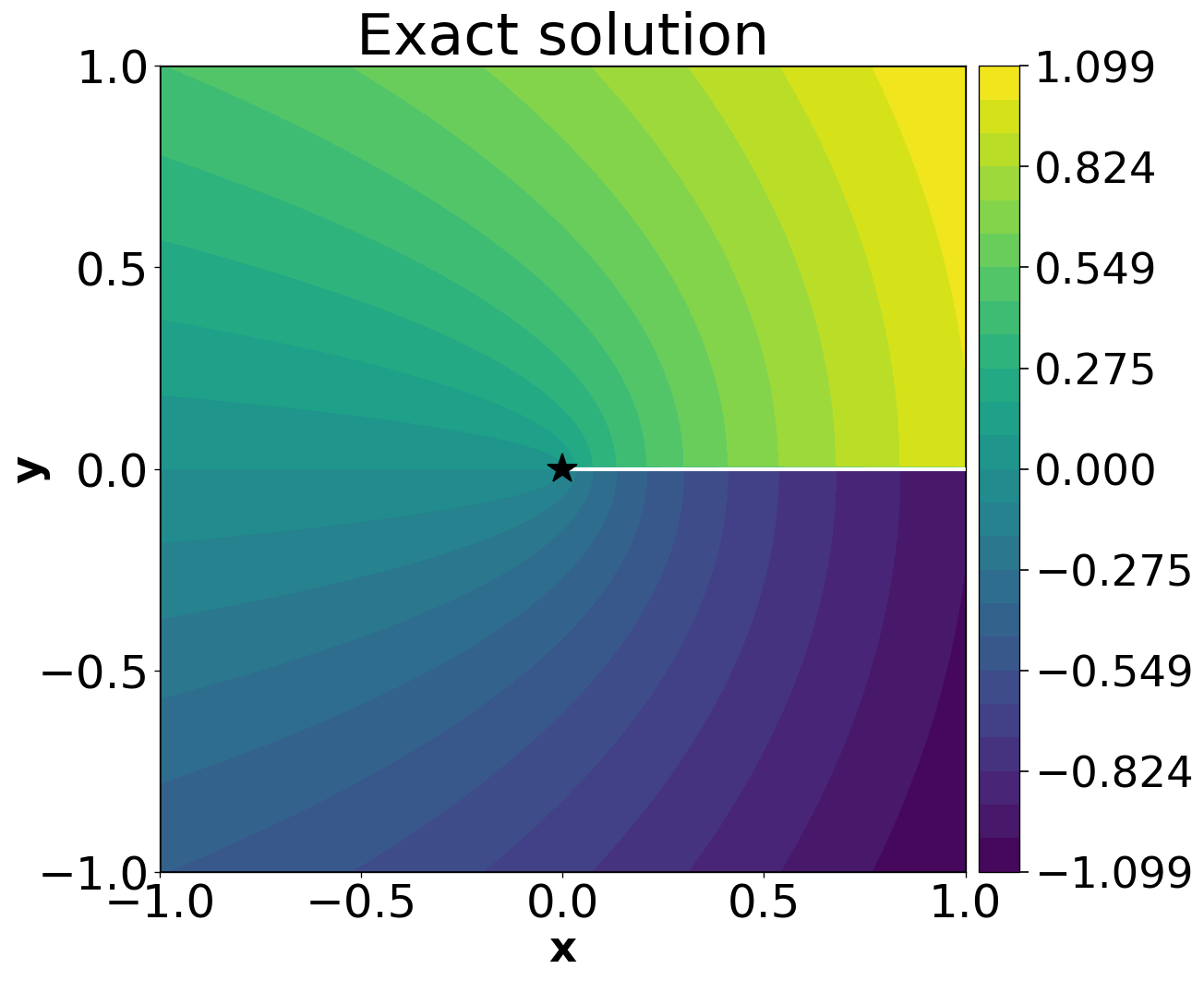}
\end{subfigure}
\hfill
\begin{subfigure}{0.32\textwidth}
    \textbf{B}\\[0.5mm]
    \includegraphics[width=\linewidth]{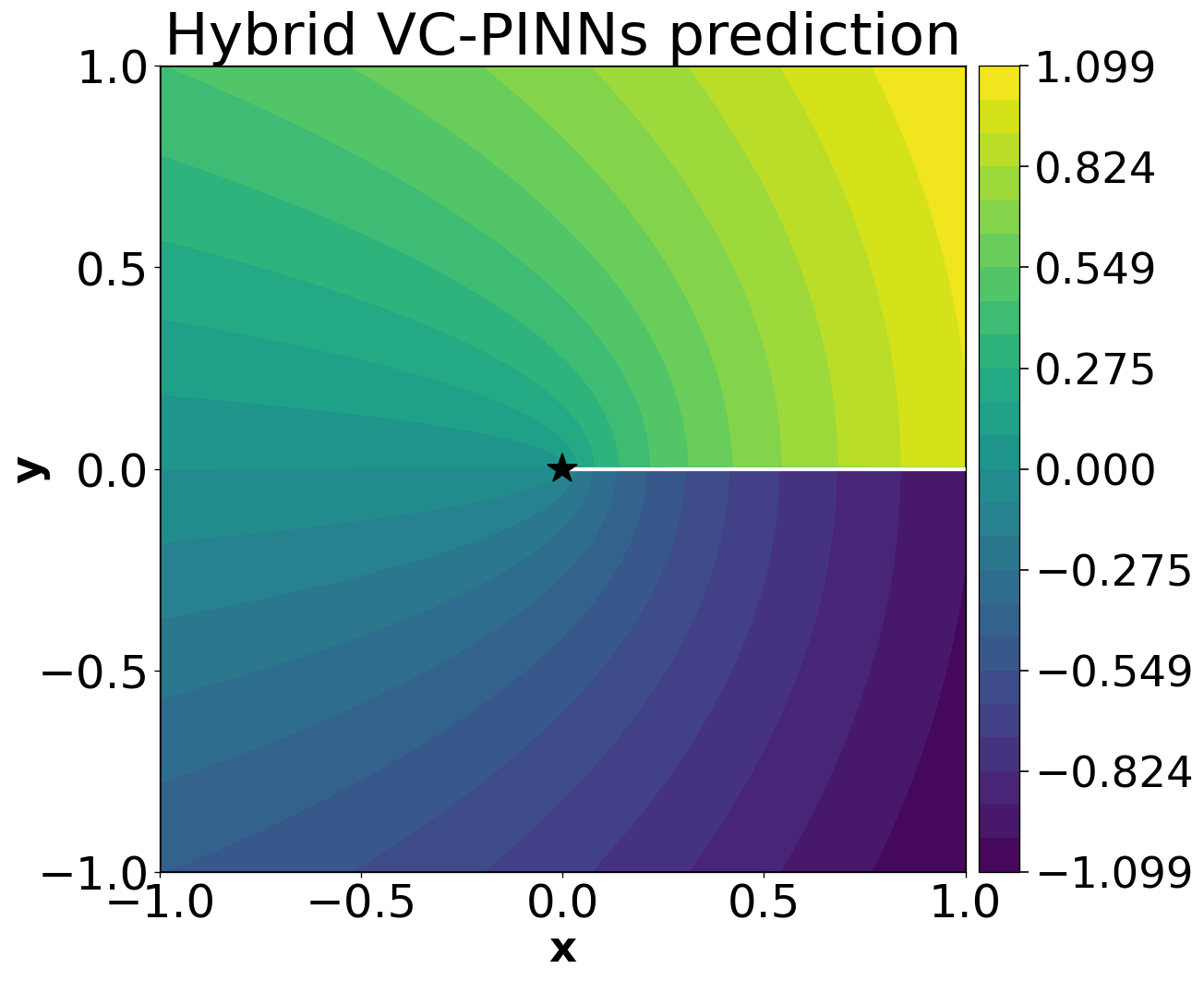}
\end{subfigure}
\hfill
\begin{subfigure}{0.32\textwidth}
    \textbf{C}\\[0.5mm]
    \includegraphics[width=\linewidth]{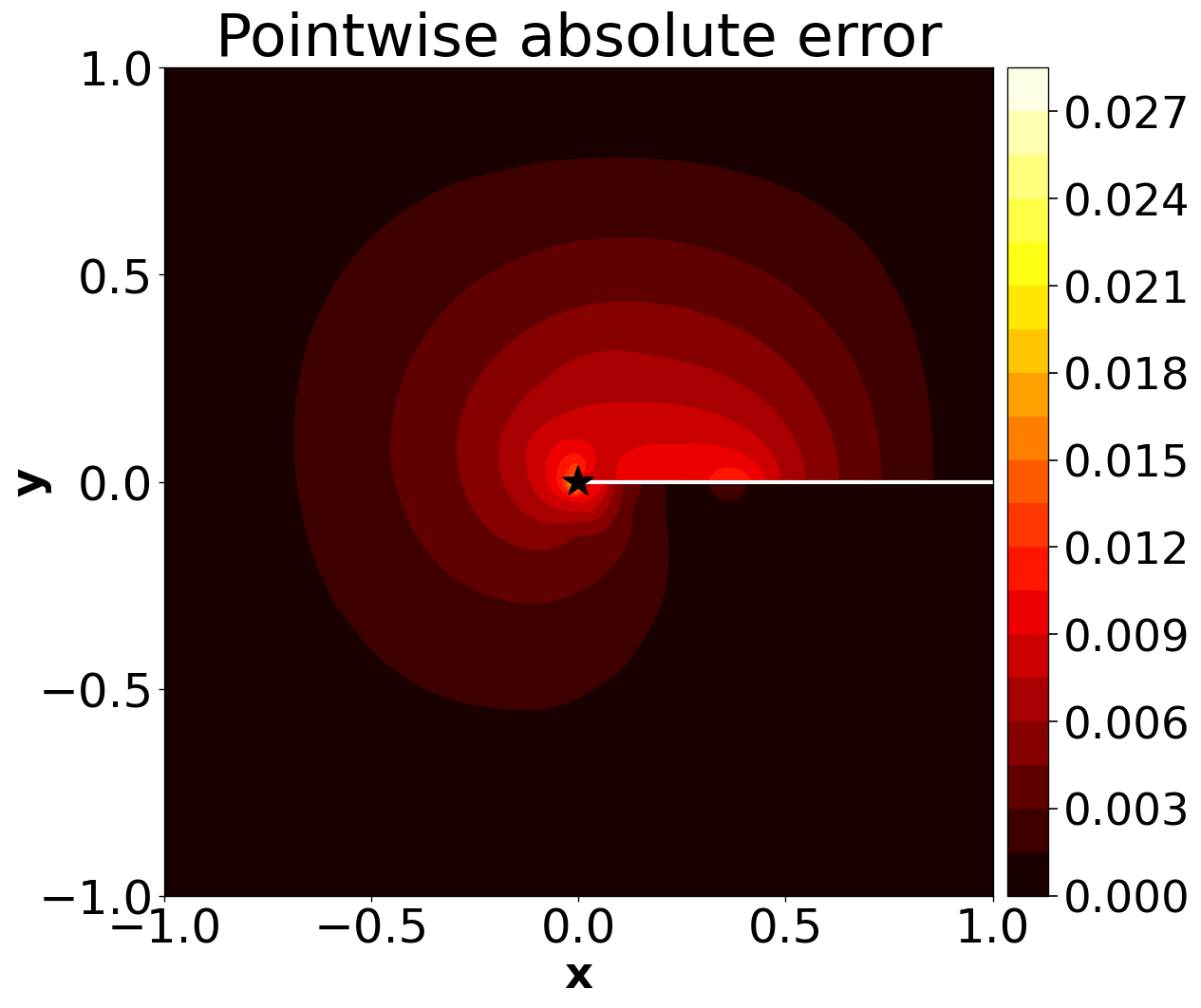}
\end{subfigure}

%---------------- Row 2 ----------------%
\begin{subfigure}{0.32\textwidth}
    \textbf{D}\\[0.5mm]
    \includegraphics[width=\linewidth]{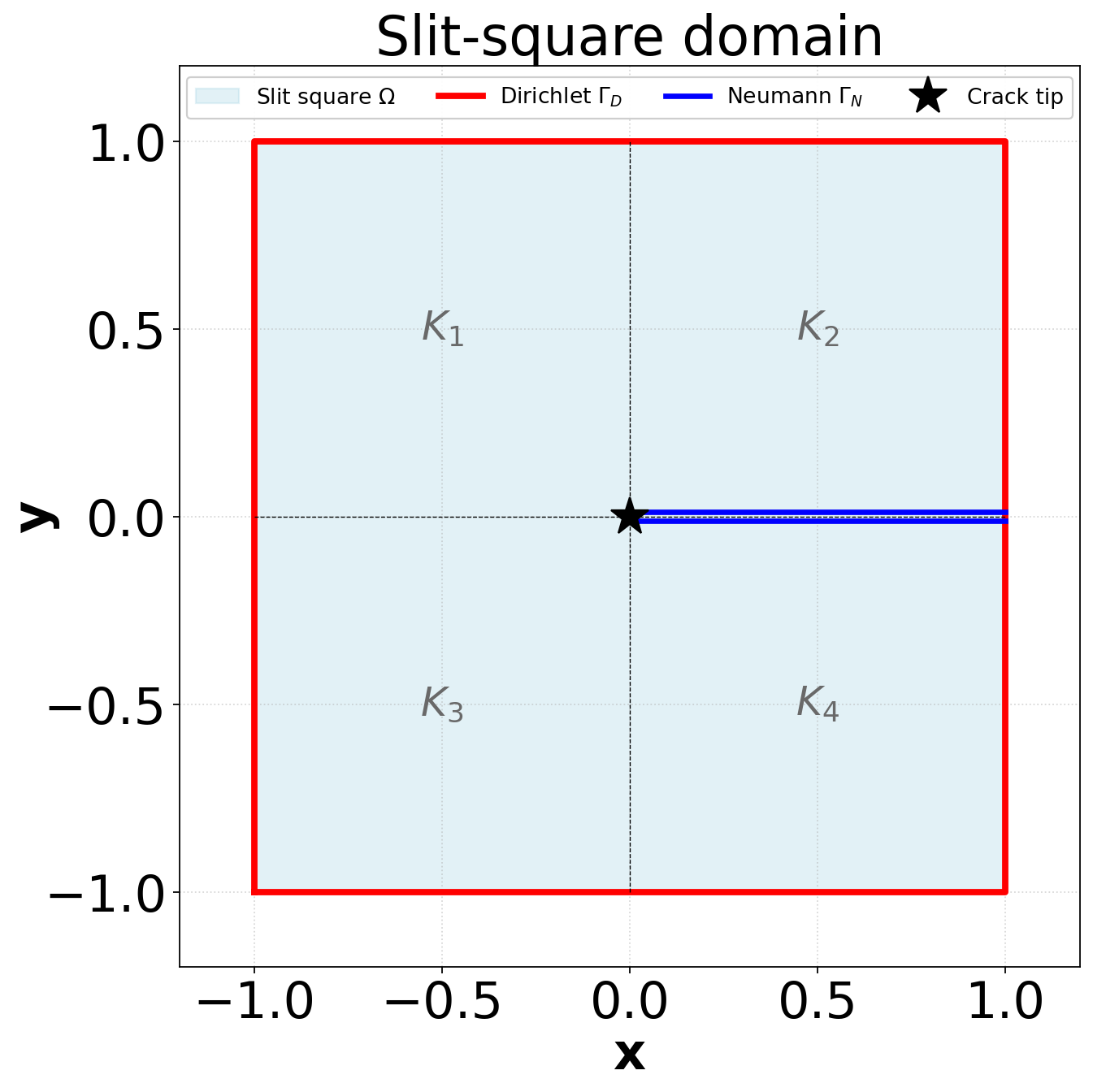}
\end{subfigure}
\hfill
\begin{subfigure}{0.32\textwidth}
    \textbf{E}\\[0.5mm]
    \includegraphics[width=\linewidth]{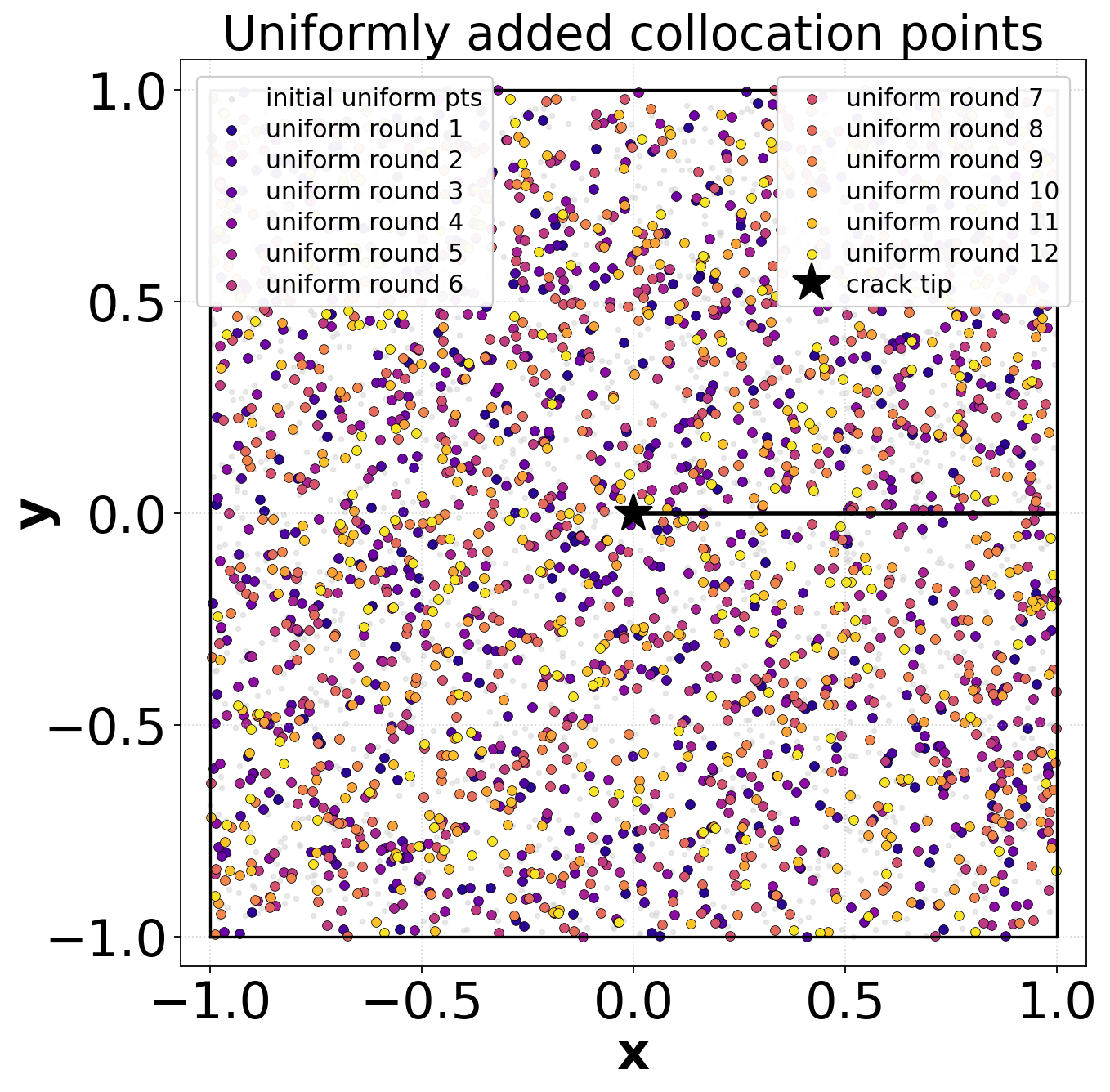}
\end{subfigure}
\hfill
\begin{subfigure}{0.32\textwidth}
    \textbf{F}\\[0.5mm]
    \includegraphics[width=\linewidth]{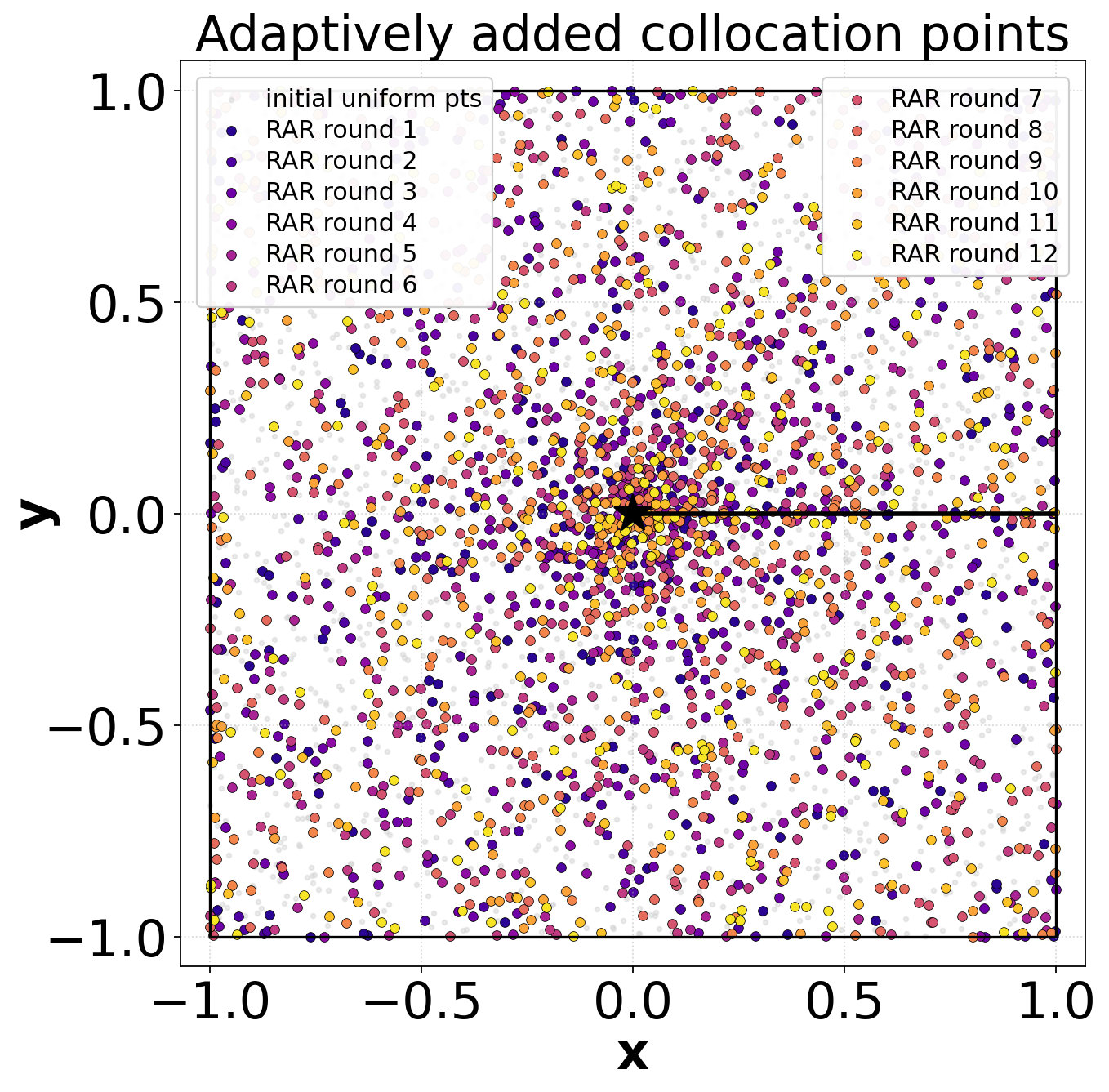}
\end{subfigure}

%---------------- Row 3 ----------------%
\begin{subfigure}{0.32\textwidth}
    \textbf{G}\\[0.5mm]
    \includegraphics[width=\linewidth]{crack_face_exact.png}
\end{subfigure}
\hfill
\begin{subfigure}{0.32\textwidth}
    \textbf{H}\\[0.5mm]
    \includegraphics[width=\linewidth]{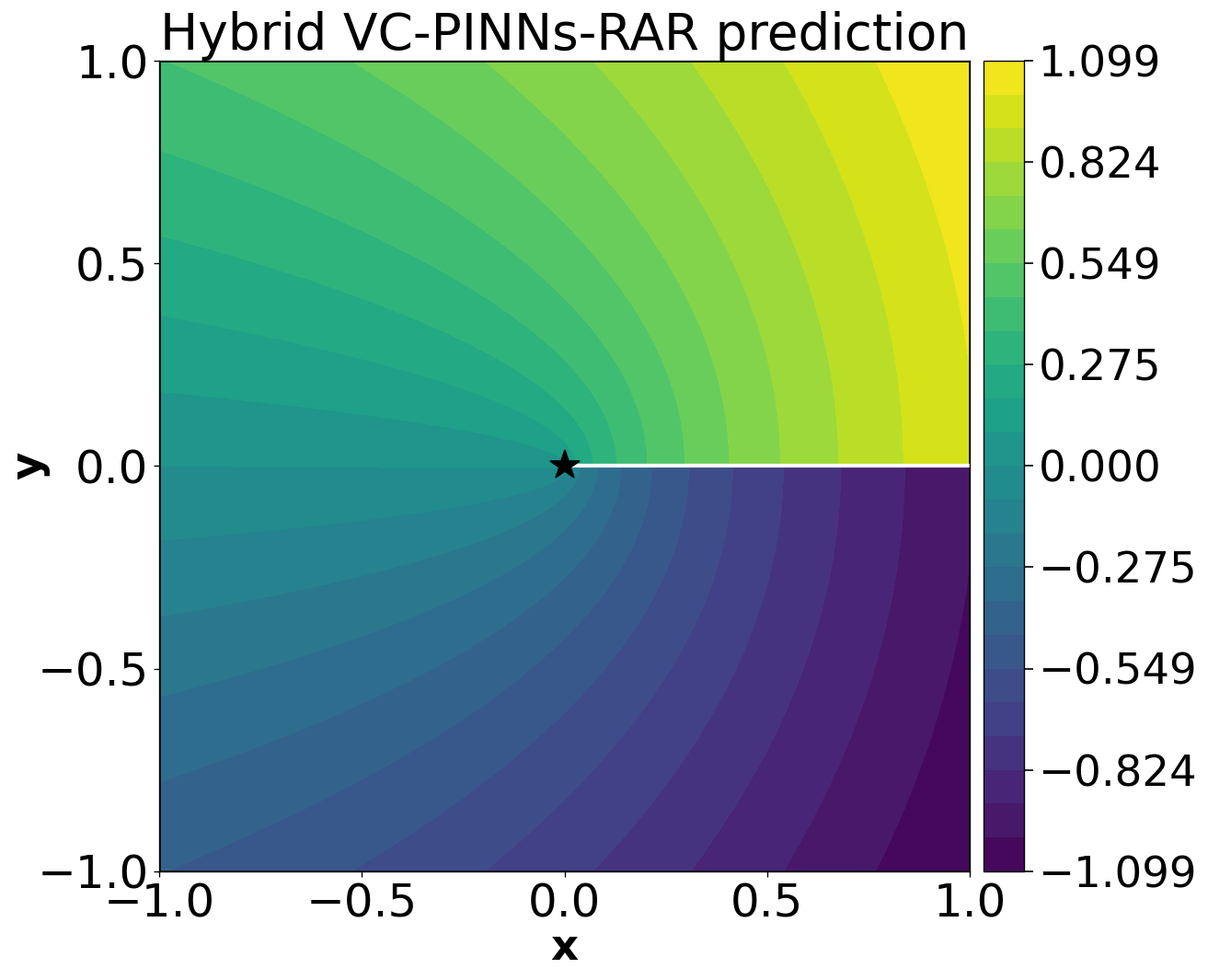}
\end{subfigure}
\hfill
\begin{subfigure}{0.32\textwidth}
    \textbf{I}\\[0.5mm]
    \includegraphics[width=\linewidth]{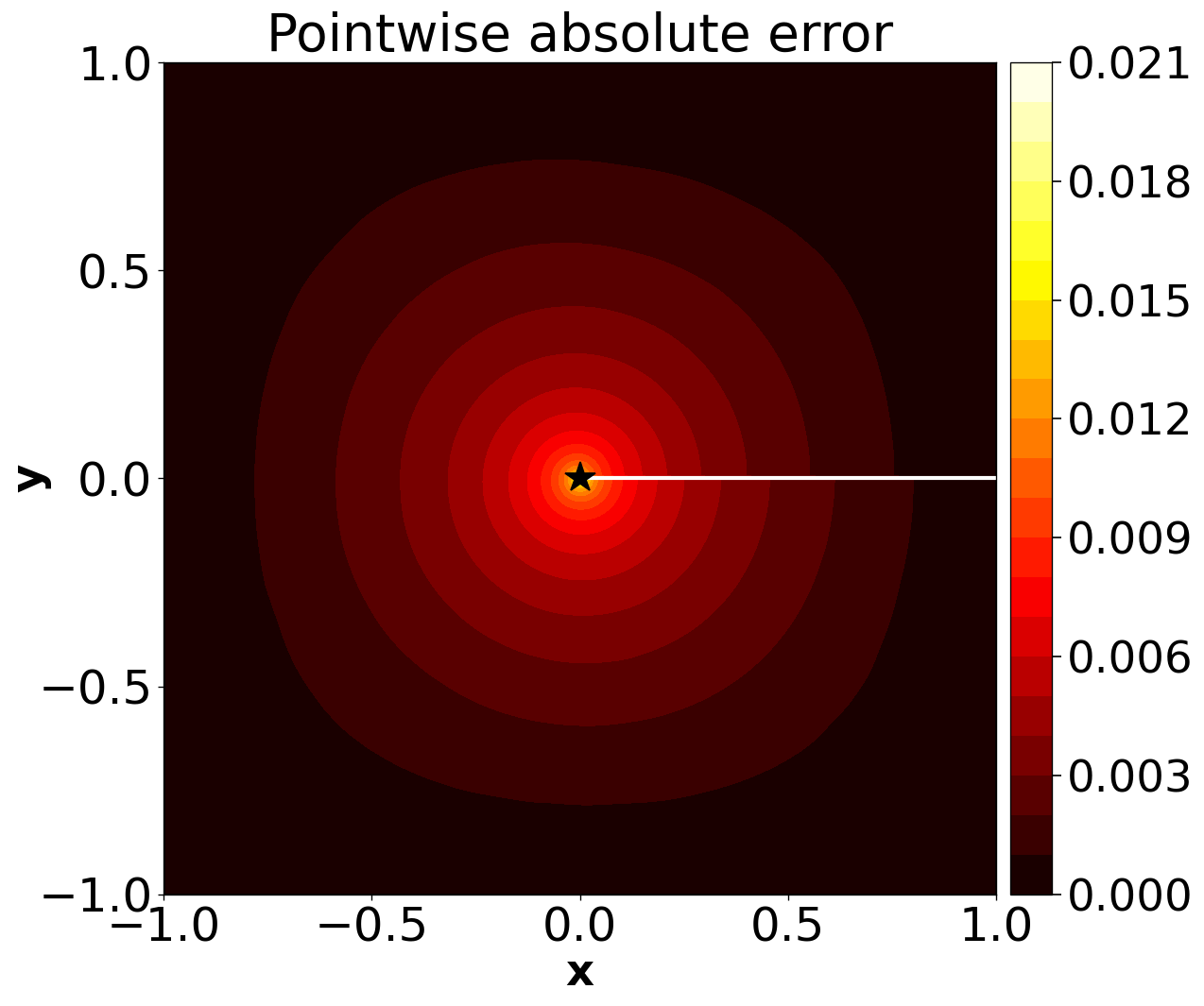}
\end{subfigure}

%---------------- Row 4 ----------------%
\begin{subfigure}{0.44\textwidth}
    \textbf{J}\\[0.5mm]
    \includegraphics[width=\linewidth]{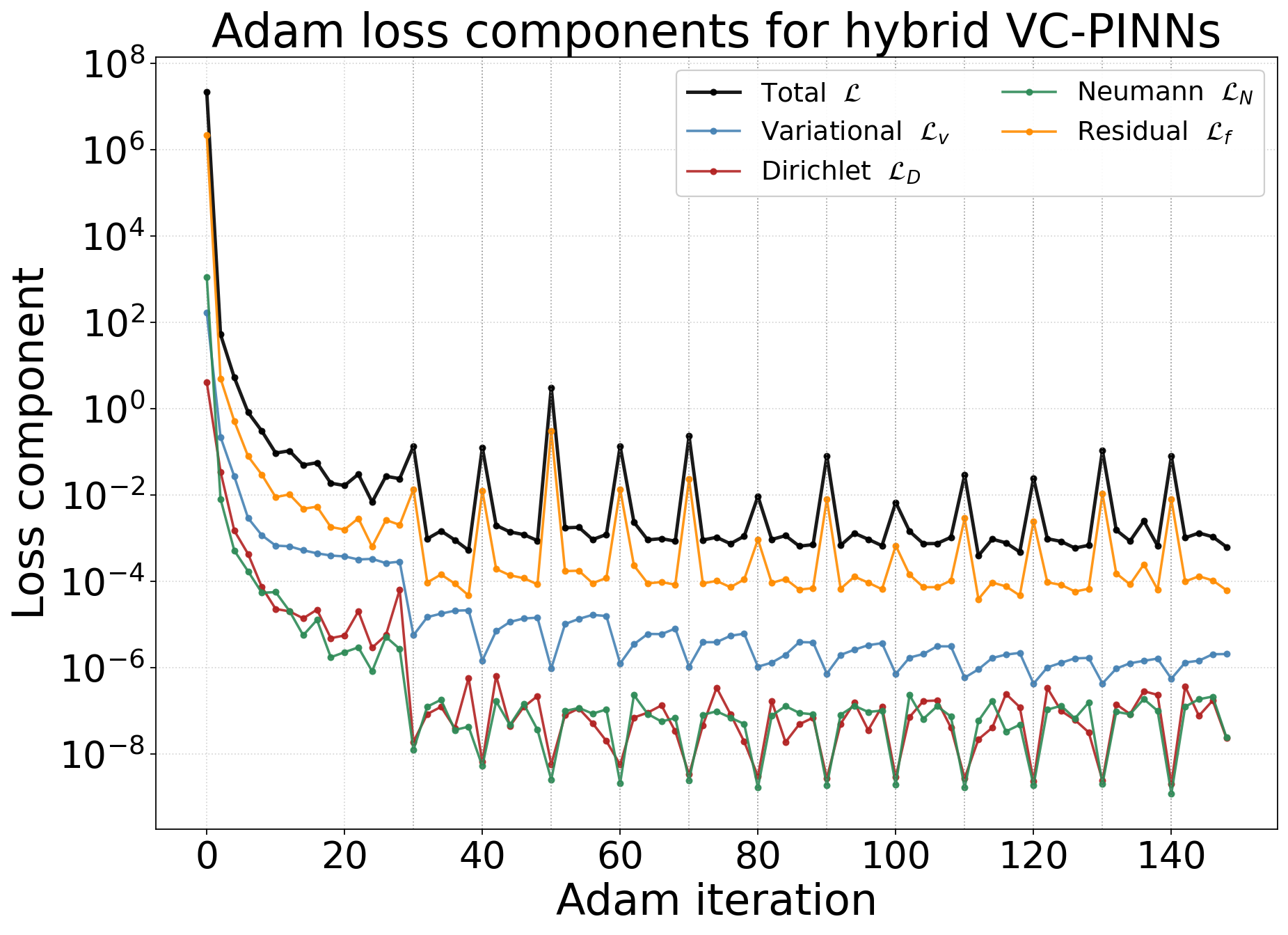}
\end{subfigure}
\hspace{0.04\textwidth}
\begin{subfigure}{0.44\textwidth}
    \textbf{K}\\[0.5mm]
    \includegraphics[width=\linewidth]{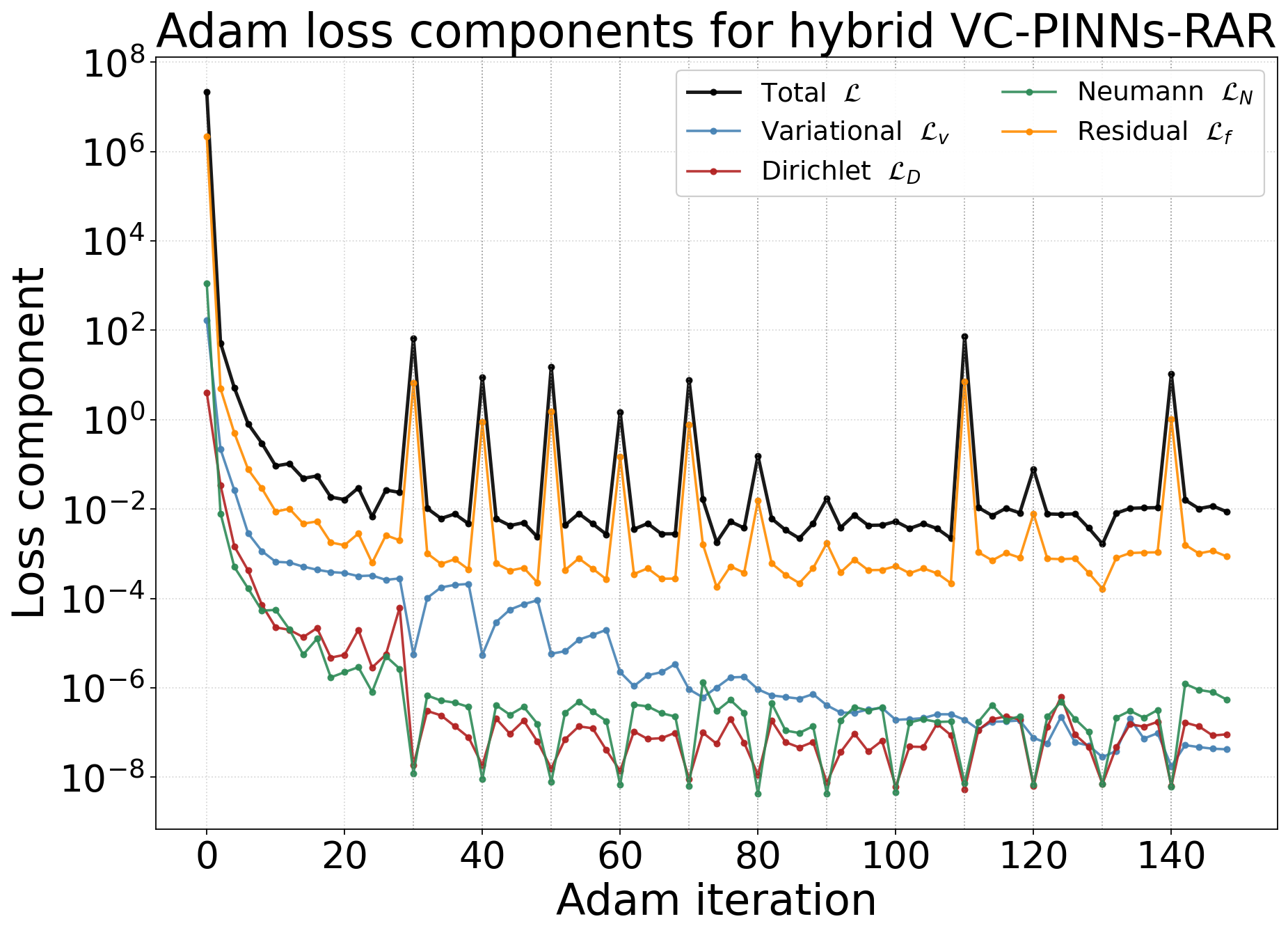}
\end{subfigure}
\vspace{-8pt}
\caption{Crack problem on the slit square (Mode~III). \textbf{(A-C)} Exact solution, VC-PINNs (uniform) prediction and pointwise absolute error. \textbf{(D-F)} Domain with boundary conditions and element partition, uniform collocation set and adaptive collocation set after twelve rounds, $|\mathcal{T}_{f}|=1500\to3900$. \textbf{(G-I)} Exact solution, VC-PINNs-RAR prediction and pointwise absolute error; panels C and I have separate color scales. \textbf{(J-K)} Adam loss components for the uniform and the RAR arm.}\label{fig:PDE1_figures}
\end{figure}

%PDE2
\subsection{Laplace equation on a cross-shaped domain with four re-entrant corners}\label{plus_shaped_PDE2}
This example tests whether RAR can resolve several singularities simultaneously and more stringently, rank them: the four corners of the cross share the same exponent yet carry different amplitudes, so an effective adaptive scheme must distribute points in proportion to their strength rather than merely locate the singular points.

We consider $\Delta u = 0$ on the cross-shaped domain $\Omega=\big((-1,1)\times(-\tfrac13,\tfrac13)\big)\cup\big((-\tfrac13,\tfrac13)\times(-1,1)\big)$, the union of a horizontal and a vertical bar. Its four re-entrant corners $C_1=(-\tfrac13,\tfrac13)$, $C_2=(\tfrac13,\tfrac13)$, $C_3=(\tfrac13,-\tfrac13)$ and $C_4=(-\tfrac13,-\tfrac13)$ each have interior angle $\omega=\tfrac{3\pi}{2}$; by classical corner regularity theory~\cite{Grisvard2011}, the leading singular exponent is therefore $\lambda_{1}=\frac{\pi}{\omega}=\tfrac23$ (table~\ref{tab:exponents}). Dirichlet data are prescribed on the four bar ends and Neumann data on the eight side walls, both read from the exact solution
\begin{equation}
u=\sum_{i=1}^{4}c_i\,r_i^{2/3}\sin\!\Big(\tfrac23\big(\theta_i-\theta_{0,i}\big)\Big),
\qquad (c_1,c_2,c_3,c_4)=(1,\,0.5,\,2,\,0.75),
\label{eq:cross_exact}
\end{equation}
where $(r_i,\theta_i)$ are polar coordinates at $C_i$ with the angular branch cut placed along the bisector of the removed quadrant and $\theta_{0,i}$ is the angle of one adjacent edge. Each mode is harmonic, so $f\equiv0$ exactly; each also vanishes on both edges meeting at its own corner, since $\sin(\tfrac23\cdot\tfrac{3\pi}{2})=0$.

Near $C_i$ the remaining three terms are analytic, so the local behavior is governed by the $i$-th mode alone. Hence $u\in H^{s}(\Omega)$ for every $s<5/3$; however, $u\notin H^{2}(\Omega)$, with $|\nabla u|\sim\tfrac23 c_i r_i^{-1/3}$ as $r_i\to0$. The regularity is lost at four singular points at once and each corner's gradient blow-up scales with its own amplitude. Unlike benchmarks whose difficulty is manufactured through a localized source, the obstruction here comes from the geometry rather than the data. The problem models steady heat conduction in a cross-shaped section with a prescribed heat flux on the side walls and prescribed temperatures at the bar ends, the re-entrant corners being the sites of flux concentration.

For uniform collocation, the difficulty is that of the single-corner case multiplied by four: a uniform budget assigns each corner neighborhood a share proportional only to its area and the four singular regions divide that share among them. Since $f\equiv0$, the strong-form residual measures purely the failure of the network to be harmonic and therefore peaks at the corners. The unequal amplitudes turn this into a ranking problem: the strongest corner $C_3$ contributes most to the global loss and the weakest, $C_2$, least.

The domain is tiled by three rectangular elements, the full vertical bar together with the left and right side bars, so that no element leaves $\Omega$. The network input is enriched with the powers $r_{i}^{2/3}$ of the four corner distances, which expose the radial rate of every singularity while leaving the angular structure to be learned. The trial space is a multiscale Fourier-feature embedding followed by six $\tanh$ hidden layers of width $100$, the last five with residual connections; the test space uses $K=8$ Legendre bubble functions and $Q=24$ Gauss-Legendre nodes per direction. The objective combines the variational residual, boundary penalties and a strong-form collocation term with weights $(w_v,w_D,w_N,w_f)=(1,50,50,10)$, using $300$ Dirichlet points per bar end and $200$ Neumann points per side wall. Both arms start from the same $8000$ collocation points, insert $200$ points in each of twelve rounds to reach $|\mathcal{T}_{f}|=10{,}400$ and are trained for $150{,}000$ Adam iterations and thirteen L-BFGS stages. They differ only in placement: the uniform arm draws the inserted points uniformly and the RAR arm takes the $200$ largest-residual points from a candidate pool of $25{,}000$ points.

\begin{table}[t]
\centering
\caption{Cross problem. All errors are reported in percent (\%) and bold marks the more accurate arm in each row. Both arms use the same collocation schedule ($|\mathcal{T}_{f}|=8000\to10{,}400$), $150{,}000$ Adam iterations and thirteen L-BFGS stages; one training run per arm. Errors are evaluated at the points of a uniform $301\times301$ grid on $[-1,1]^2$ that lie in $\overline{\Omega}$. The relative $L^{\infty}$ error is $\max|u-\hat u|/\max|u|$ over these points, where $\max|u|\approx3.577$; the $H^{1}$ semi-norm uses the network gradient from automatic differentiation and excludes points within $10^{-3}$ of a corner.}
\label{PDE_2_table}
\begin{tabular}{lcc}
\toprule
Metric & VC-PINNs (uniform) & VC-PINNs-RAR \\
\midrule
Relative $L^{2}$ error           & $5.078$ & $\mathbf{0.1601}$ \\
Relative $L^{\infty}$ error      & $13.85$ & $\mathbf{1.283}$  \\
Relative $H^{1}$ semi-norm error & $32.36$ & $\mathbf{2.248}$  \\
\bottomrule
\end{tabular}
\end{table}

RAR reduces the relative $L^{2}$ error from $5.078\%$ to $0.1601\%$ (a factor of about $32$), the relative $H^{1}$ semi-norm error by a factor of about $14$ and the relative $L^{\infty}$ error by a factor of about $11$ (table~\ref{PDE_2_table}; Fig.~\ref{fig:PDE2_figures}C and I). The remaining error of the RAR arm is largest around $C_{1}$. The inserted points (Fig.~\ref{fig:PDE2_figures}F) reach all four corners, most densely $C_{3}$, the corner of largest amplitude; however, a large share of them lies along the bar ends and side walls and in small clusters in the interior. Since the corner locations are already encoded in the input features, this example tests how the residual distributes points among the four corners rather than whether it finds them.

\begin{figure}[!htbp]
\centering
%---------------- Row 1 ----------------%
\begin{subfigure}{0.32\textwidth}
    \textbf{A}\\[0.5mm]
    \includegraphics[width=\linewidth]{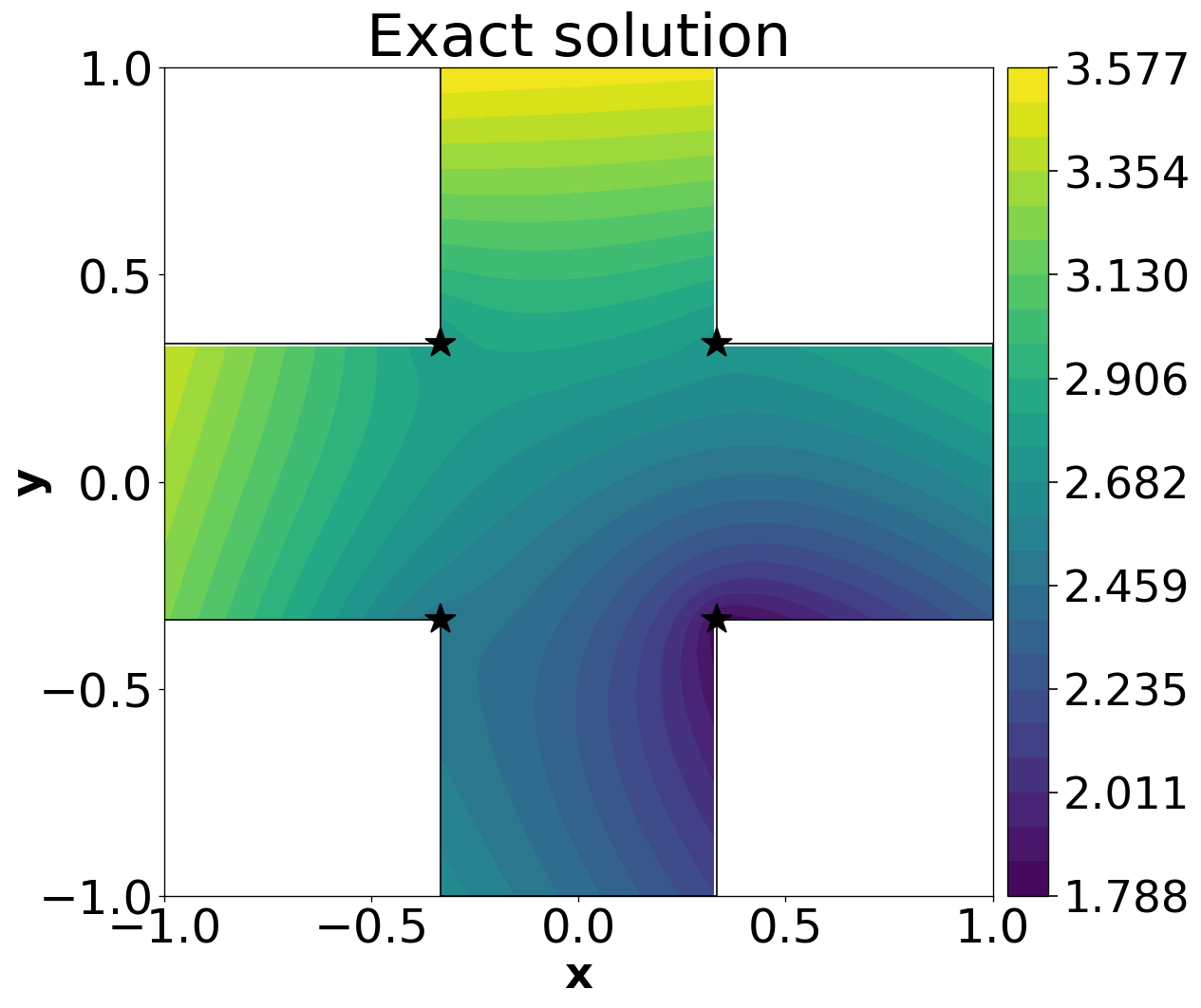}
\end{subfigure}
\hfill
\begin{subfigure}{0.32\textwidth}
    \textbf{B}\\[0.5mm]
    \includegraphics[width=\linewidth]{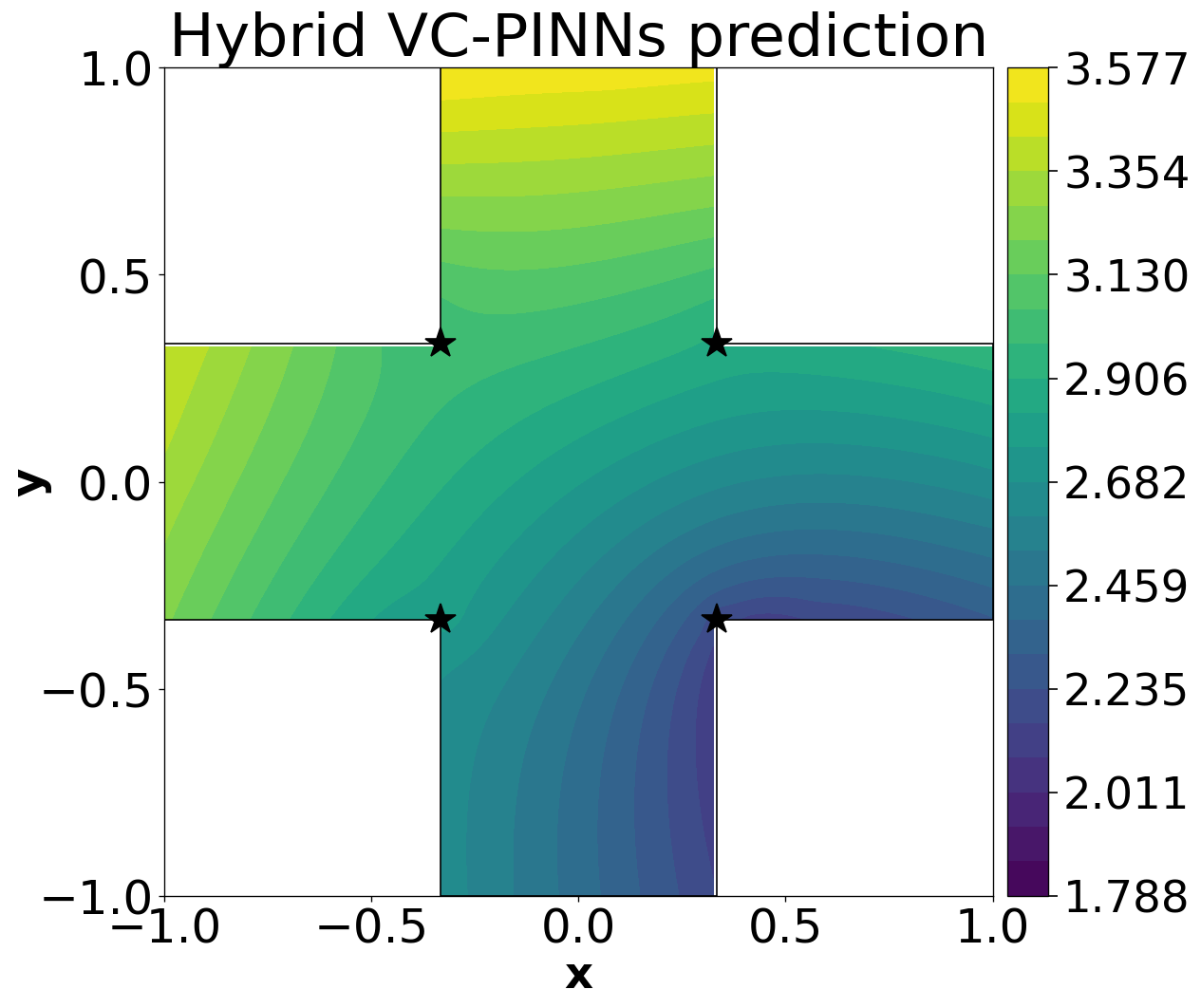}
\end{subfigure}
\hfill
\begin{subfigure}{0.32\textwidth}
    \textbf{C}\\[0.5mm]
    \includegraphics[width=\linewidth]{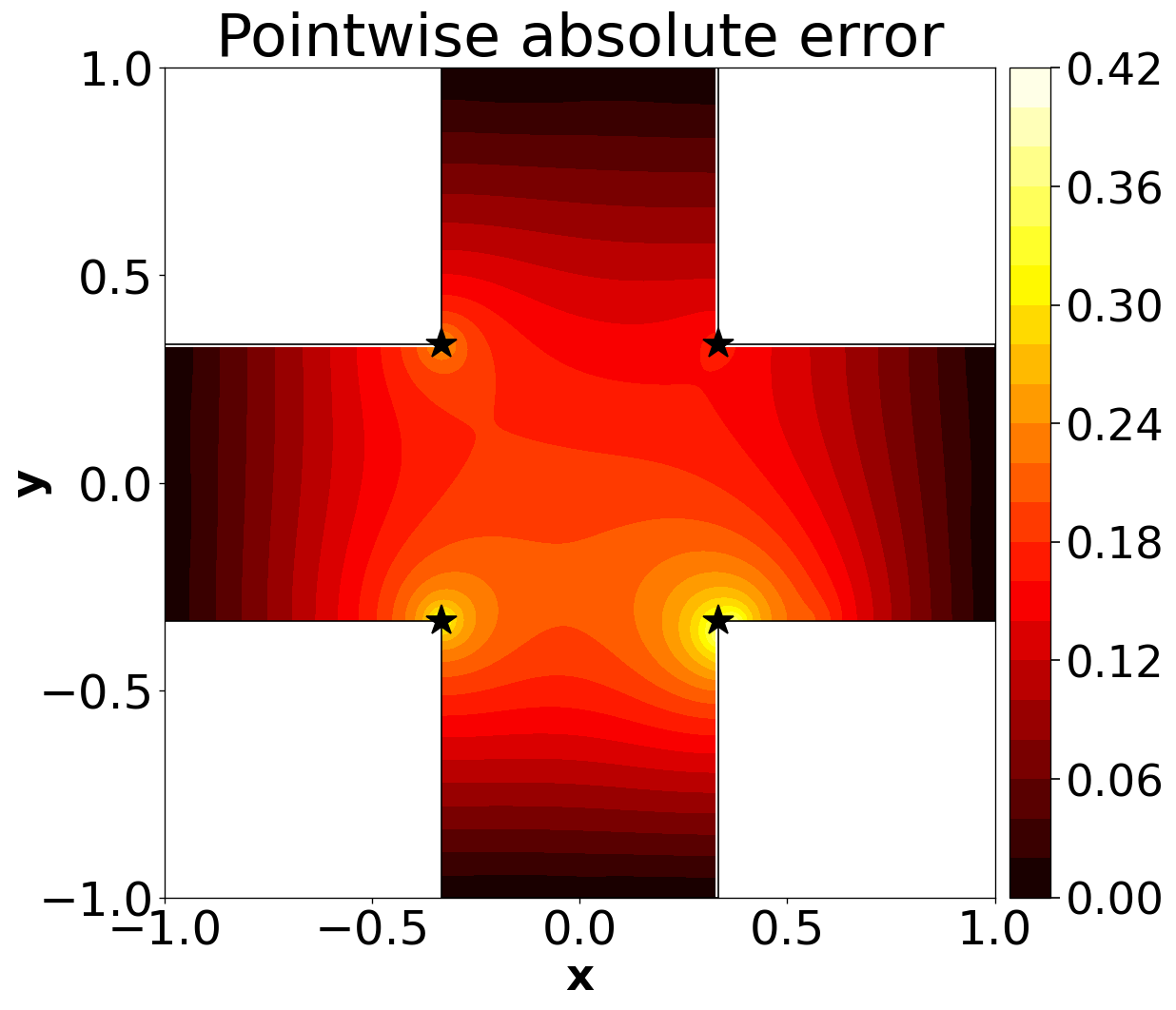}
\end{subfigure}

%---------------- Row 2 ----------------%
\begin{subfigure}{0.32\textwidth}
    \textbf{D}\\[0.5mm]
    \includegraphics[width=\linewidth]{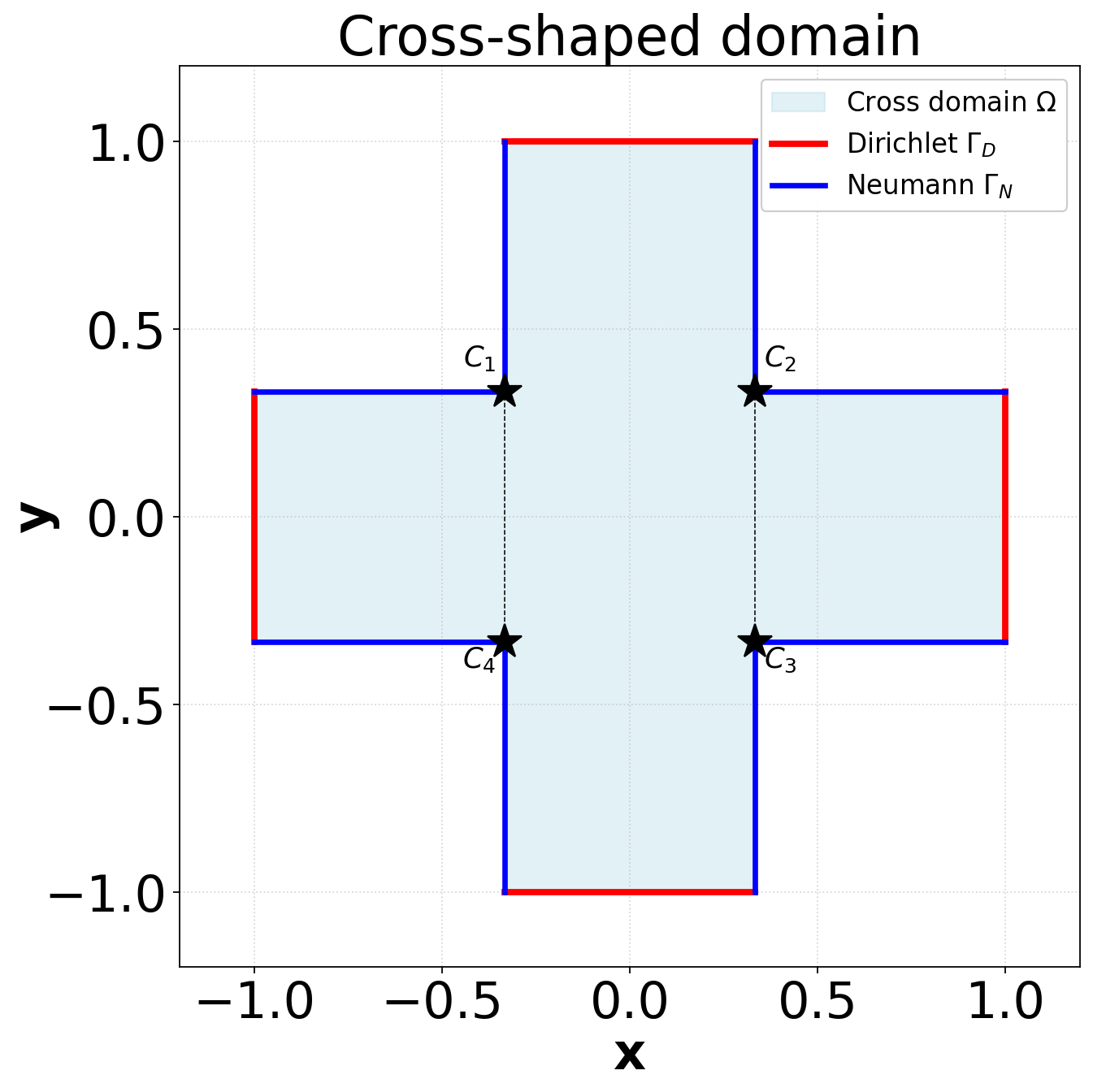}
\end{subfigure}
\hfill
\begin{subfigure}{0.32\textwidth}
    \textbf{E}\\[0.5mm]
    \includegraphics[width=\linewidth]{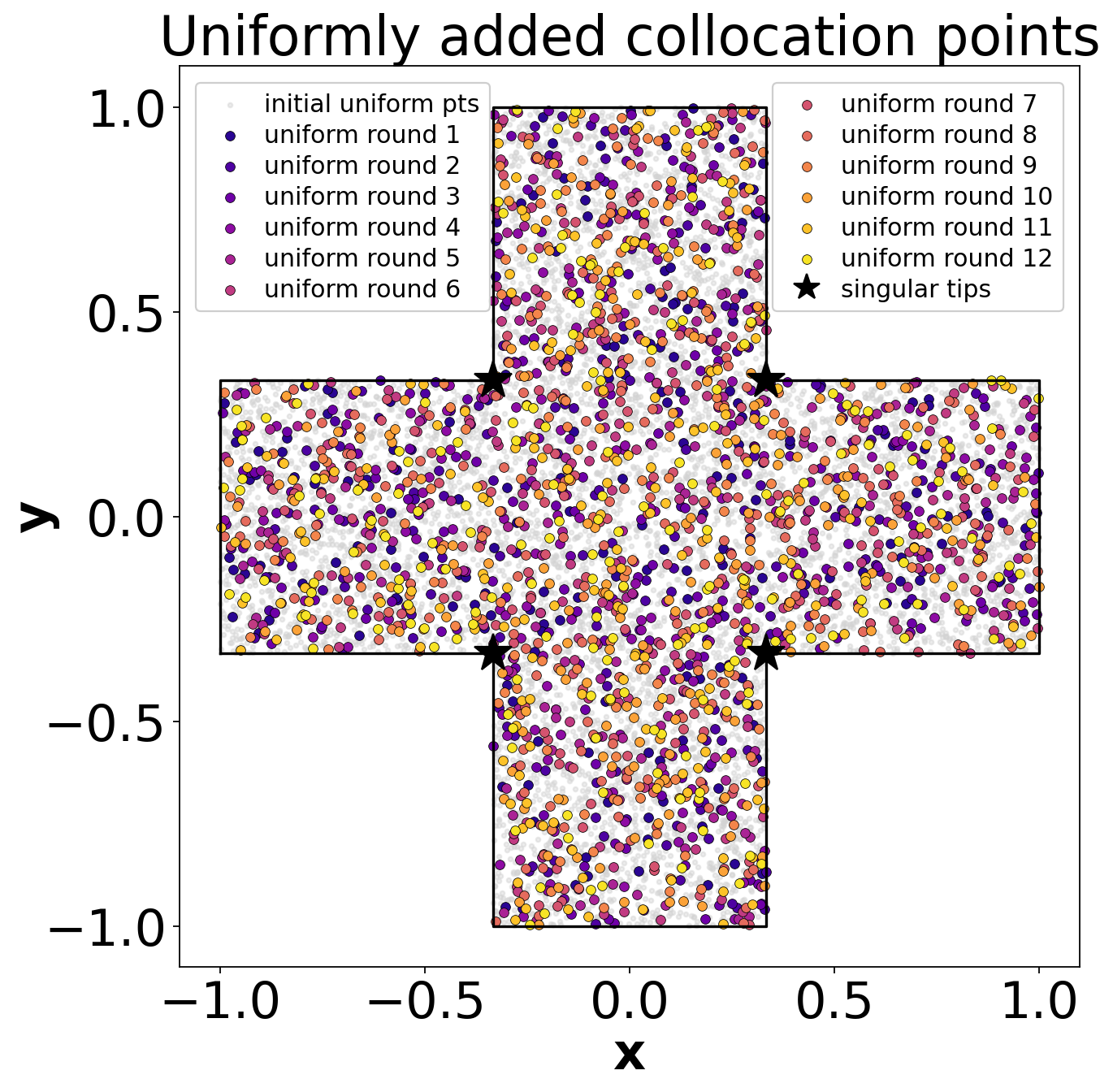}
\end{subfigure}
\hfill
\begin{subfigure}{0.32\textwidth}
    \textbf{F}\\[0.5mm]
    \includegraphics[width=\linewidth]{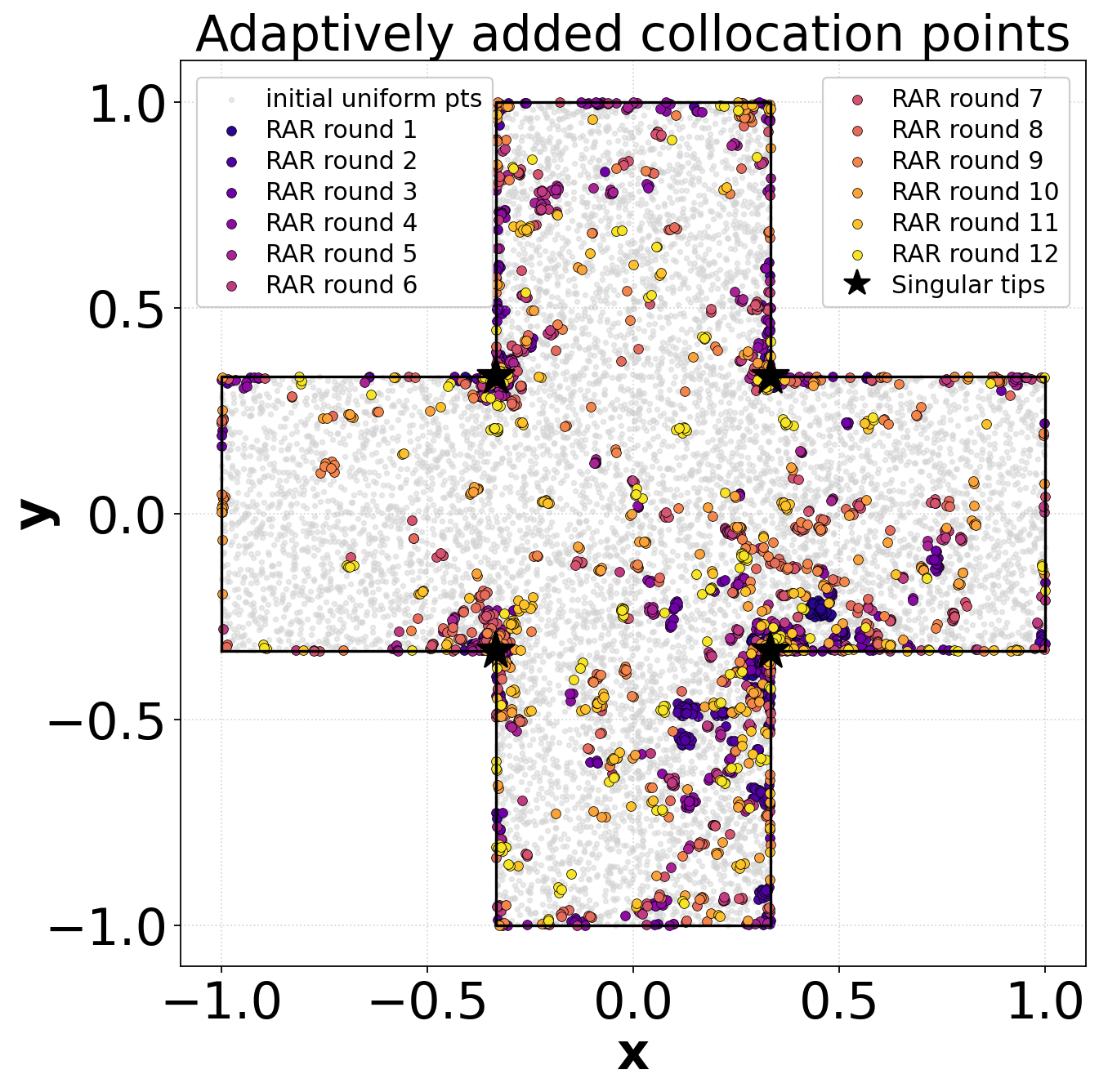}
\end{subfigure}

%---------------- Row 3 ----------------%
\begin{subfigure}{0.32\textwidth}
    \textbf{G}\\[0.5mm]
    \includegraphics[width=\linewidth]{plus_exact.png}
\end{subfigure}
\hfill
\begin{subfigure}{0.32\textwidth}
    \textbf{H}\\[0.5mm]
    \includegraphics[width=\linewidth]{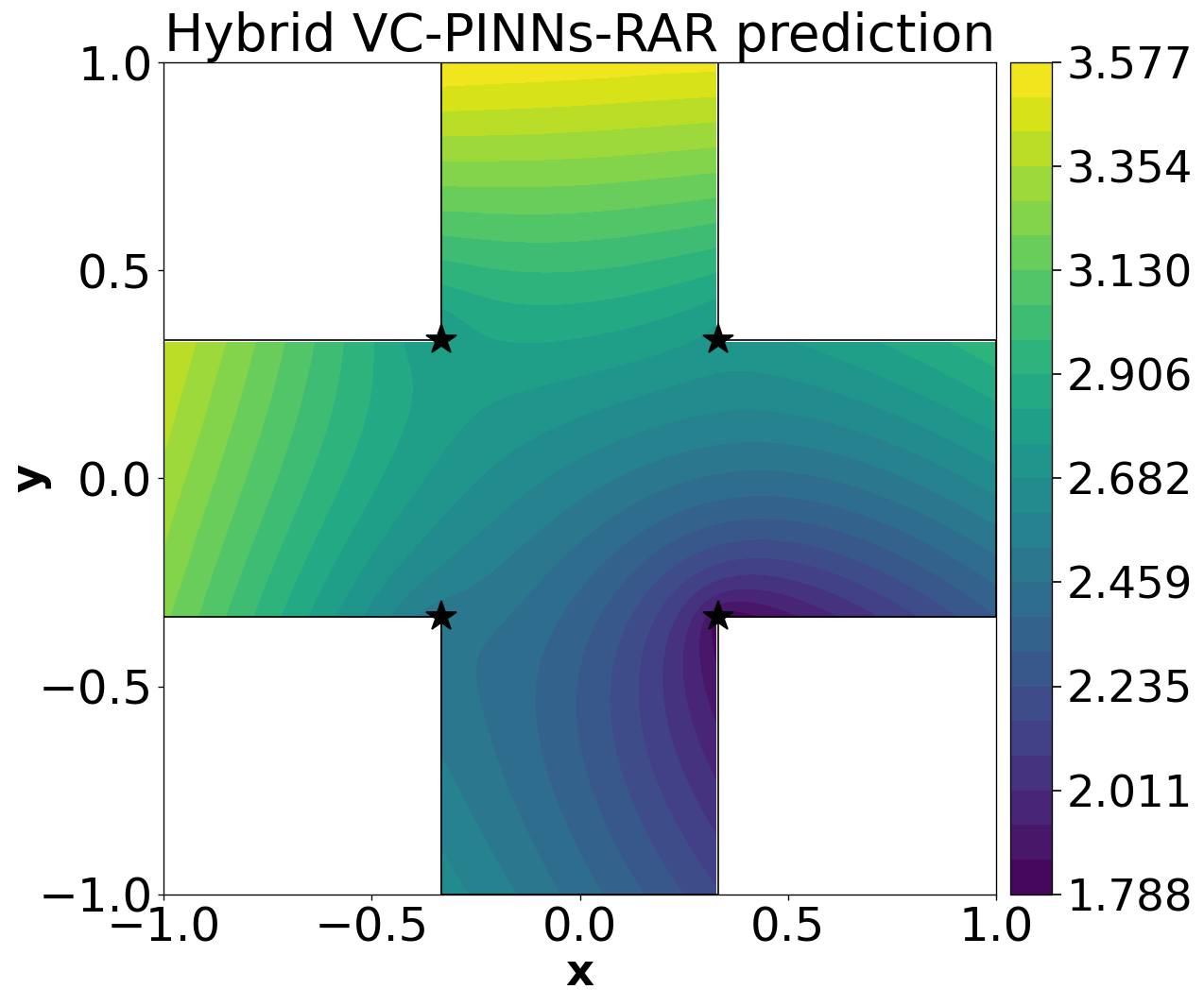}
\end{subfigure}
\hfill
\begin{subfigure}{0.32\textwidth}
    \textbf{I}\\[0.5mm]
    \includegraphics[width=\linewidth]{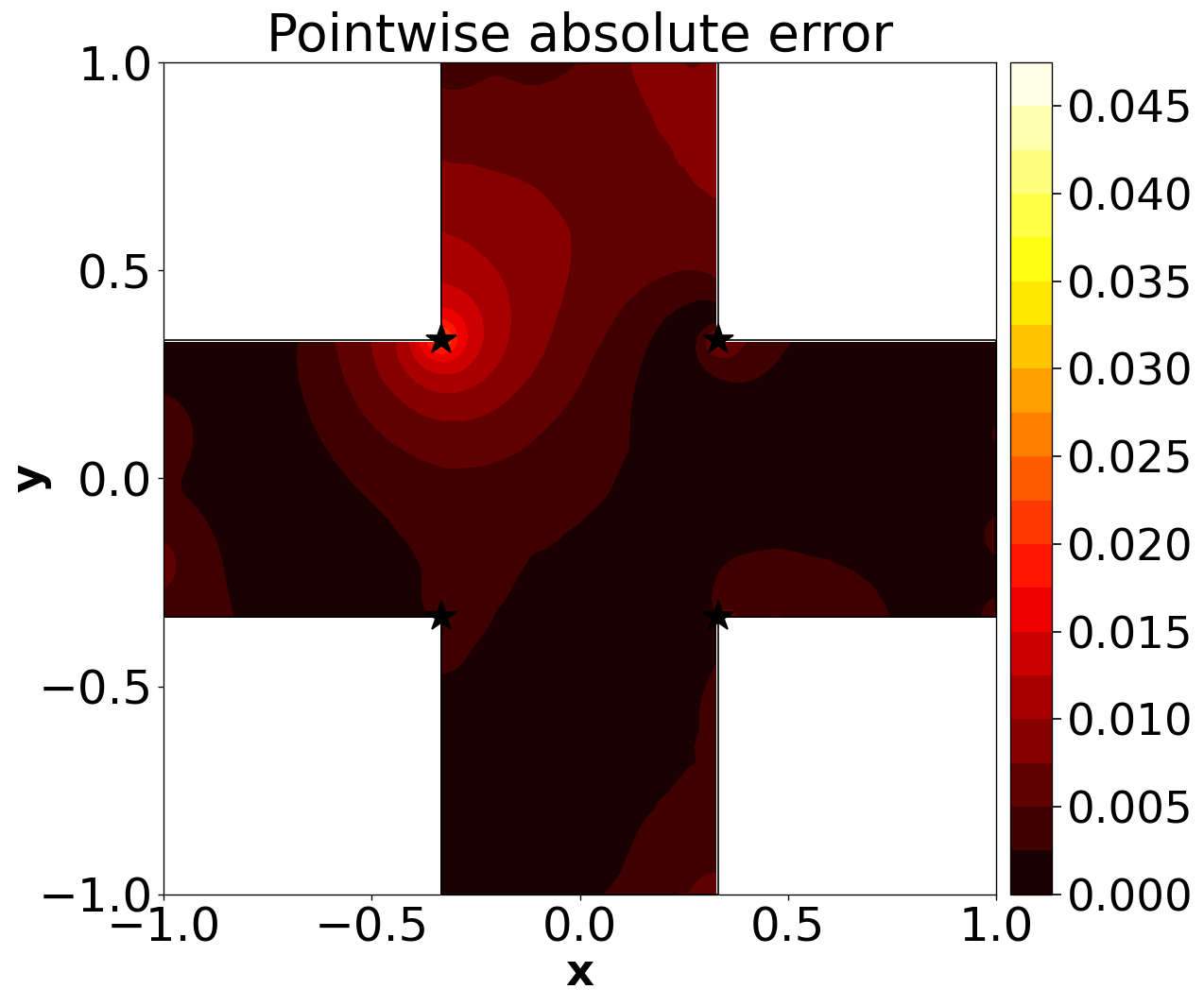}
\end{subfigure}

%---------------- Row 4 ----------------%
\begin{subfigure}{0.44\textwidth}
    \textbf{J}\\[0.5mm]
    \includegraphics[width=\linewidth]{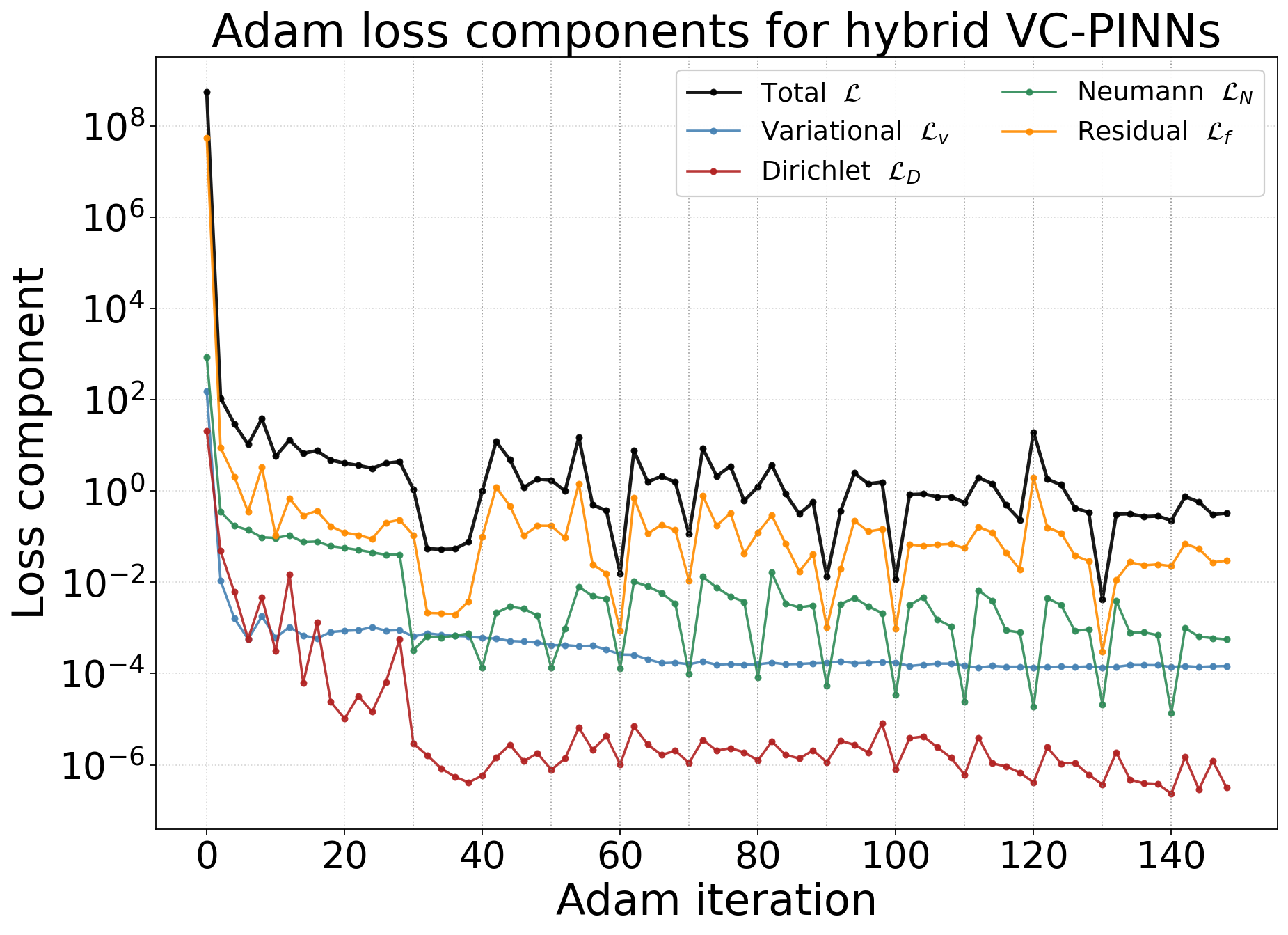}
\end{subfigure}
\hspace{0.04\textwidth}
\begin{subfigure}{0.44\textwidth}
    \textbf{K}\\[0.5mm]
    \includegraphics[width=\linewidth]{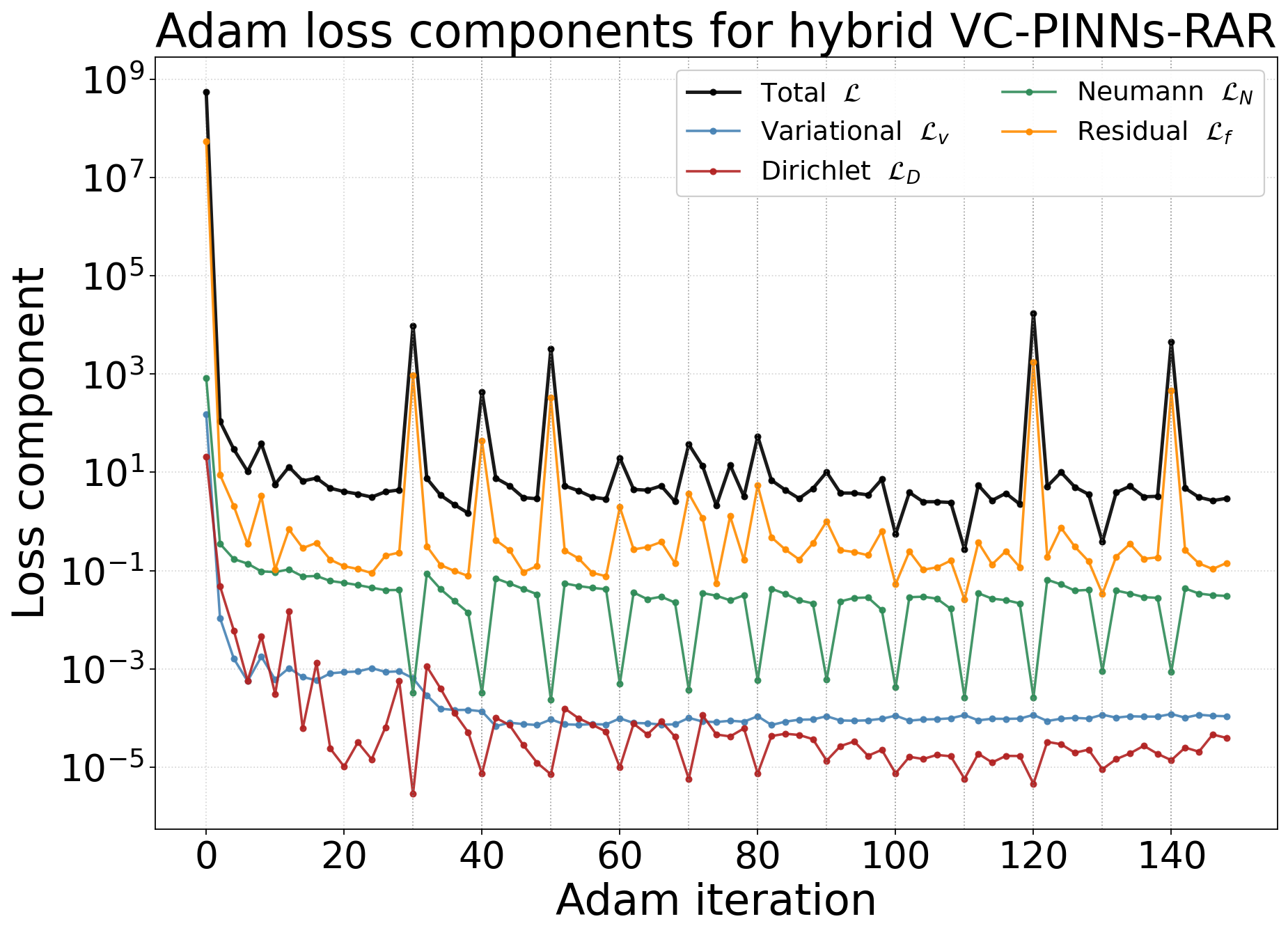}
\end{subfigure}
\vspace{-8pt}
\caption{Cross problem: Laplace equation on the cross-shaped domain. \textbf{(A--C)} Exact solution, VC-PINNs (uniform) prediction and pointwise absolute error. \textbf{(D-F)} Domain with boundary conditions and element partition, uniform collocation set and adaptive collocation set after twelve rounds, $|\mathcal{T}_{f}|=8000\to10{,}400$. \textbf{(G-I)} Exact solution, VC-PINNs-RAR prediction and pointwise absolute error; panels C and I have separate color scales. \textbf{(J-K)} Adam loss components for the uniform and the RAR arm.}\label{fig:PDE2_figures}
\end{figure}

%PDE3
\subsection{Nonlinear reaction-diffusion with four localized peaks}\label{reaction_diffusion_PDE4}
The preceding examples carry singularities whose locations are dictated by the geometry and are supplied to the network through input enrichment. This example removes both advantages: the solution structure is localized in the interior rather than at the boundary, its position is not encoded anywhere in the architecture and the equation is nonlinear. It therefore tests whether residual-driven refinement can locate and resolve sharp features that the trial space knows nothing about.

We consider the steady reaction-diffusion equation
\begin{equation}
-\varepsilon^{2}\Delta u + u^{3} - u = f \quad\text{in } \Omega=(-1,1)^{2},
\qquad \partial_{\bm{n}}u=0 \ \text{ on } \partial\Omega,
\label{eq:rd_pde}
\end{equation}
with $\varepsilon=0.01$, the homogeneous Neumann condition modeling a closed system with zero flux across the boundary. The source $f$ is obtained by substituting into~\eqref{eq:rd_pde} the exact solution
\begin{equation}
u=\sum_{k=1}^{4}A_k\exp\!\Big(-\frac{\|\bm{x}-\bar{\bm{x}}_{k}\|^{2}}{2\sigma_{\mathrm{p}}^{2}}\Big),
\qquad \sigma_{\mathrm{p}}=0.05,\quad (A_{1},A_{2},A_{3},A_{4})=(1,-1,1,-1),
\label{eq:rd_exact}
\end{equation}
with centers $\bar{\bm{x}}_{1}=(-0.6,0.4)$, $\bar{\bm{x}}_{2}=(0.5,-0.3)$, $\bar{\bm{x}}_{3}=(0.2,0.7)$ and $\bar{\bm{x}}_{4}=(-0.4,-0.6)$, one per quadrant. Since every center lies at least $0.3$ from $\partial\Omega$, the exact solution satisfies the homogeneous Neumann condition to within $10^{-6}$, which is negligible at the accuracy levels reported below. The problem is a steady state of the Allen-Cahn equation, each peak representing a localized phase domain embedded in an otherwise homogeneous medium; the solution is nearly zero over the bulk and rises to $\pm1$ only within narrow neighborhoods of the centers, where it carries very steep gradients.

Three features make this problem hard for uniform collocation. The four $\sigma_{\mathrm{p}}$-radius disks cover under $1\%$ of the domain, so a uniform collocation set places almost all of its points in the flat region where the residual is already small, while the few points near a peak are too sparse to resolve its gradient. Since the reaction term $u^{3}-u$ is bistable, the optimization landscape admits the spurious low-energy state $u\approx0$, toward which an under-resolved model is naturally drawn. Bistability under purely Neumann data also leaves the solution branch undetermined, so a small set of interior anchors (the center of each peak and four neighbors one peak width away, twenty points in all) is PINNs based to the values of the exact solution with weight $w_{D}=100$ to select the intended branch. The anchors localize the peaks yet do not resolve them: given the branch, an under-resolved peak still produces a large strong-form residual and RAR exploits that signal.

The square is partitioned into four rectangular elements, its quadrants, each carrying $K=8$ Legendre bubble functions per direction. The quadrature order is raised to $Q=80$ nodes per direction, giving eight to ten nodes across the $3\sigma_{\mathrm{p}}$ support of each peak; at the order $Q=24$ used for the singular benchmarks, the peaks would be invisible to the variational residual. Unlike in the preceding examples, the network input is not enriched: the trial space is a multiscale Fourier-feature network acting on $(x,y)$ alone, with four $\tanh$ hidden layers of width $200$, the last three with residual connections. The objective combines the variational residual, the Neumann penalty, the anchor term and the collocation term with weights $(w_v,w_D,w_N,w_f)=(1,100,50,5)$, using $200$ Neumann points per edge. The RAR arm begins from $10{,}000$ uniformly drawn collocation points and, in each of ten rounds, inserts the $300$ points with the largest strong-form residual from a candidate pool of $30{,}000$ uniformly drawn points. It reaches $|\mathcal{T}_{f}|=13{,}000$ after $140{,}000$ Adam iterations and eleven L-BFGS stages. Table~\ref{tab:PDE4_table} reports the errors.

\begin{table}[t]
\centering
\caption{Reaction-diffusion problem. All errors are reported in percent (\%) and bold marks the more accurate arm in each row. The RAR arm grows its collocation set from $10{,}000$ to $13{,}000$ points over ten rounds of $300$ insertions, with $140{,}000$ Adam iterations and eleven L-BFGS stages. The uniform arm is trained on a fixed set of uniformly drawn collocation points without insertion rounds. One training run per arm. Errors are evaluated on a uniform $401\times401$ grid on $[-1,1]^2$ and the $H^{1}$ semi-norm is computed with central differences on the same grid. As $\max|u|=1$, the relative $L^{\infty}$ error equals the absolute one.}
\label{tab:PDE4_table}
\begin{tabular}{lcc}
\toprule
Metric & VC-PINNs (uniform) & VC-PINNs-RAR \\
\midrule
Relative $L^{2}$ error           & $16.00$    & $\mathbf{1.015}$  \\
Relative $L^{\infty}$ error      & $26.50$    & $\mathbf{0.6986}$ \\
Relative $H^{1}$ semi-norm error & $64.46$    & $\mathbf{4.738}$  \\
\bottomrule
\end{tabular}
\end{table}

The diagnostic result is the gap between the moderate $L^{2}$ error of the uniform arm and its very large $H^{1}$ error. That network locates all four peaks and recovers their approximate amplitudes (Fig.~\ref{fig:PDE4_figures}B); however, it does not recover their gradients, so its pointwise error concentrates in the peak cores (Fig.~\ref{fig:PDE4_figures}C, up to about $0.27$) while the smooth background is accurate. Redirecting the same final number of collocation points with the residual indicator reduces the $L^{2}$, $H^{1}$ and $L^{\infty}$ errors by factors of roughly $16$, $14$ and $38$, respectively (table~\ref{tab:PDE4_table}). The remaining error collapses to a low-amplitude pattern spread over the domain rather than a set of cores (Fig.~\ref{fig:PDE4_figures}I, up to about $0.007$). Successive rounds cluster points at the four centers (Fig.~\ref{fig:PDE4_figures}F). That the largest gain appears in the gradient measure confirms that the refinement resolves the peak structure itself rather than merely lowering the global error and since both runs end at $13{,}000$ points the difference is attributable to placement alone.

\begin{figure}[!htbp]
\centering
%---------------- Row 1 ----------------%
\begin{subfigure}{0.32\textwidth}
    \textbf{A}\\[0.5mm]
    \includegraphics[width=\linewidth]{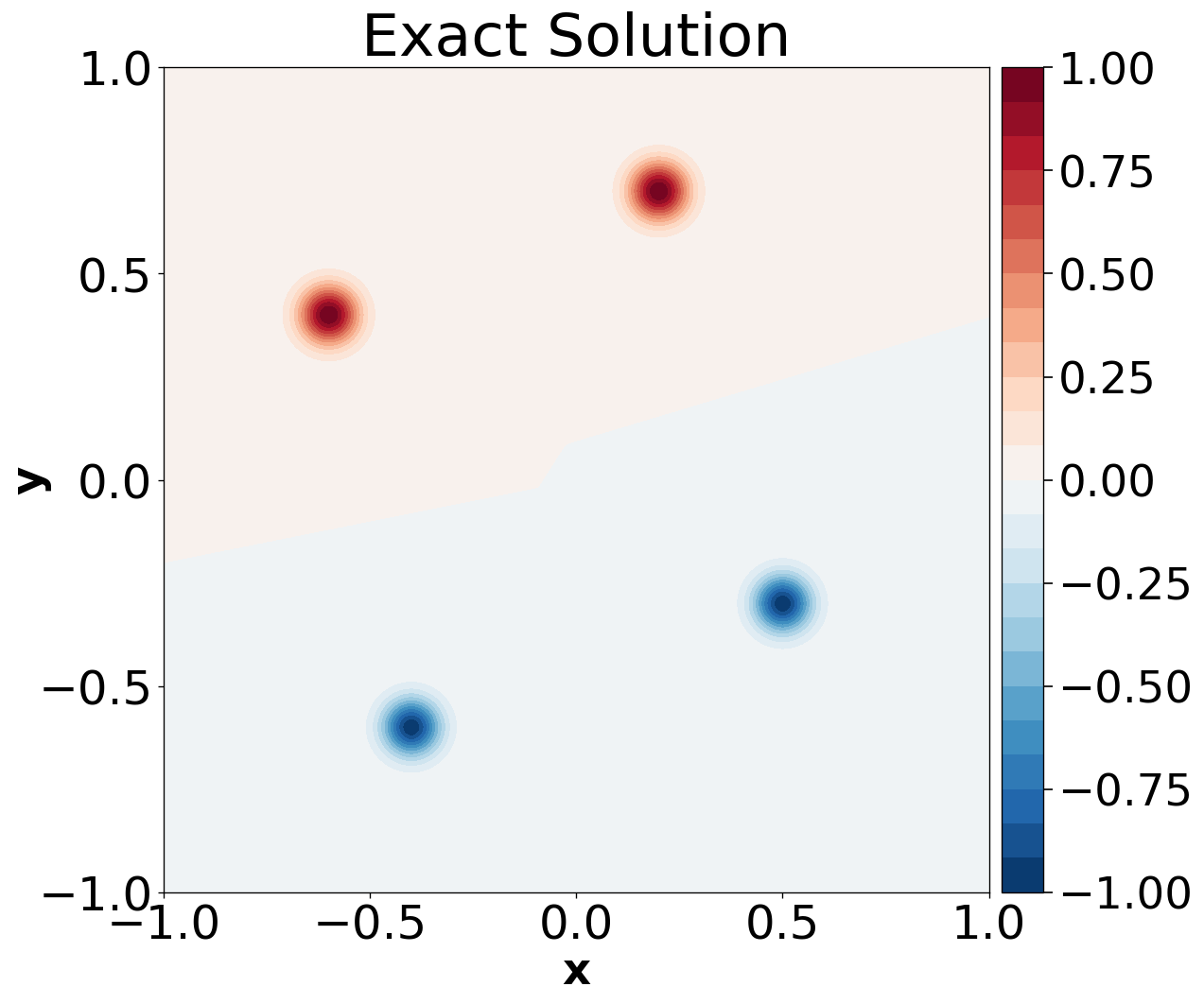}
\end{subfigure}
\hfill
\begin{subfigure}{0.32\textwidth}
    \textbf{B}\\[0.5mm]
    \includegraphics[width=\linewidth]{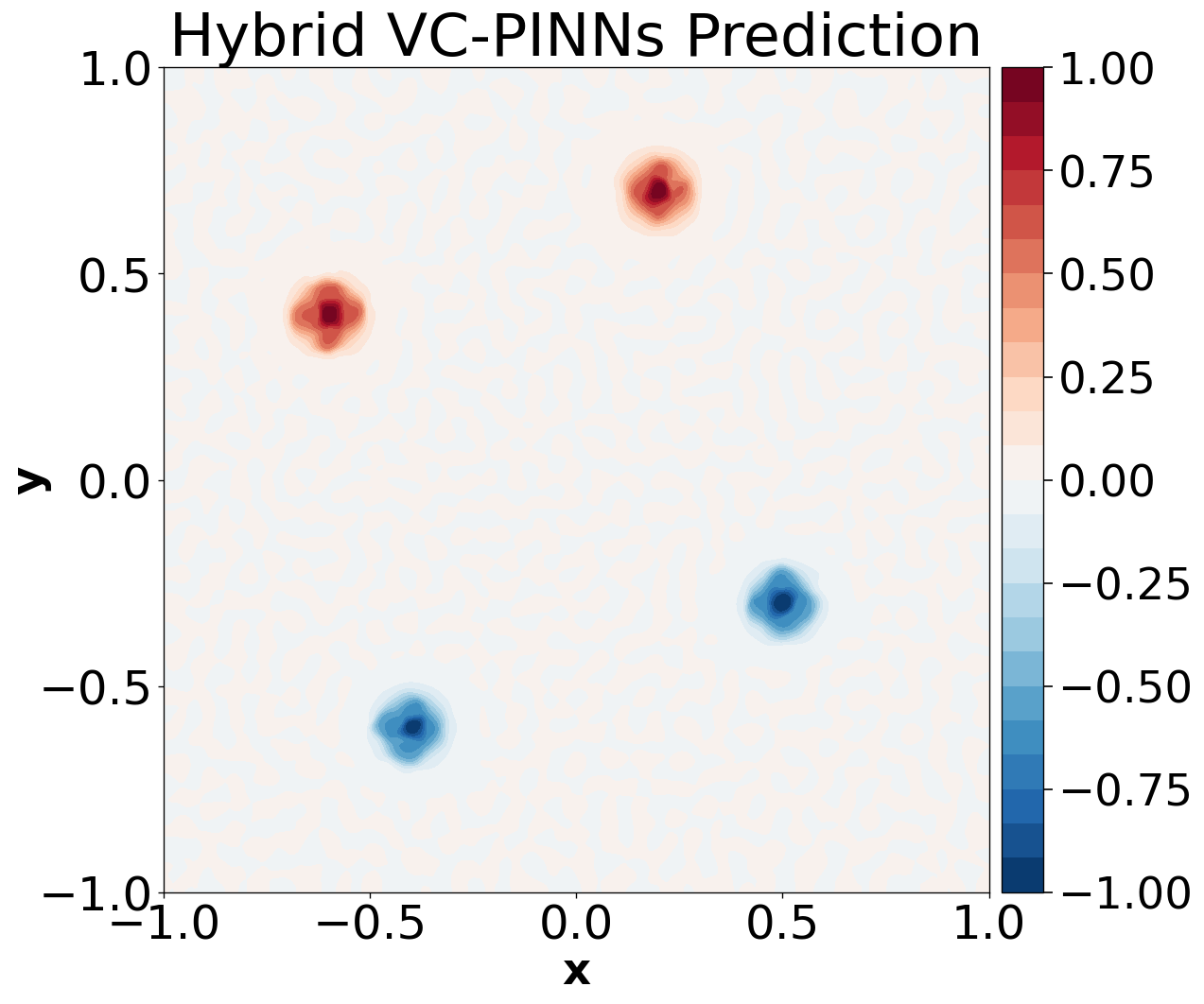}
\end{subfigure}
\hfill
\begin{subfigure}{0.32\textwidth}
    \textbf{C}\\[0.5mm]
    \includegraphics[width=\linewidth]{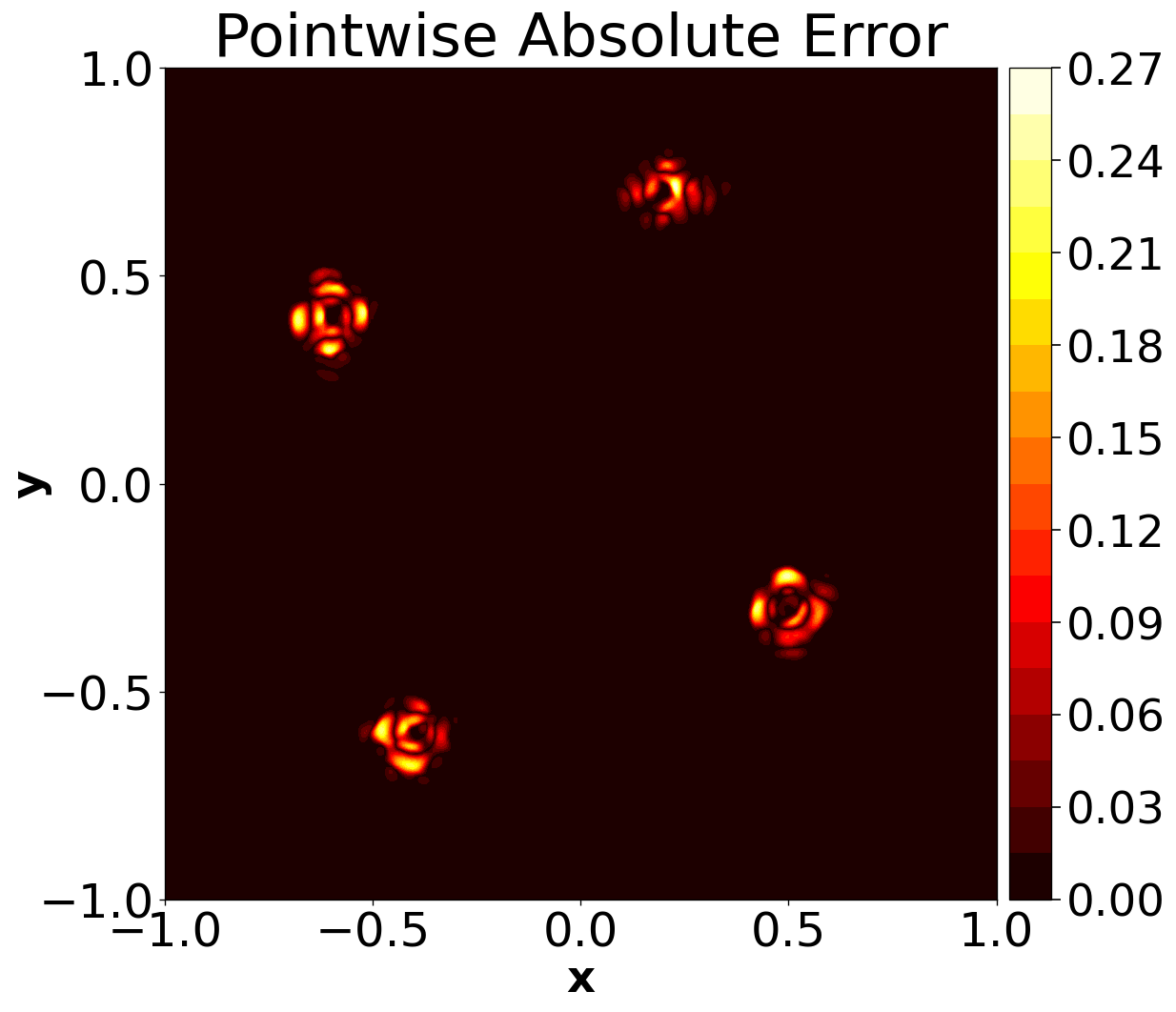}
\end{subfigure}

%---------------- Row 2 ----------------%
\begin{subfigure}{0.32\textwidth}
    \textbf{D}\\[0.5mm]
    \includegraphics[width=\linewidth]{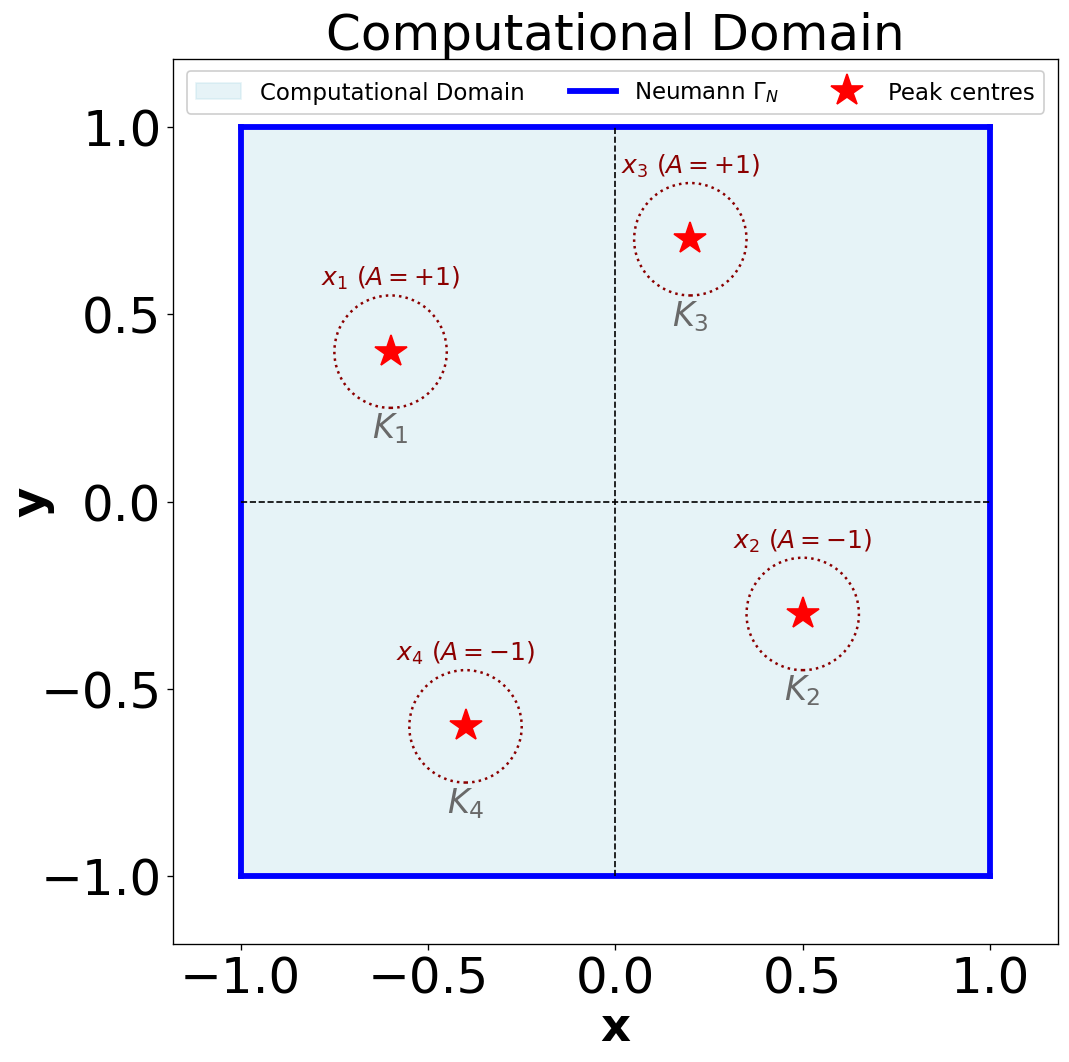}
\end{subfigure}
\hfill
\begin{subfigure}{0.32\textwidth}
    \textbf{E}\\[0.5mm]
    \includegraphics[width=\linewidth]{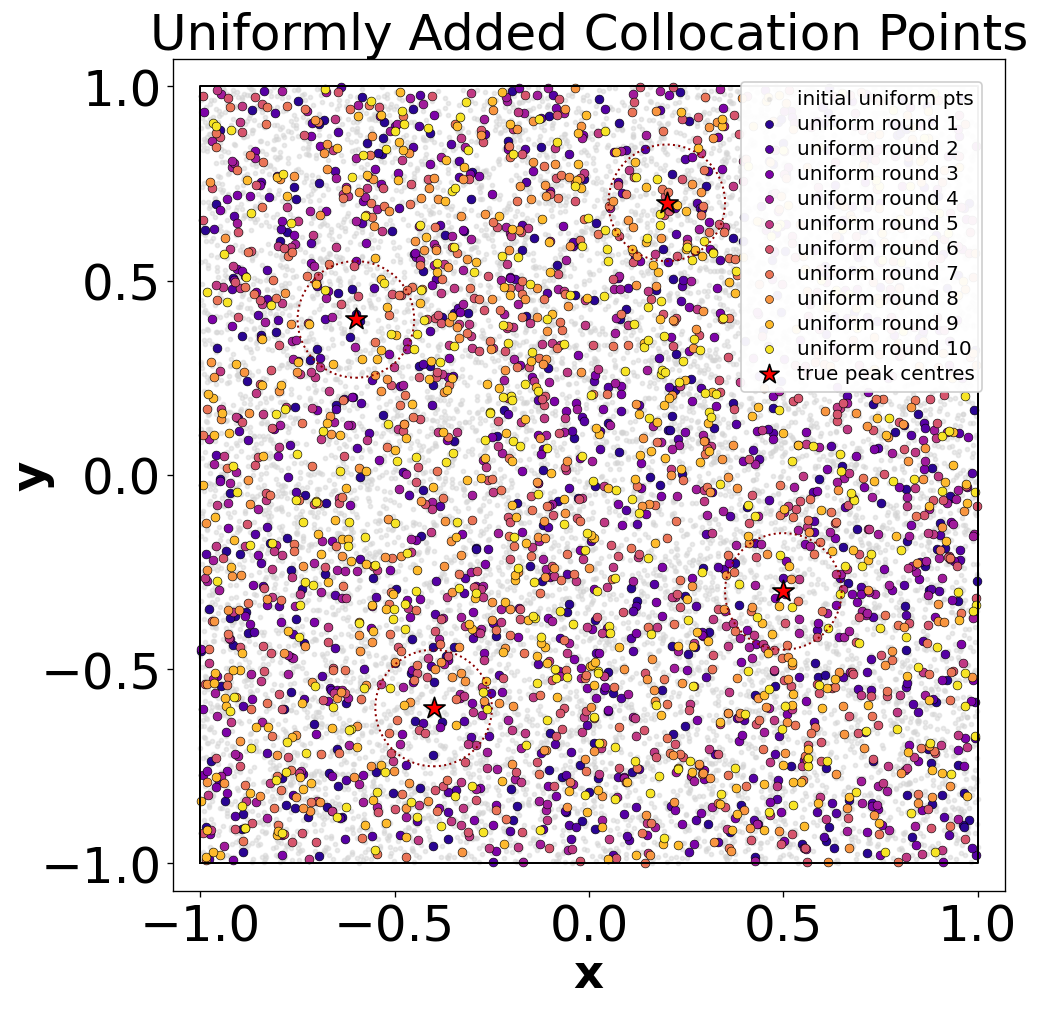}
\end{subfigure}
\hfill
\begin{subfigure}{0.32\textwidth}
    \textbf{F}\\[0.5mm]
    \includegraphics[width=\linewidth]{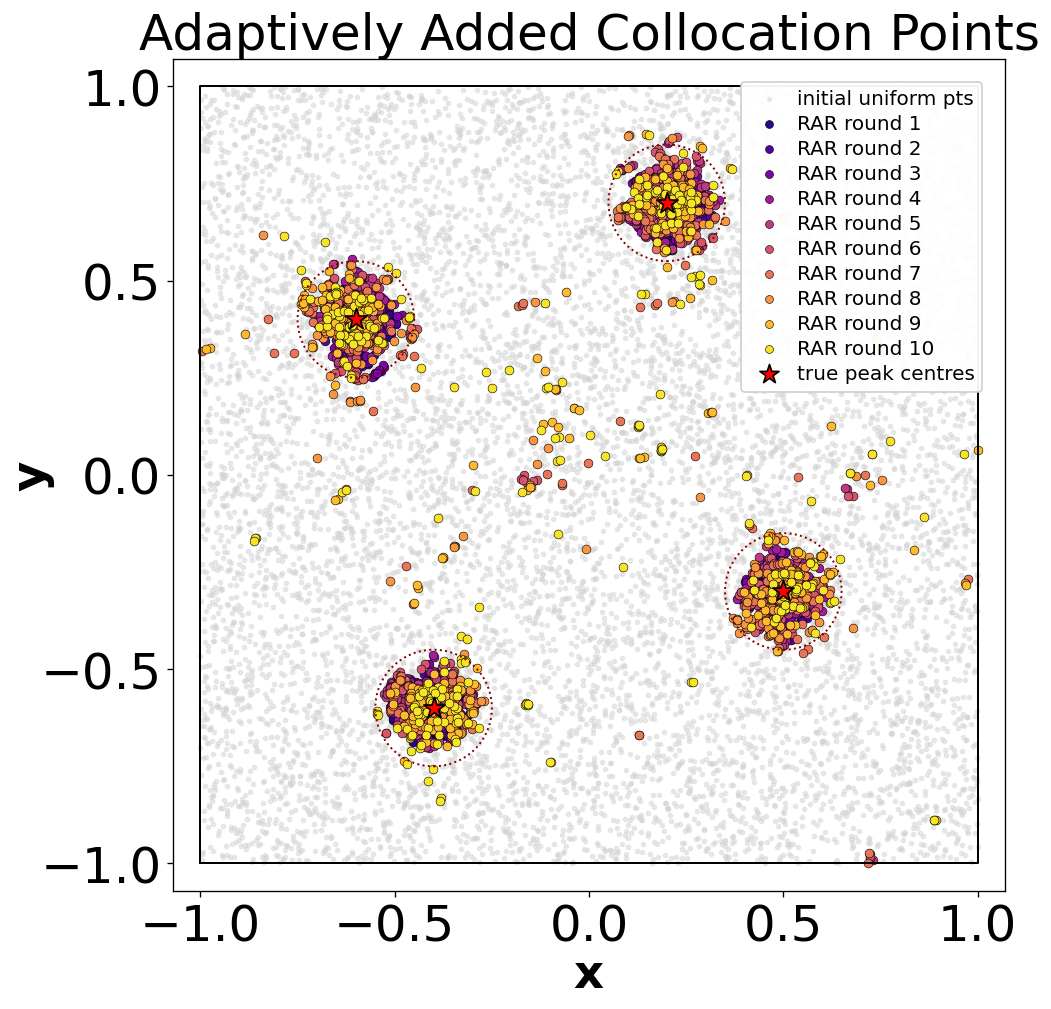}
\end{subfigure}

%---------------- Row 3 ----------------%
\begin{subfigure}{0.32\textwidth}
    \textbf{G}\\[0.5mm]
    \includegraphics[width=\linewidth]{non_linear_exact_solution.png}
\end{subfigure}
\hfill
\begin{subfigure}{0.32\textwidth}
    \textbf{H}\\[0.5mm]
    \includegraphics[width=\linewidth]{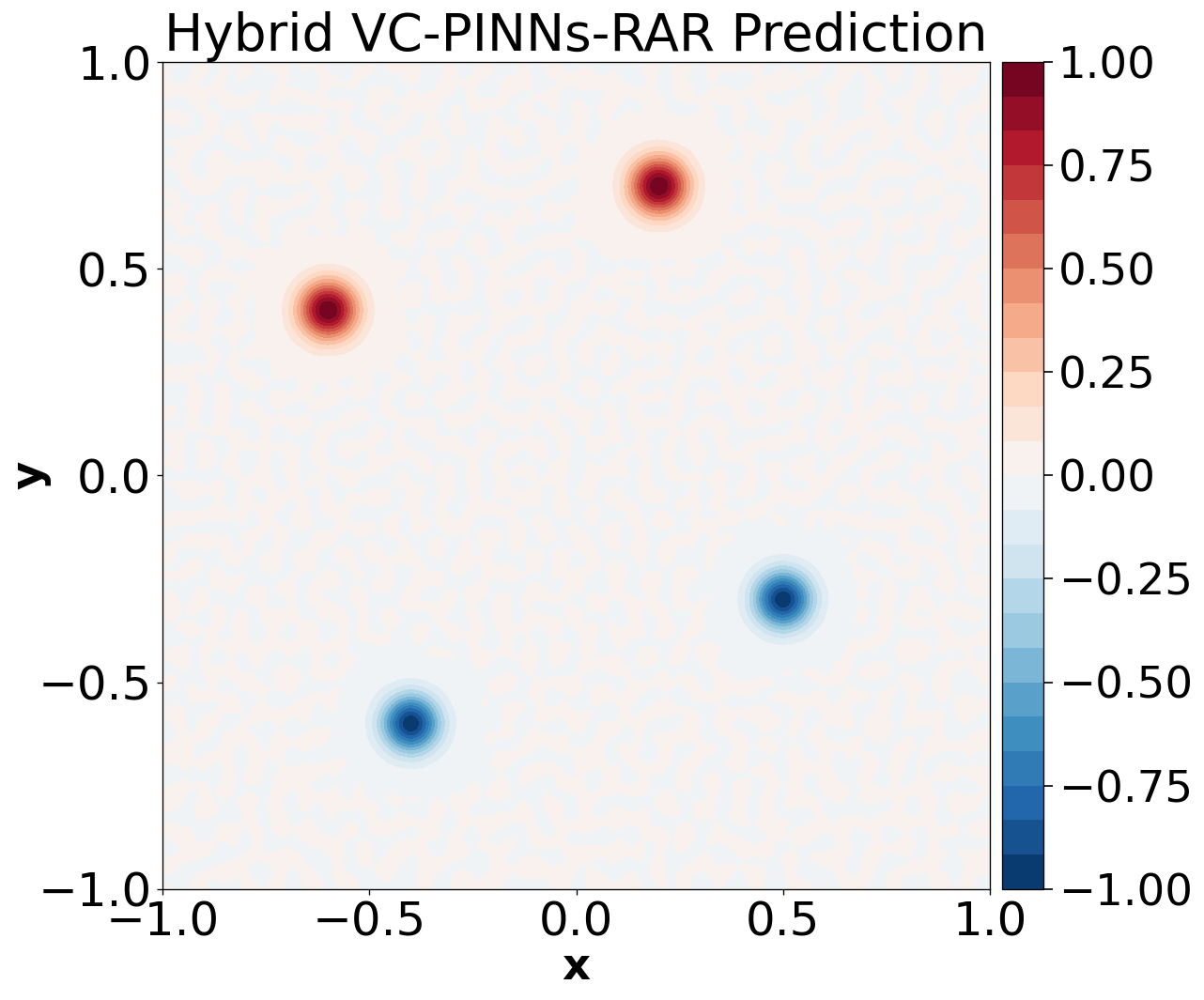}
\end{subfigure}
\hfill
\begin{subfigure}{0.32\textwidth}
    \textbf{I}\\[0.5mm]
    \includegraphics[width=\linewidth]{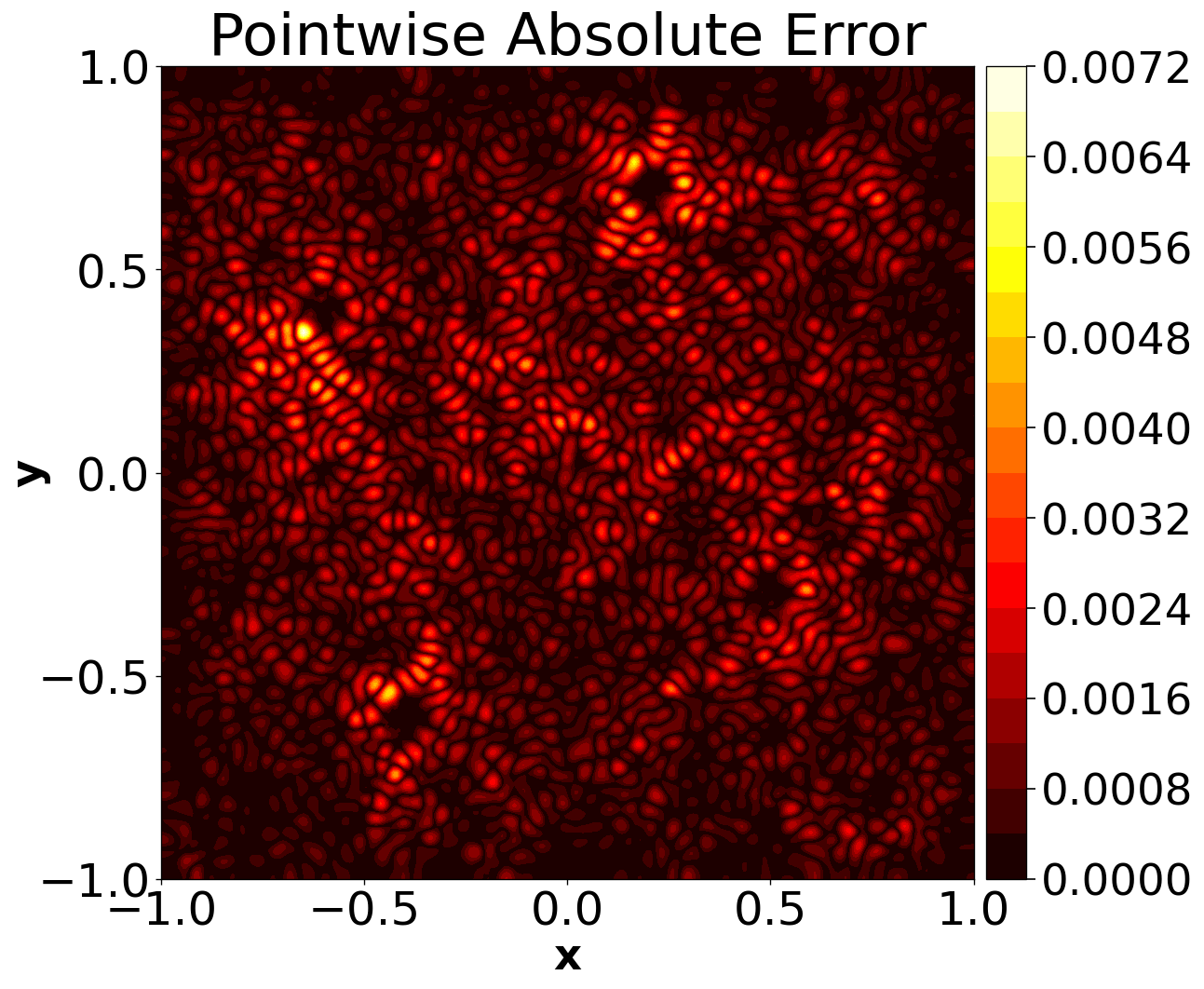}
\end{subfigure}

%---------------- Row 4 ----------------%
\begin{subfigure}{0.44\textwidth}
    \textbf{J}\\[0.5mm]
    \includegraphics[width=\linewidth]{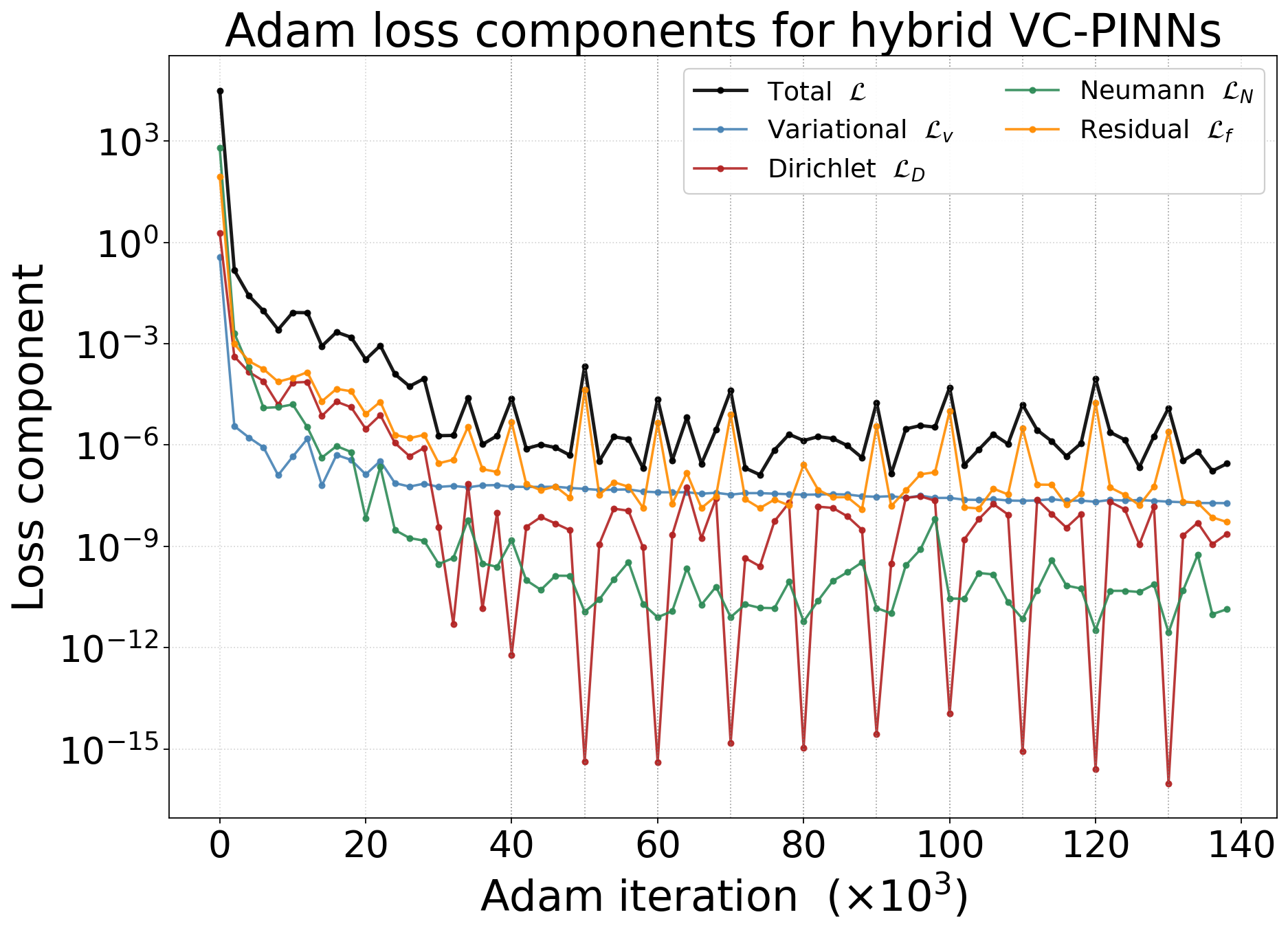}
\end{subfigure}
\hspace{0.04\textwidth}
\begin{subfigure}{0.44\textwidth}
    \textbf{K}\\[0.5mm]
    \includegraphics[width=\linewidth]{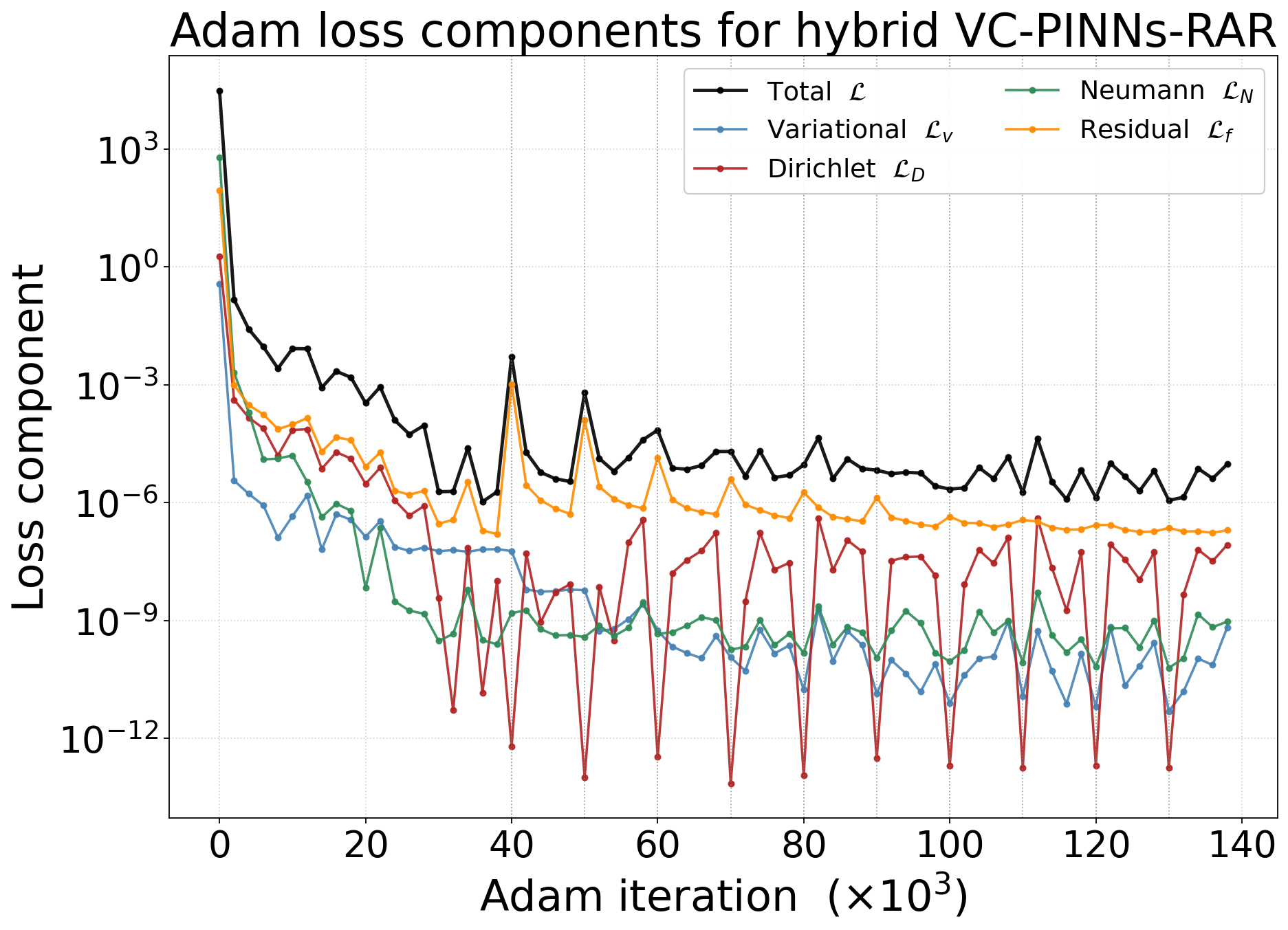}
\end{subfigure}
\caption{Reaction-diffusion problem with four localized peaks. \textbf{(A-C)} Exact solution, VC-PINNs (uniform) prediction and pointwise absolute error. \textbf{(D-F)} Domain with the four peak centers and element partition, uniform collocation set and adaptive collocation set after ten rounds, $|\mathcal{T}_{f}|=10{,}000\to13{,}000$. \textbf{(G-I)} Exact solution, VC-PINNs-RAR prediction and pointwise absolute error. \textbf{(J-K)} Adam loss components for the uniform and the RAR arm.}\label{fig:PDE4_figures}
\end{figure}

%PDE5
\subsection{Monotone semilinear reaction-diffusion on an L-shaped domain}\label{semilinear_PDE5}
The preceding examples each carry one kind of singular feature. This one carries two of different dimension at the same time and admits no closed from solution, so it tests whether a single scalar indicator can distribute refinement between a curve and a point without any manufactured solution to guide it.

We solve the singularly perturbed semilinear problem
\begin{equation}
-\varepsilon^{2}\Delta u+u^{3}=1 \ \text{ in } \Omega, \qquad
u=0 \ \text{ on } \Gamma_{D}, \qquad
\partial_{\bm{n}}u=0 \ \text{ on } \Gamma_{N},
\label{eq:semilinear_pde}
\end{equation}
with $\varepsilon=0.05$ on $\Omega=(-1,1)^{2}\setminus([0,1]\times[0,1])$, where $\Gamma_{D}=\{(x,0):0\le x\le1\}\cup\{(0,y):0\le y\le1\}$ are the two re-entrant edges and $\Gamma_{N}$ the outer boundary. The right hand side is the constant $f\equiv1$. Nothing is manufactured, the network is given only $f$, the boundary data and the geometry.

Writing $A(u)=-\varepsilon^{2}\Delta u+u^{3}$ on $V=\{v\in H^{1}(\Omega):v|_{\Gamma_D}=0\}$, the pointwise identity $(u^{3}-v^{3})(u-v)=(u-v)^{2}(u^{2}+uv+v^{2})\ge0$ gives $\langle A(u)-A(v),u-v\rangle\ge\varepsilon^{2}|u-v|_{H^{1}}^{2}$, so $A$ is strictly monotone and being coercive and hemi continuous it admits a unique weak solution by the Browder-Minty theorem. No branch has to be selected, so this example uses no interior anchors, the refinement locates the features from the residual alone.

The solution has two distinct sharp features. Away from $\Gamma_{D}$, the diffusion term is negligible and $u$ approaches the root $u\equiv1$ of $u^{3}=1$. Along $\Gamma_{D}$, the homogeneous data forces a boundary layer of width $\frac{\varepsilon}{\sqrt{3}}\approx0.029$, a codimension one feature of total length $2$. At the origin, the interior angle $\tfrac{3\pi}{2}$ produces the corner behavior $u\sim r^{2/3}$, a point feature. A scalar indicator must allocate points between a curve and a point using one number per candidate. As in section~\ref{plus_shaped_PDE2}, the leading corner mode is supplied by appending $r^{2/3}$ to the input, so the network acts on $(x,y,r^{2/3})$, the layer receives no analogous treatment and must be learned.

Errors are measured against a quadratic ($\mathbb{P}_{2}$) finite element reference computed on a mesh graded toward $\Gamma_{D}$ and the corner. It is obtained by Newton's method with continuation in $\varepsilon$ from $0.4$ down to $0.05$, extrapolated across refinement levels and sampled on a $601\times601$ grid. The reference is used only post hoc and never enters the loss, grid points within $10^{-3}$ of the corner are excluded from the $H^{1}$ semi-norm, where the reference gradient is itself unbounded.

A second check bypasses the reference entirely. Along a straight portion of $\Gamma_{D}$ and away from the corner, the solution depends only on the wall distance $\zeta$ and \eqref{eq:semilinear_pde} reduces to $-\varepsilon^{2}u''+u^{3}=1$ with $u(0)=0$, $u(\infty)=1$. Multiplying by $u'$ and integrating once gives the first integral
\begin{equation}
\tfrac{1}{2}\varepsilon^{2}(u')^{2}=\tfrac14\big(u^{4}-4u+3\big)
=\tfrac14(1-u)^{2}\big(u^{2}+2u+3\big),
\label{eq:first_integral}
\end{equation}
which gives the exact wall gradient $u'(0)=\sqrt{3/2}/\varepsilon=24.4949\ldots$ and the tail $1-u\sim1.4558\,e^{-\sqrt{3}\zeta/\varepsilon}$. This profile solves the two dimensional equation exactly in the strip, so away from the corner both the reference and the network must reproduce it.

Both arms use a multiscale Fourier-feature network with six residual hidden layers of width $128$ and Fourier scales $\{\sigma_{1},\sigma_{2},\sigma_{3}\}=\{1,10,30\}$, the L-shaped domain is tiled by three rectangular elements (unit squares), each carrying $K=8$ Legendre bubble functions and $Q=48$ Gauss-Legendre nodes per direction. The loss weights are $(w_v,w_D,w_N,w_f)=(1,50,50,10)$, with $600$ Dirichlet and $1000$ Neumann points. Both arms begin from $4000$ collocation points, insert $300$ per round over twelve rounds to $|\mathcal{T}_{f}|=7600$ and train for $150{,}000$ Adam iterations and thirteen L-BFGS stages, drawing candidates from the same admissible region $\{r\ge2\times10^{-3}\}$. They differ only in selection, the uniform arm draws its $300$ points at random from that region, whereas the RAR arm scores a candidate pool of $60{,}000$ points by $\rho=|-\varepsilon^{2}\Delta\hat{u}+\hat{u}^{3}-1|$ and retains the largest $300$.

\begin{table}[t]
\centering
\caption{Semilinear problem. All errors are reported in percent (\%) and bold marks the more accurate arm in each row. Both arms use the same collocation schedule ($|\mathcal{T}_{f}|=4000\to7600$), $150{,}000$ Adam iterations and thirteen L-BFGS stages, with no interior anchors; one training run per arm. Errors are measured against the finite element reference on its $601\times601$ grid and the $H^{1}$ semi-norm excludes grid points within $10^{-3}$ of the corner. As $0\le u\le1$ with $\max u=1$, the relative $L^{\infty}$ error equals the absolute one.}
\label{tab:PDE_5_table}
\begin{tabular}{lcc}
\toprule
Metric & VC-PINNs (uniform) & VC-PINNs-RAR \\
\midrule
Relative $L^{2}$ error           & $2.403$ & $\mathbf{0.1904}$ \\
Relative $L^{\infty}$ error      & $55.85$ & $\mathbf{3.546}$ \\
Relative $H^{1}$ semi-norm error & $39.00$ & $\mathbf{3.366}$ \\
\bottomrule
\end{tabular}
\end{table}

The spread between the three error measures of the uniform arm is diagnostic rather than incidental. As $u\approx1$ over most of $\Omega$ and departs from $1$ only in a strip of width $O(\varepsilon)$, a network that reproduces the constant state yet fails to resolve the layer (Fig.~\ref{fig:PDE5}B and C) already attains a small $L^{2}$ error, that norm does not weight a set of measure $O(\varepsilon)$, whereas the $H^{1}$ semi-norm does and it reports a relative gradient error of $39.00\%$. Over twelve rounds the uniform arm reduces its relative $L^{2}$ error from $4.382\%$ to $2.403\%$ and its $H^{1}$ semi-norm error from $50.35\%$ to $39.00\%$ however, only $135$ of its $3600$ inserted points lie within $0.05$ of $\Gamma_{D}$ (Fig.~\ref{fig:PDE5}E), close to the $3.4\%$ share of the domain that this strip occupies, residual-based insertion places $909$ points there.

The layer check based on~\eqref{eq:first_integral} is more direct than any norm. At the stations $x=0.3$, $0.5$ and $0.7$ on $\{(x,0):0<x<1\}$, the uniform arm returns wall gradients of $12.62$, $16.75$ and $15.03$ against the exact $24.4949$, too low by $48.5\%$, $31.6\%$ and $38.6\%$ and its wall normal profile deviates from the exact one by up to $25.93\%$. The RAR arm returns $24.50$, $24.51$ and $24.47$, within $0.12\%$ of the exact value and its profile deviates by less than $0.42\%$. With the same point and iteration budget, it also reduces the relative $L^{2}$, $L^{\infty}$ and $H^{1}$ semi-norm errors by factors of roughly $13$, $16$ and $12$ (table~\ref{tab:PDE_5_table}, Fig.~\ref{fig:PDE5}C and I). Since the layer check does not use the finite element reference, it is evidence that the network has resolved the layer rather than matched a particular discretization.

The point distribution shows how one scalar handles two feature types. Early rounds are dominated by the layer, inserted points forming a dense band along both segments of $\Gamma_{D}$ together with a tight cluster at the origin (Fig.~\ref{fig:PDE5}F), later rounds spread into the bulk once the layer residual has fallen to the background level. This is the amplitude ordered behavior of section~\ref{plus_shaped_PDE2}, here operating across features of different dimension. Inserted points affect training only since the objective retains a collocation term alongside the variational one, in a purely variational formulation, points ranked by $|R(\hat{u})|$ would have nothing to act on and refinement would require an element level estimator whose spatial resolution is coarser than the layer itself.

\begin{figure}[!htbp]
\centering
%---------------- Row 1 ----------------%
\begin{subfigure}{0.32\textwidth}
    \textbf{A}\\[0.5mm]
    \includegraphics[width=\linewidth]{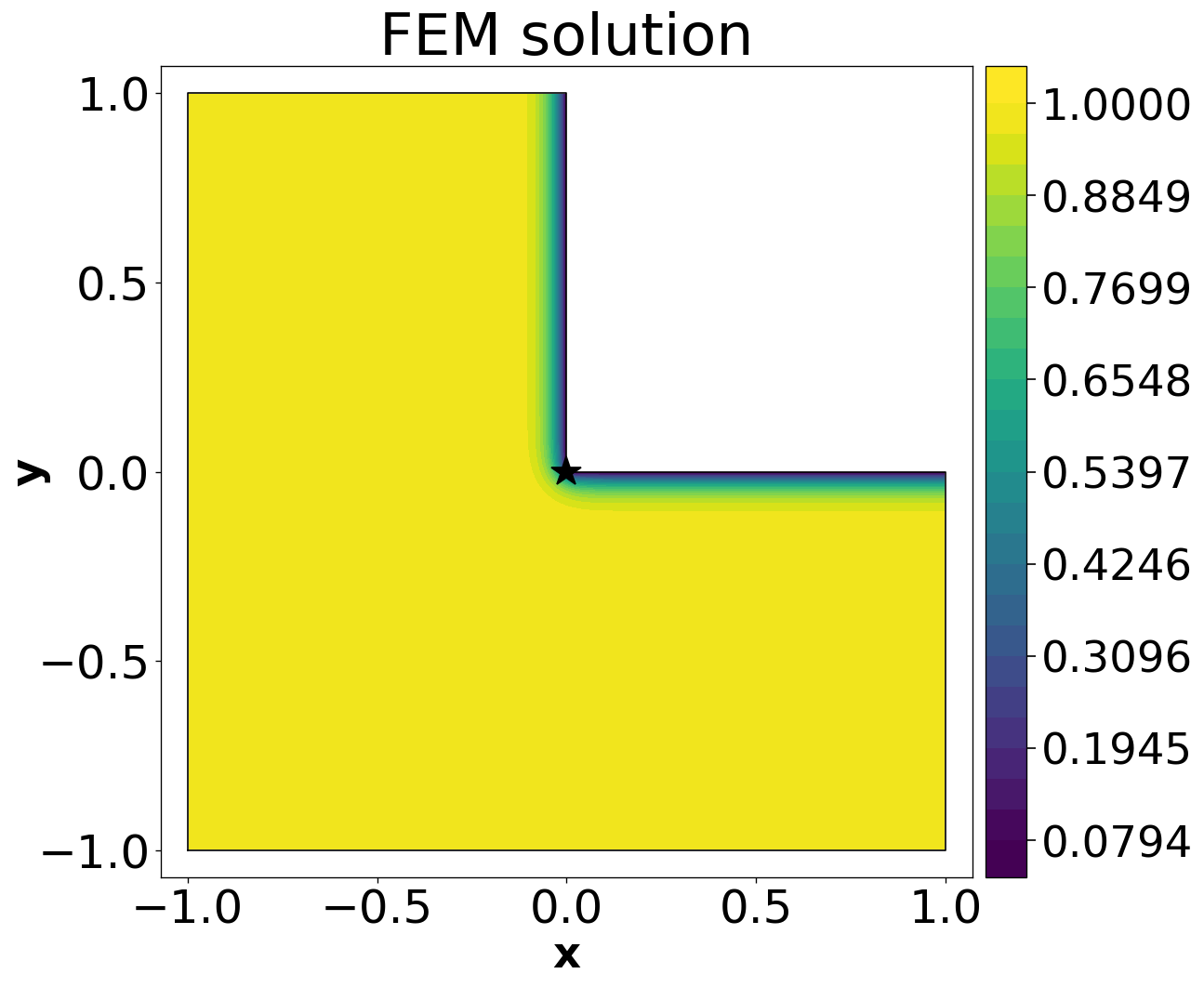}
\end{subfigure}
\hfill
\begin{subfigure}{0.32\textwidth}
    \textbf{B}\\[0.5mm]
    \includegraphics[width=\linewidth]{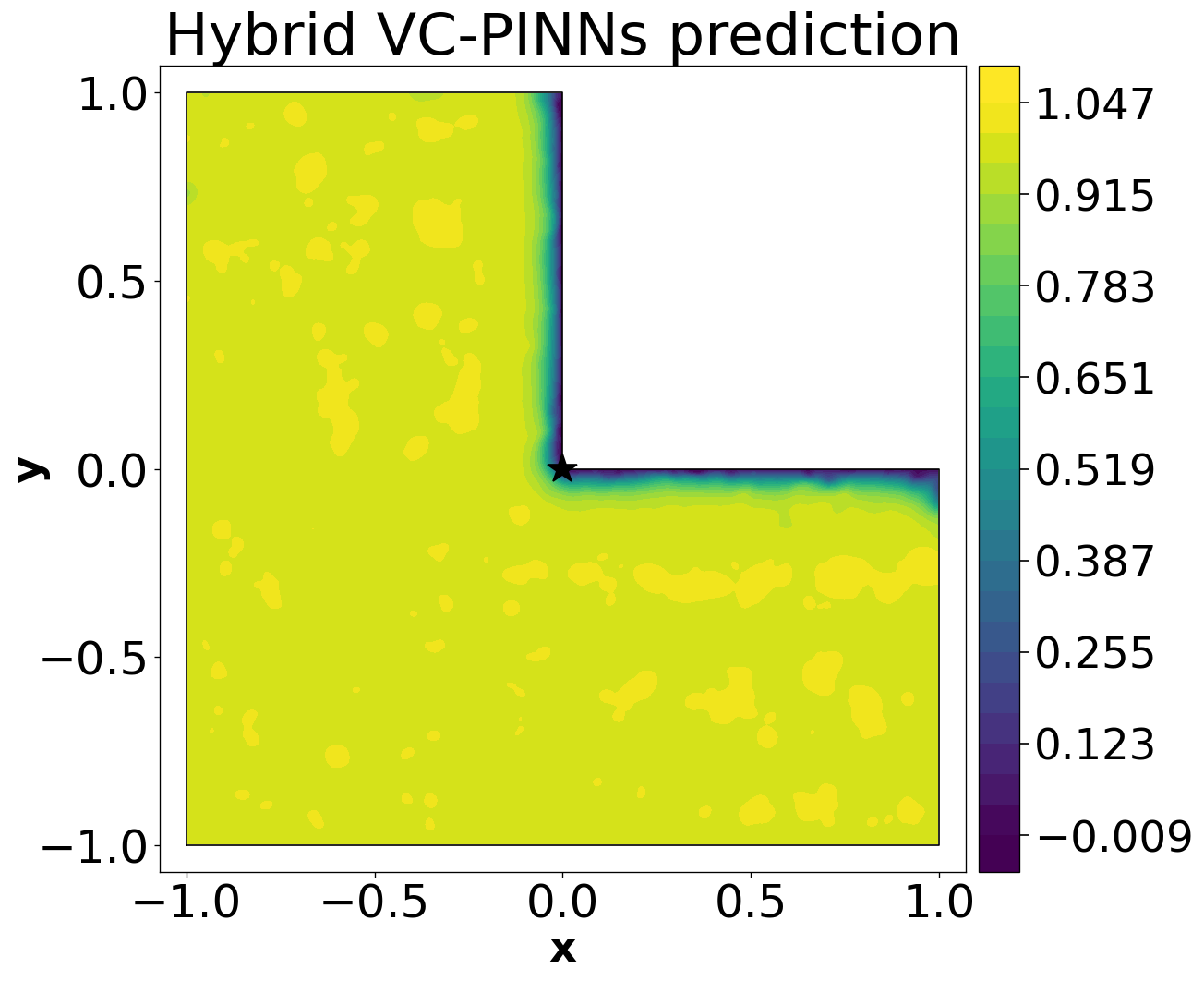}
\end{subfigure}
\hfill
\begin{subfigure}{0.32\textwidth}
    \textbf{C}\\[0.5mm]
    \includegraphics[width=\linewidth]{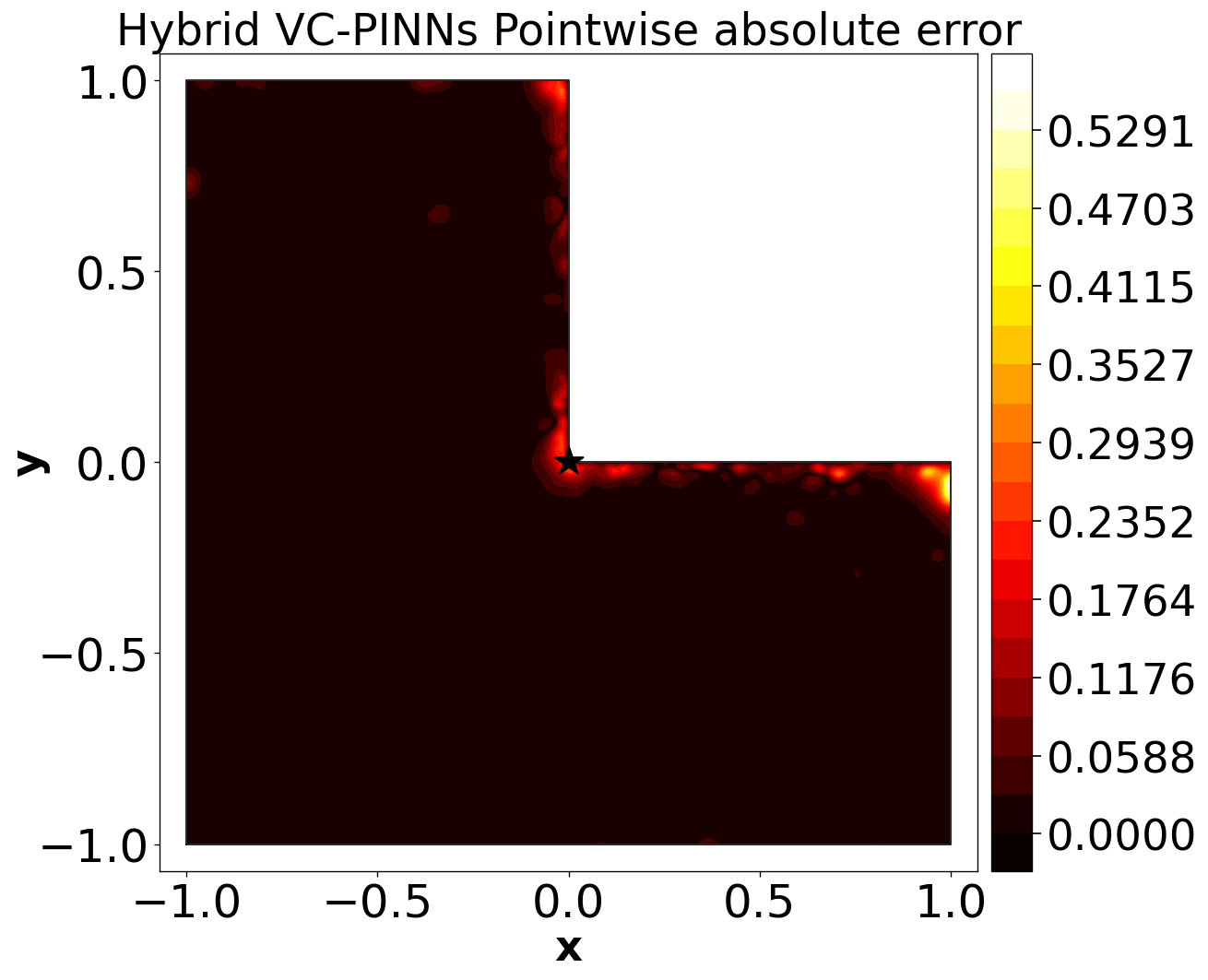}
\end{subfigure}

%---------------- Row 2 ----------------%
\begin{subfigure}{0.32\textwidth}
    \textbf{D}\\[0.5mm]
    \includegraphics[width=\linewidth]{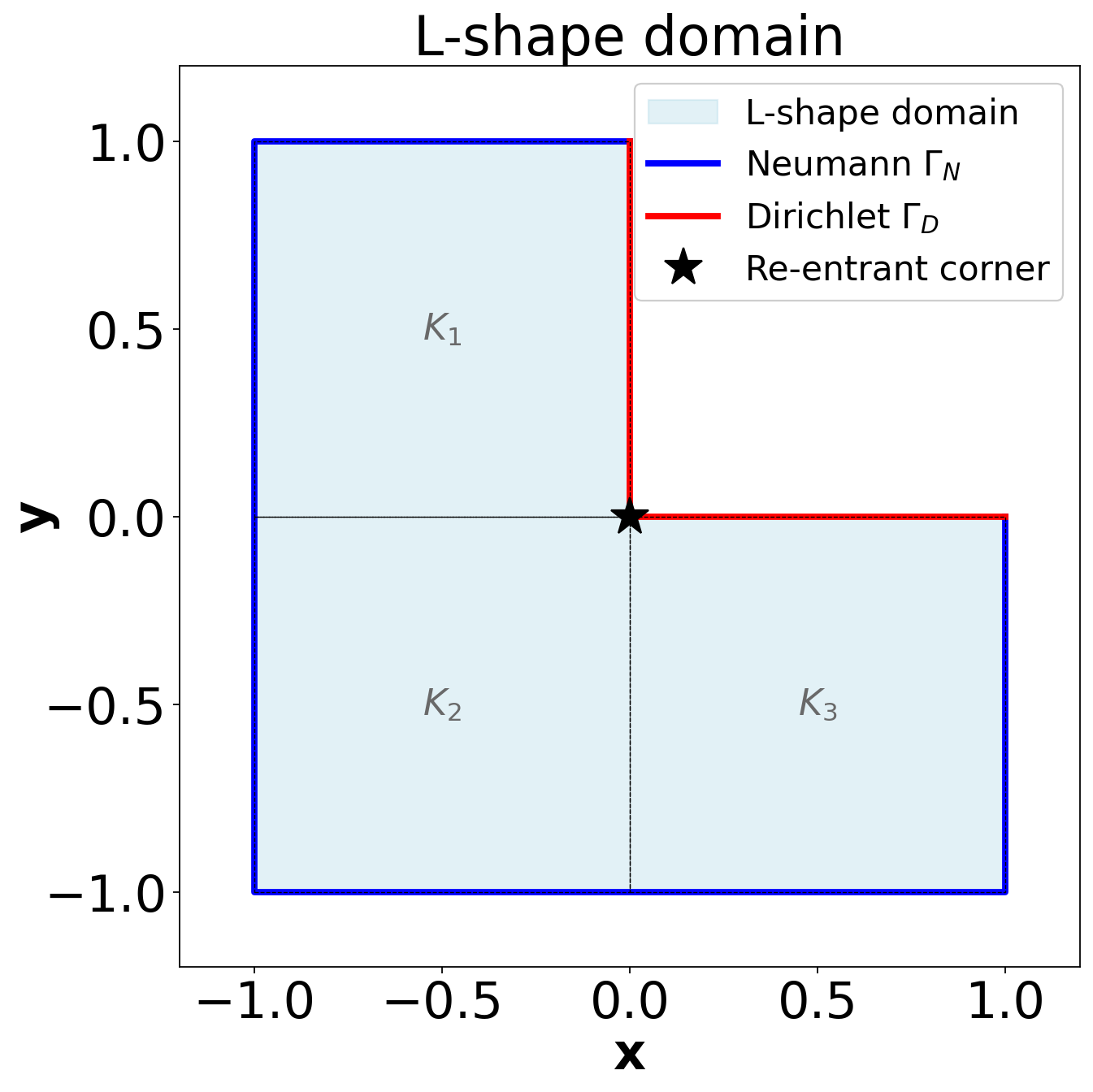}
\end{subfigure}
\hfill
\begin{subfigure}{0.32\textwidth}
    \textbf{E}\\[0.5mm]
    \includegraphics[width=\linewidth]{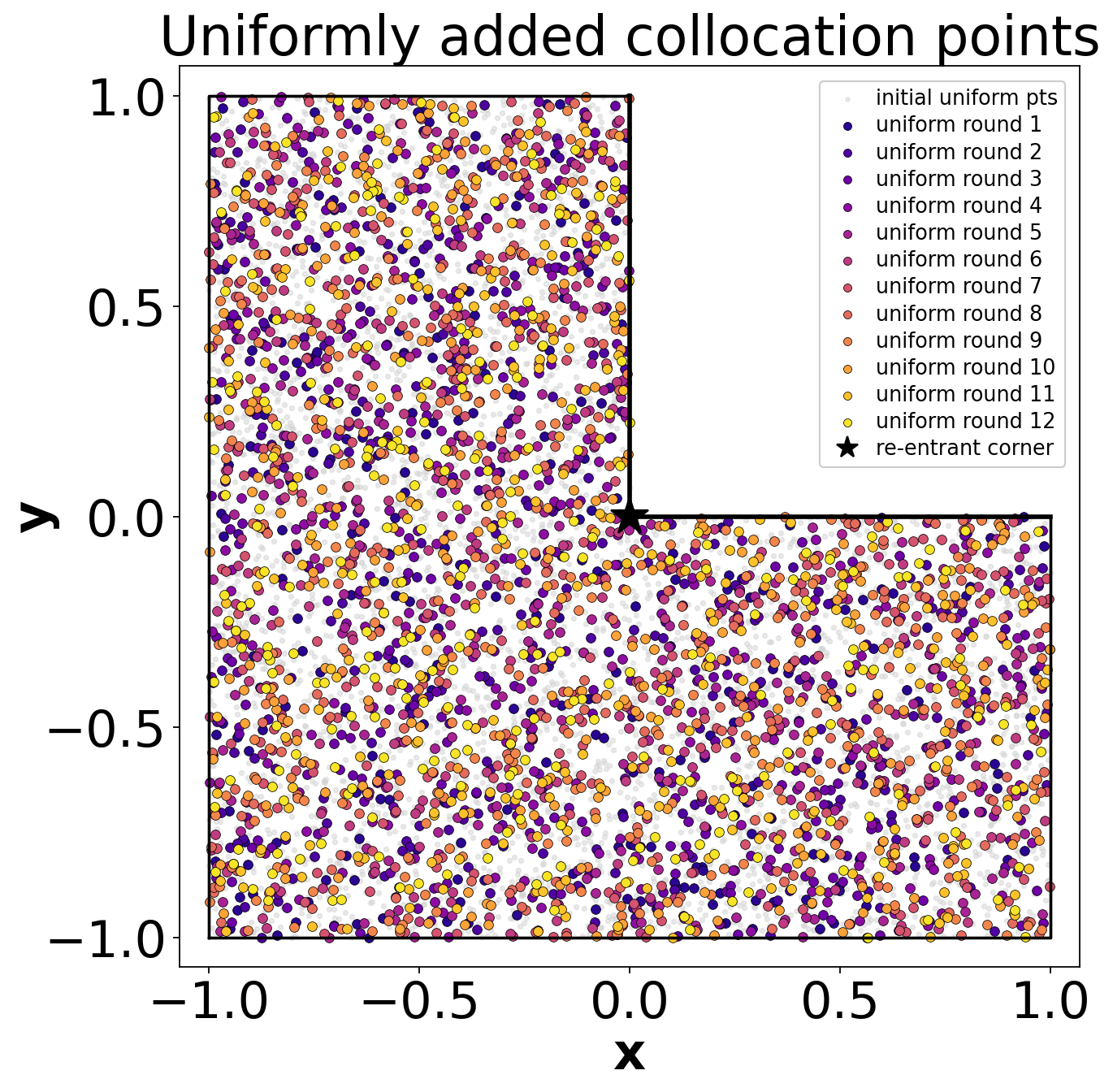}
\end{subfigure}
\hfill
\begin{subfigure}{0.32\textwidth}
    \textbf{F}\\[0.5mm]
    \includegraphics[width=\linewidth]{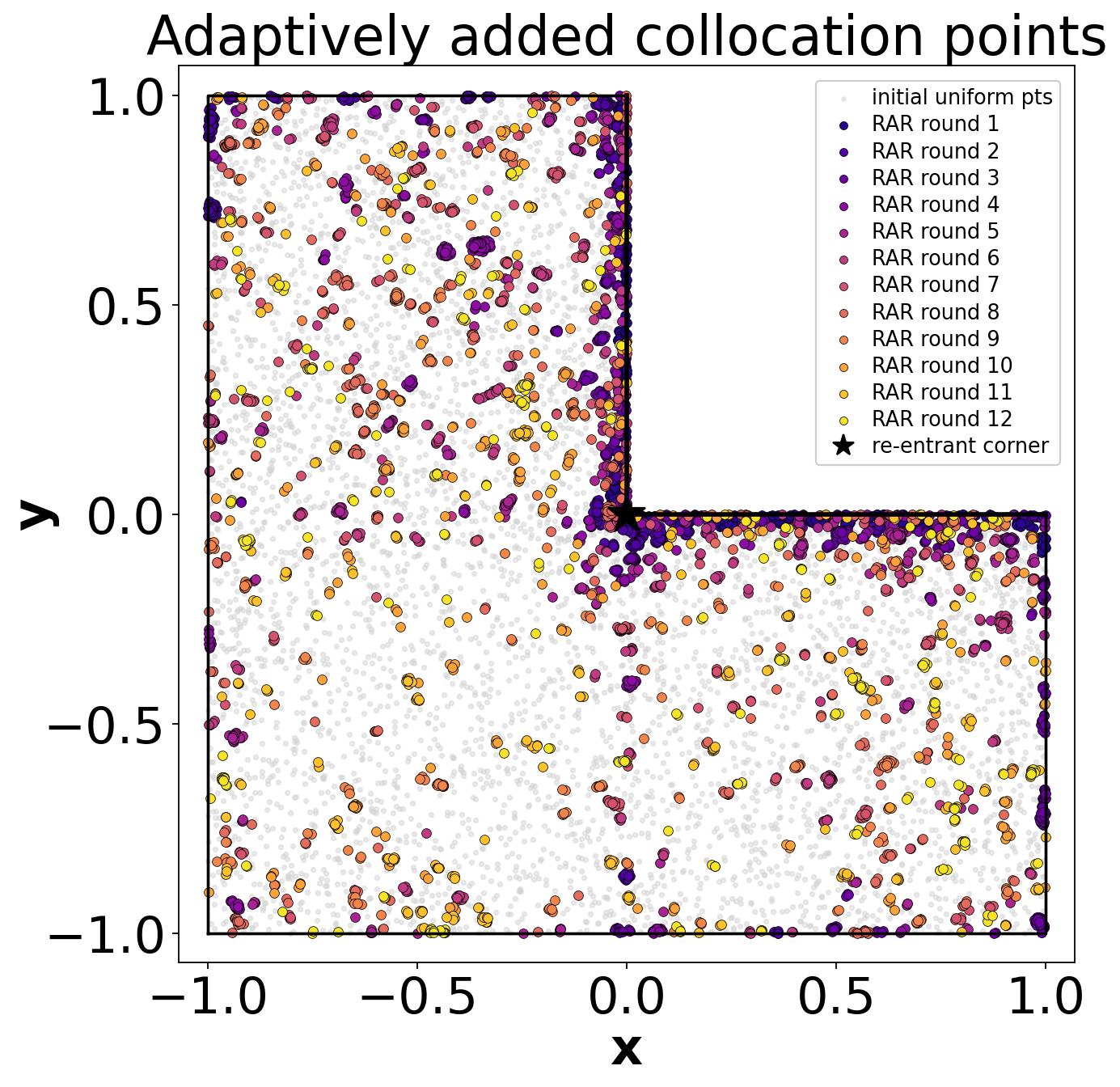}
\end{subfigure}

%---------------- Row 3 ----------------%
\begin{subfigure}{0.32\textwidth}
    \textbf{G}\\[0.5mm]
    \includegraphics[width=\linewidth]{semilinear_fem_solution.png}
\end{subfigure}
\hfill
\begin{subfigure}{0.32\textwidth}
    \textbf{H}\\[0.5mm]
    \includegraphics[width=\linewidth]{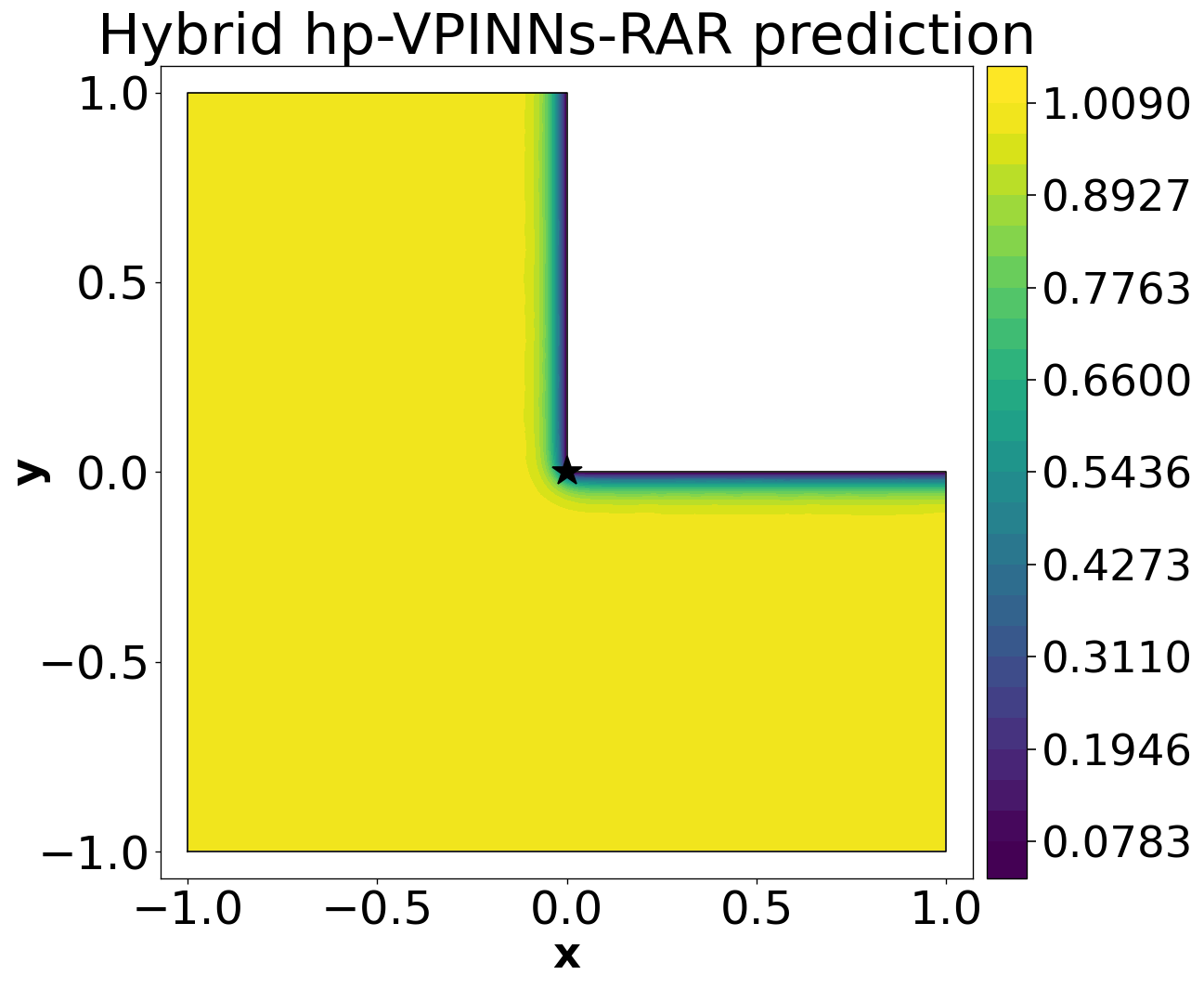}
\end{subfigure}
\hfill
\begin{subfigure}{0.32\textwidth}
    \textbf{I}\\[0.5mm]
    \includegraphics[width=\linewidth]{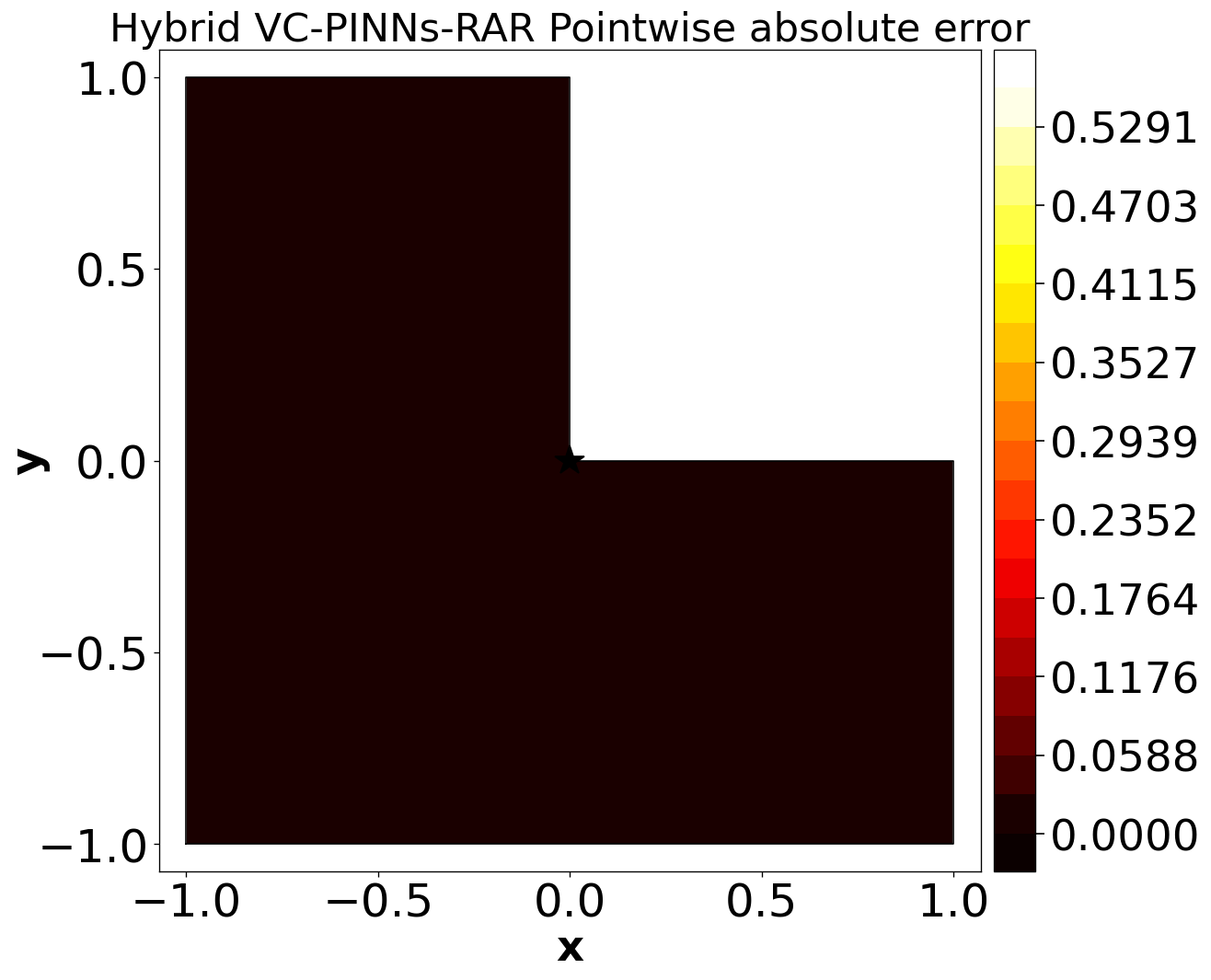}
\end{subfigure}

%---------------- Row 4 ----------------%
\begin{subfigure}{0.44\textwidth}
    \textbf{J}\\[0.5mm]
    \includegraphics[width=\linewidth]{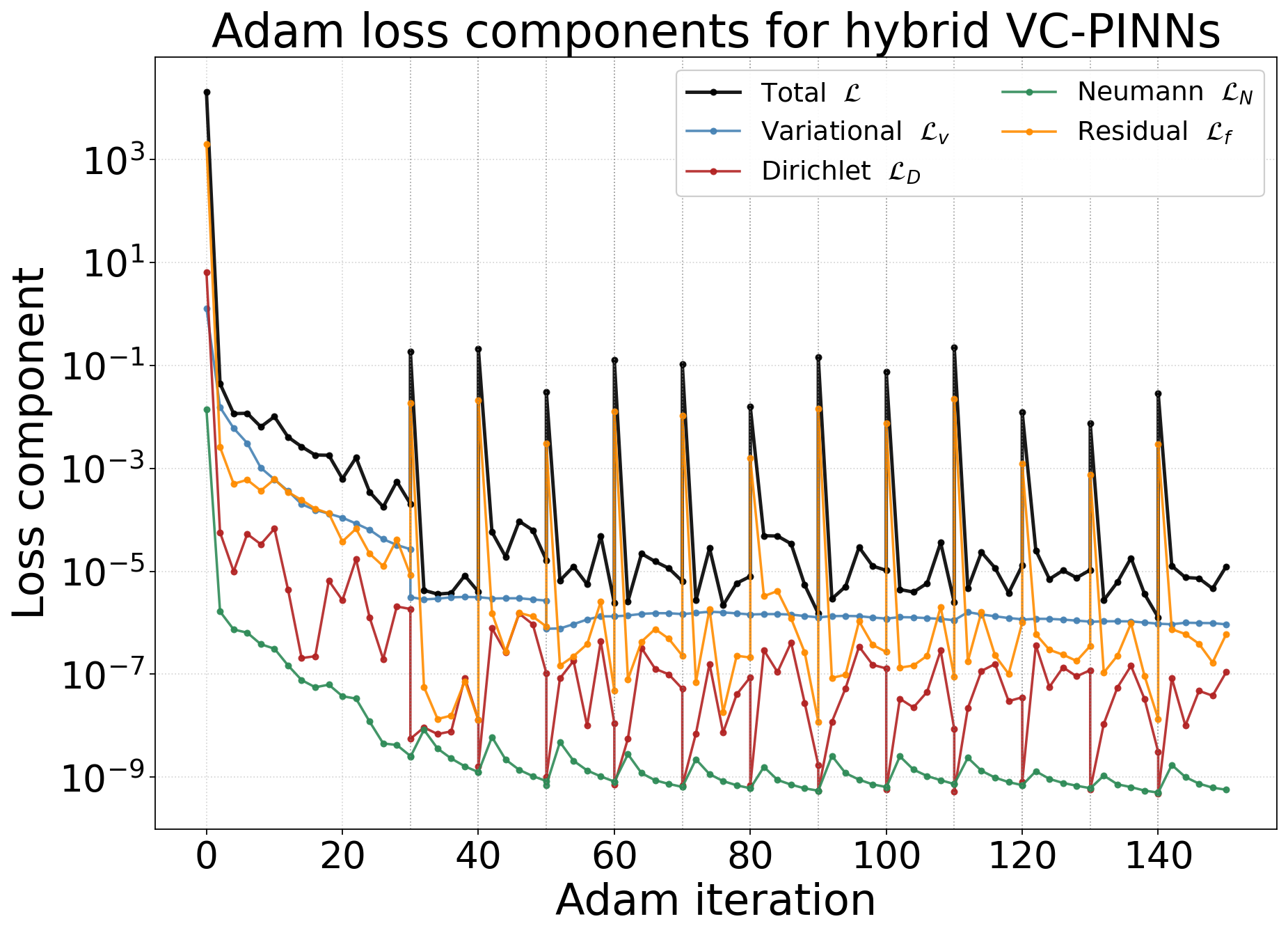}
\end{subfigure}
\hspace{0.04\textwidth}
\begin{subfigure}{0.44\textwidth}
    \textbf{K}\\[0.5mm]
    \includegraphics[width=\linewidth]{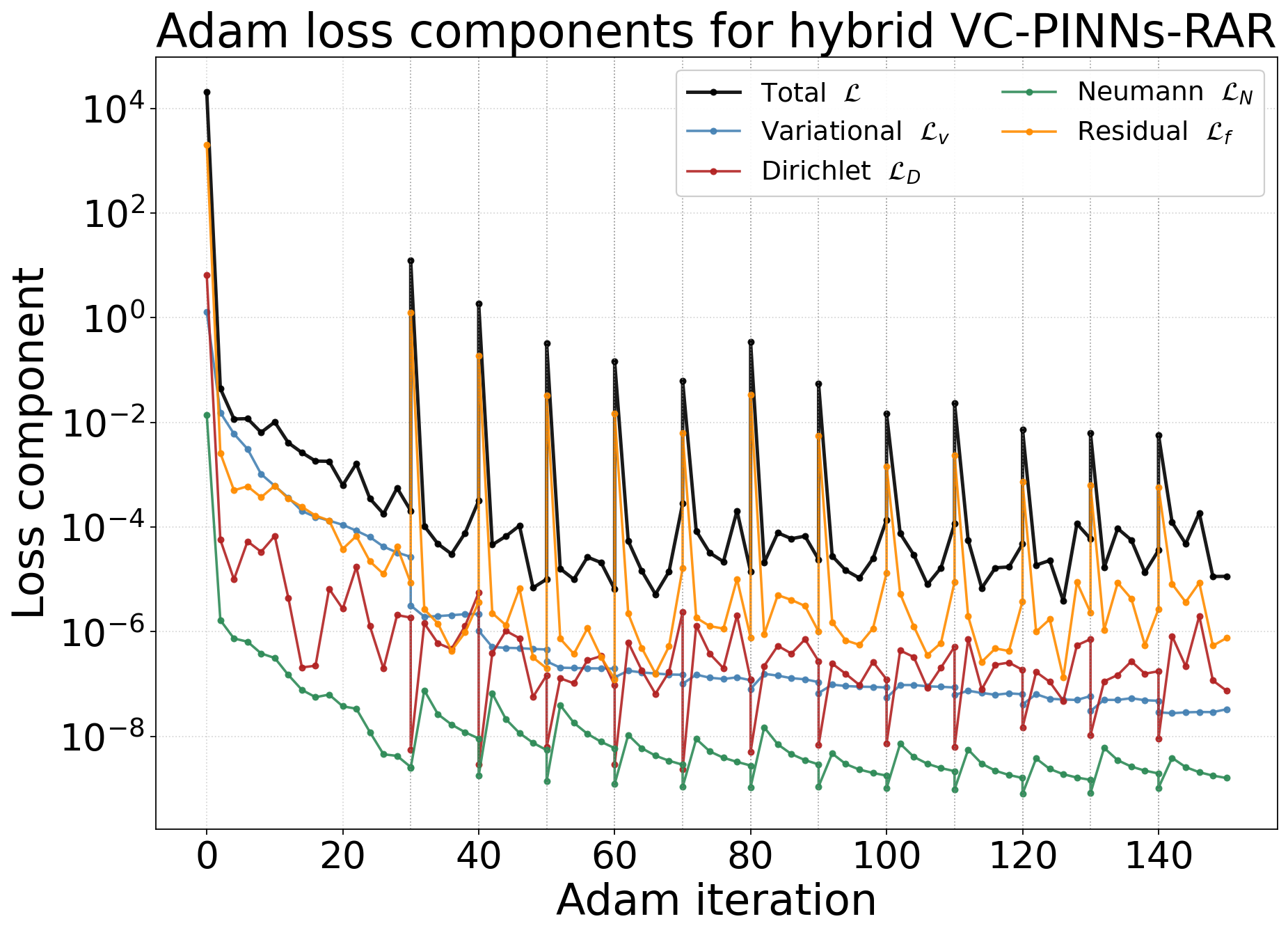}
\end{subfigure}
\caption{Semilinear problem on the L-shaped domain. \textbf{(A-C)} Finite element reference, VC-PINNs (uniform) prediction and pointwise absolute error against the reference. \textbf{(D-F)} Domain with boundary conditions and element partition, uniform collocation set and adaptive collocation set after twelve rounds, $|\mathcal{T}_{f}|=4000\to7600$. \textbf{(G-I)} Finite element reference, VC-PINNs-RAR prediction and pointwise absolute error. \textbf{(J-K)} Adam loss components for the uniform and the RAR arm.}\label{fig:PDE5}
\end{figure}

%PDE6
\subsection{The classical Motz problem}\label{motz_PDE6}
Each preceding singular example draws its singularity from the geometry, a slit or a re-entrant corner. The Motz problem removes that source entirely. The domain is a convex rectangle and the solution is not in $H^{2}(\Omega)$, since the boundary condition changes type at an interior point of a straight edge. It is a standard benchmark for singularity resolving numerical methods~\cite{LuHuLi2004,Li2006}. Here it tests whether RAR responds to the singular behavior of the solution rather than to geometric cues.

We consider $\Delta u = 0$ on $\Omega=(-1,1)\times(0,1)$ with
\begin{equation}
u=0 \ \text{ on } \{(x,0):-1\le x\le0\}, \qquad
u=500 \ \text{ on } \{(1,y):0\le y\le1\}, \qquad
\partial_{\bm{n}}u=0 \ \text{ on } \Gamma_N,
\label{eq:motz_bc}
\end{equation}
where $\Gamma_D$ is the union of the two Dirichlet segments and $\Gamma_N$ comprises the right half of the bottom edge $\{(x,0):0\le x\le1\}$, the top edge and the left edge. The boundary condition changes from Dirichlet to Neumann at the origin, which is an interior point of the bottom edge, the interior angle there is $\pi$. Separation of variables about that point admits solutions $r^{\lambda}\cos(\lambda\theta)$ with $\theta\in[0,\pi]$ satisfying $u=0$ at $\theta=\pi$ and $\frac{\partial u}{\partial \theta}=0$ at $\theta=0$, which forces $\lambda\in\{\tfrac12,\tfrac32,\tfrac52,\dots\}$ and gives the Williams expansion~\cite{Williams1952}
\begin{equation}
u(r,\theta)=\sum_{k\ge0}a_k\,r^{k+1/2}\cos\!\big((k+\tfrac12)\theta\big).
\label{eq:motz_series}
\end{equation}
The leading term has $\lambda_{1}=\tfrac12$, so $u\in H^{s}(\Omega)$ for every $s<3/2$ however, $u\notin H^{2}(\Omega)$ and $|\nabla u|\sim r^{-1/2}$ as $r\to0$. This is the same rate as in the crack problem (section~\ref{creak_slit_PDE1}), produced here by a change of boundary condition type on a straight edge rather than by a slit. The problem originates in electrostatics, $u$ being the potential in a rectangular region held at fixed potentials on part of the boundary and insulated on the rest.

The difficulty is that nothing about $\Omega$ marks the origin as special. A refinement criterion driven by geometry has no signal to act on and a uniform collocation set treats the neighborhood of the origin exactly as it treats the bulk, whereas the error concentrates there. Since $f\equiv0$, the strong-form residual measures purely the failure of the network to be harmonic and is therefore available as an indicator that depends only on the solution.

The lines $x=0$ and $y=0.5$ partition the rectangle into four rectangular elements, so that the singular point lies on an element interface. The network input is enriched with the leading mode alone, so that $\hat{u}$ acts on $(x,y,r^{1/2}\cos(\theta/2))$, since $\theta\in[0,\pi]$ on this domain, the map is single valued and needs no branch cut. Only one term of \eqref{eq:motz_series} is supplied and the rest of the series must be learned, this distinguishes the example from the crack problem of section~\ref{creak_slit_PDE1}, where the enrichment reproduces the exact solution (remark~\ref{rem:crackenrich}). As stated in section~\ref{Sec3.1_trial_enrich}, the network represents $u/500$ ($c_{u}=500$), which is $O(1)$ and the scaling is undone before any error is reported. The two arms are budget matched in collocation count, Adam iterations and L-BFGS stages, each run recording its budget so that a drift between them is detected rather than assumed absent, they differ only in whether the points inserted each round are drawn uniformly or taken as the largest strong-form residuals from a candidate pool.

With no closed from solution available, errors are measured against a quadratic ($\mathbb{P}_{2}$) finite element reference that plays the role of the exact solution. Accuracy at the singularity itself is assessed independently of any pointwise norm by extracting the first four coefficients of \eqref{eq:motz_series}. Using the orthogonality
$\int_0^\pi\cos((k+\tfrac12)\theta)\cos((j+\tfrac12)\theta)\,d\theta=\tfrac{\pi}{2}\delta_{kj}$,
\begin{equation}
a_k=\frac{2}{\pi r_0^{\,k+1/2}}\int_0^{\pi}u(r_0,\theta)\cos\!\big((k+\tfrac12)\theta\big)\,d\theta,
\label{eq:motz_extraction}
\end{equation}
evaluated at $r_0=0.5$ by $400$ node Gauss-Legendre quadrature. The same projection is applied to the finite element reference, so network coefficients, reference coefficients and the published collocation Trefftz values~\cite{LuHuLi2004} are directly comparable.

\begin{table}[t]
\centering
\caption{Motz problem. The first three rows are relative errors in percent (\%), the last row is the leading coefficient $a_{0}$. Bold marks the more accurate arm in each row. Both arms use the same collocation schedule ($|\mathcal{T}_{f}|=1500\to3900$), $150{,}000$ Adam iterations and thirteen L-BFGS stages, one training run per arm. Errors are measured against the quadratic ($\mathbb{P}_{2}$) finite element reference, whose own error floor is $0.000079\%$. As $0\le u\le500$ with $\max u=500$, the relative $L^{\infty}$ error is the absolute error divided by $500$. The leading coefficient is extracted by \eqref{eq:motz_extraction} at $r_0=0.5$, its reference value is $a_0=401.16191$.}
\label{tab:motz}
\begin{tabular}{lcc}
\toprule
Metric & VC-PINNs (uniform) & VC-PINNs-RAR \\
\midrule
Relative $L^{2}$ error           & $0.2407$          &  $\mathbf{0.2096}$   \\
Relative $L^{\infty}$ error      & $0.7872$          & $\mathbf{0.4458}$ \\
Relative $H^{1}$ semi-norm error & $1.666$           & $\mathbf{0.9920}$ \\
\midrule Leading coefficient $a_{0}$ & $399.92050$       & $\mathbf{400.15007}$ \\
\bottomrule
\end{tabular}
\end{table}

\begin{figure}[!tb]
\centering
%---------------- Row 1 ----------------%
\begin{subfigure}{0.32\textwidth}
    \textbf{A}\\[0.5mm]
    \includegraphics[width=\linewidth]{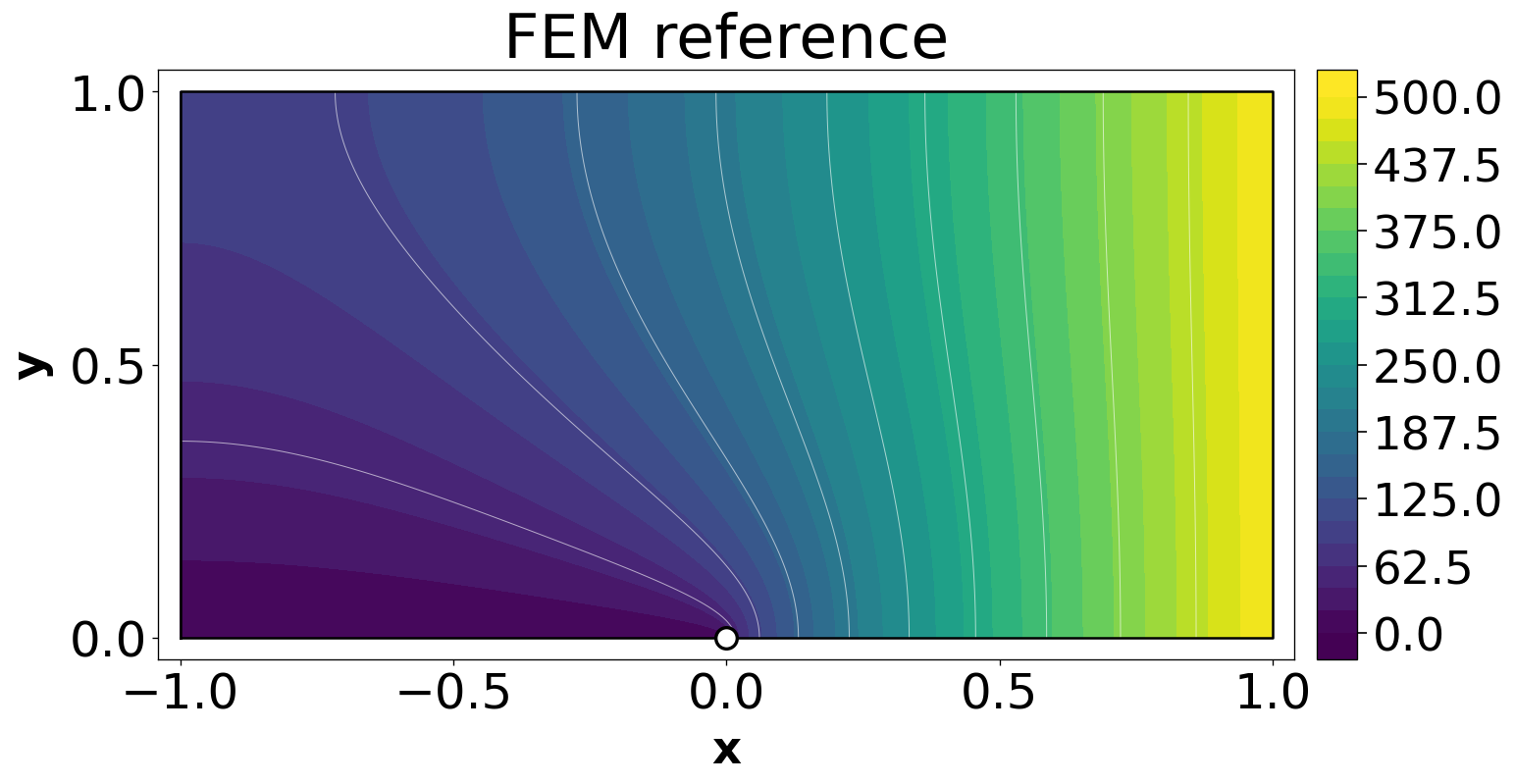}
\end{subfigure}
\hfill
\begin{subfigure}{0.32\textwidth}
    \textbf{B}\\[0.5mm]
    \includegraphics[width=\linewidth]{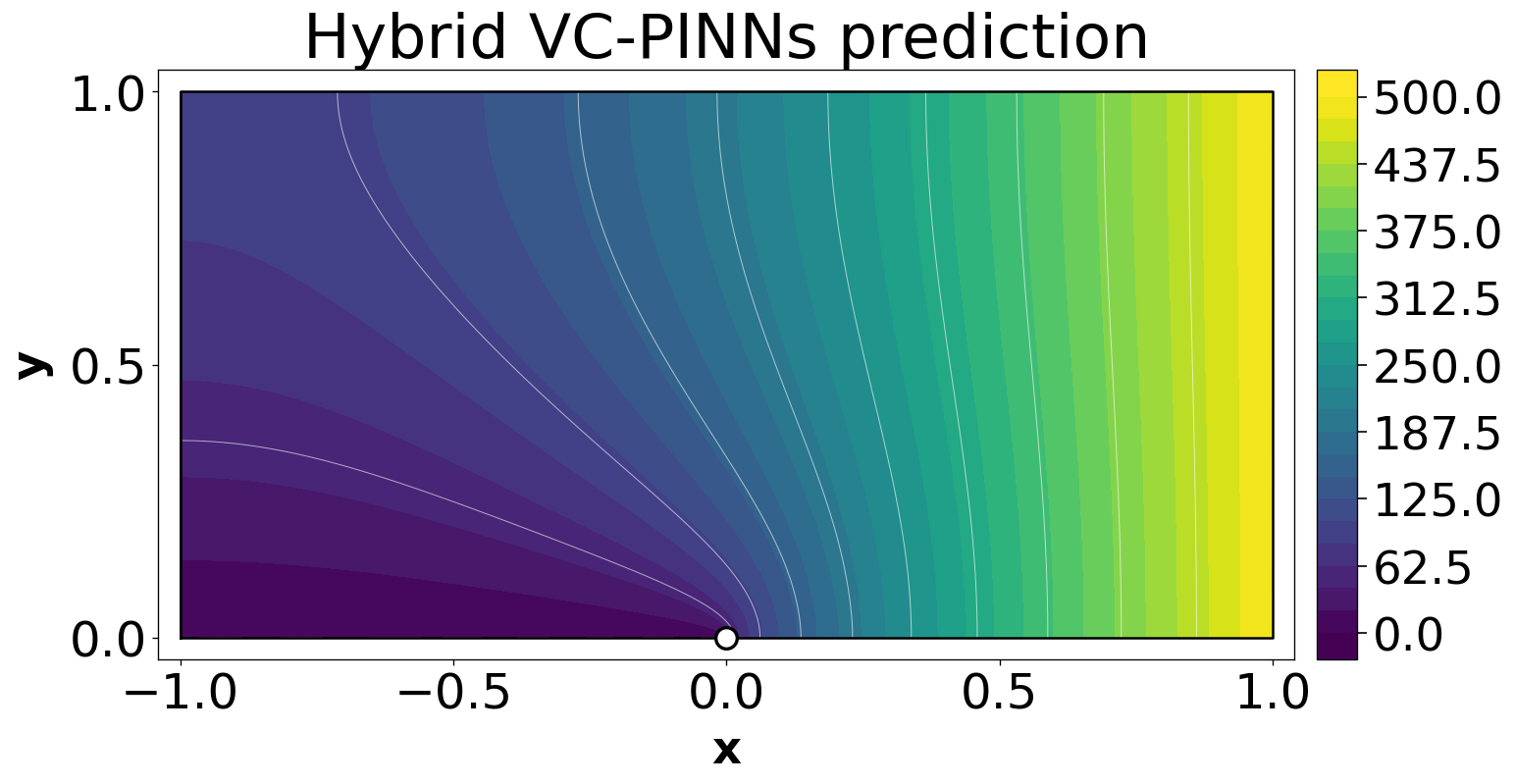}
\end{subfigure}
\hfill
\begin{subfigure}{0.32\textwidth}
    \textbf{C}\\[0.5mm]
    \includegraphics[width=\linewidth]{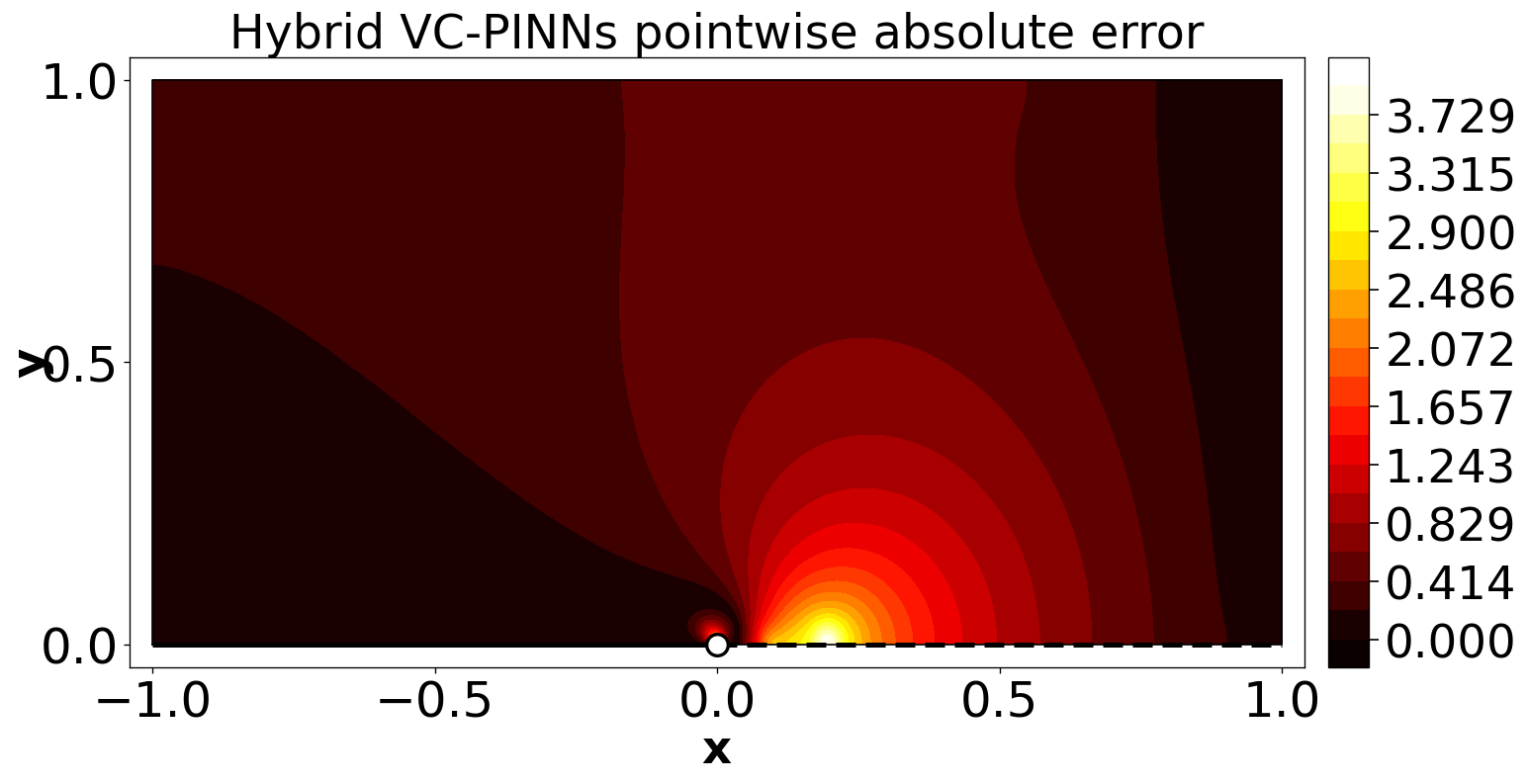}
\end{subfigure}

%---------------- Row 2 ----------------%
\begin{subfigure}{0.32\textwidth}
    \textbf{D}\\[0.5mm]
    \includegraphics[width=\linewidth]{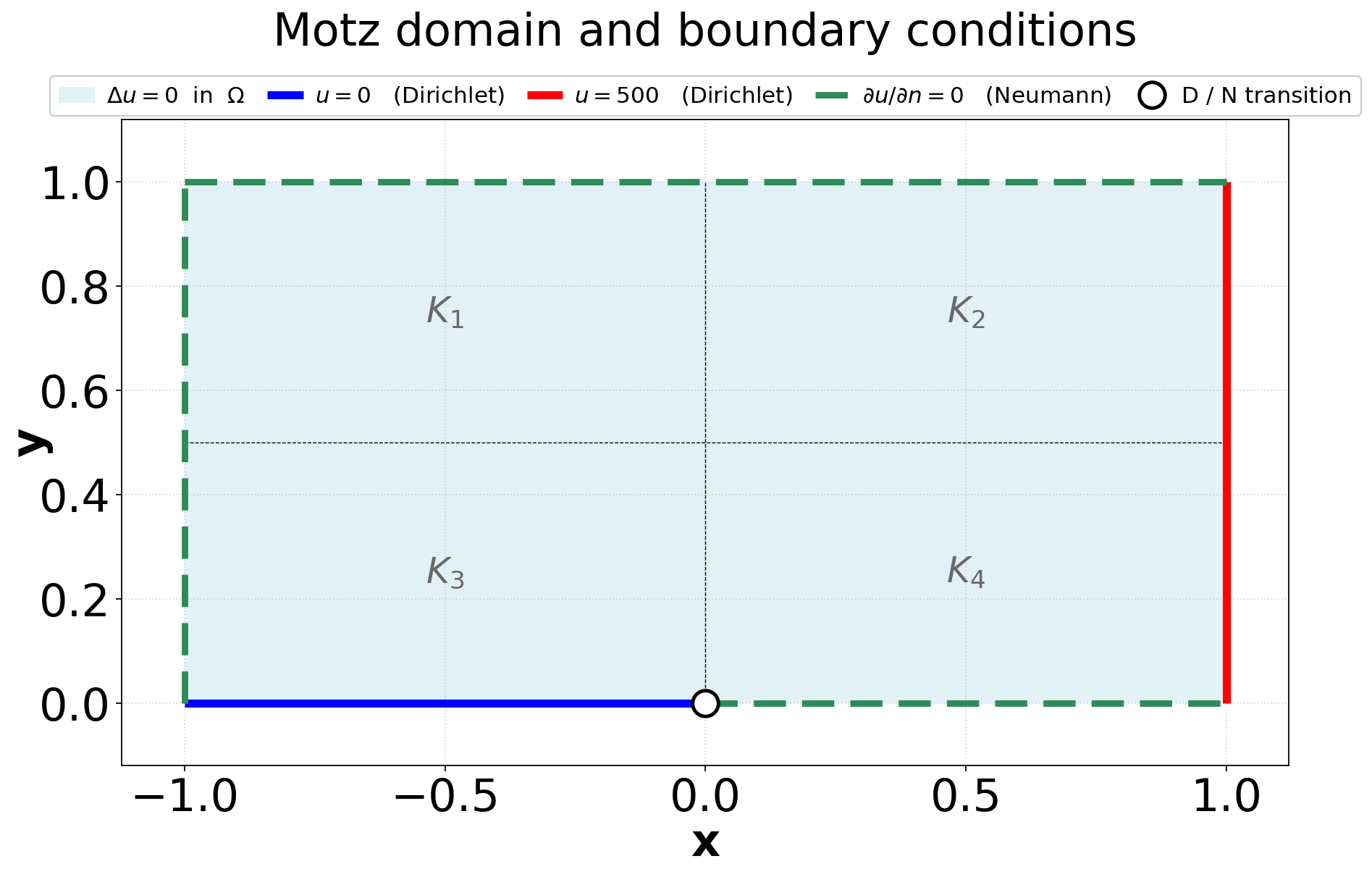}
\end{subfigure}
\hfill
\begin{subfigure}{0.32\textwidth}
    \textbf{E}\\[0.5mm]
    \includegraphics[width=\linewidth]{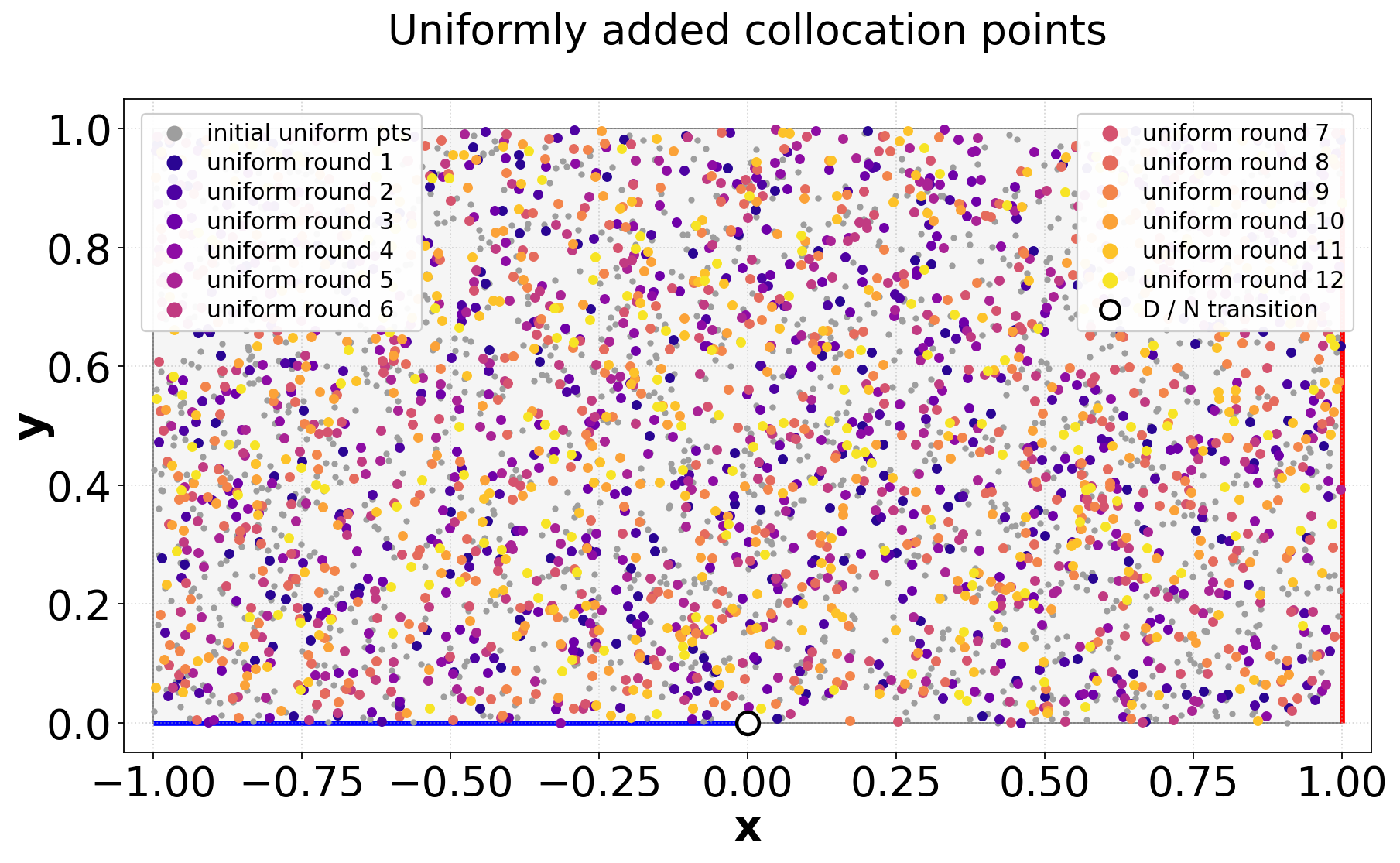}
\end{subfigure}
\hfill
\begin{subfigure}{0.32\textwidth}
    \textbf{F}\\[0.5mm]
    \includegraphics[width=\linewidth]{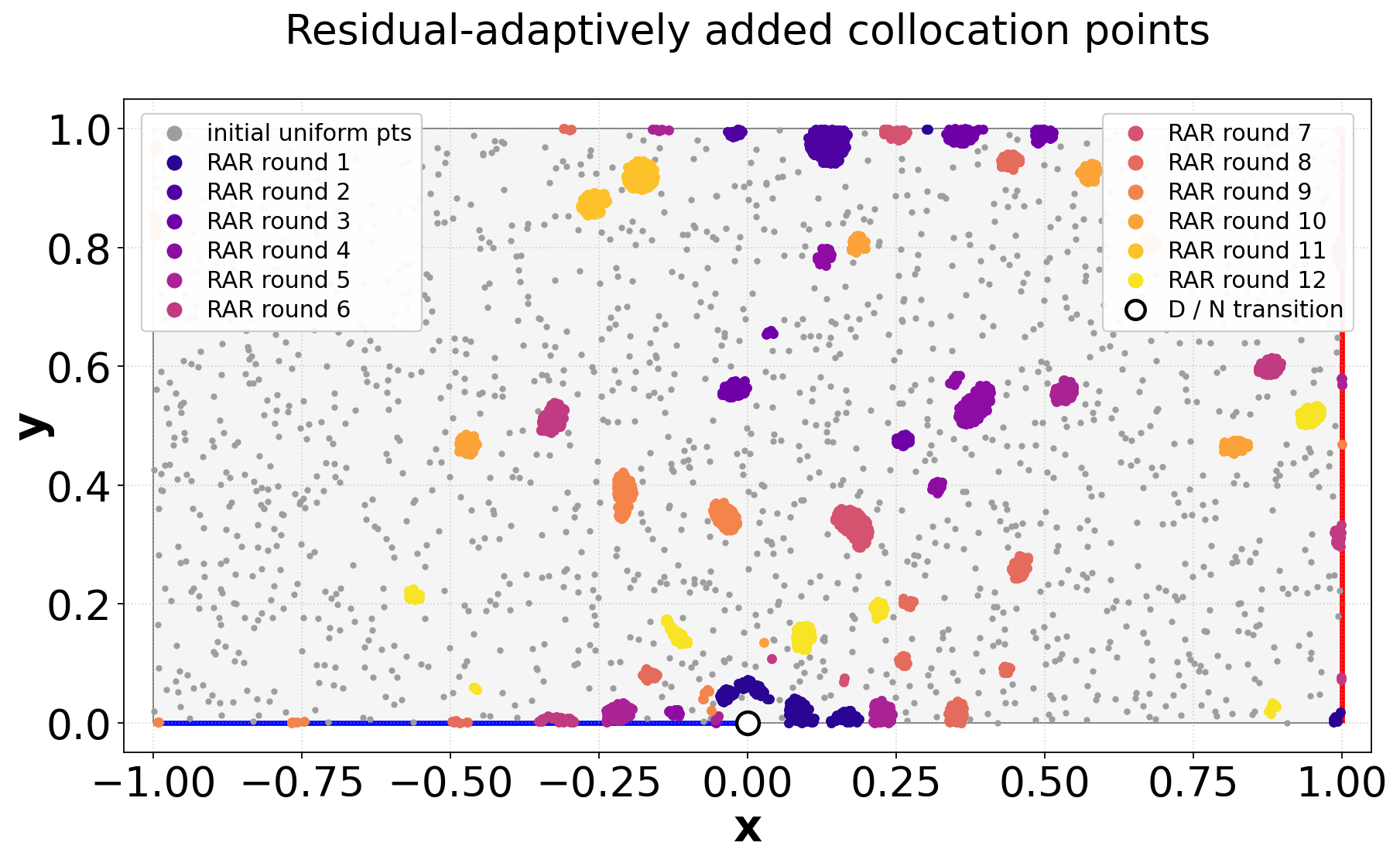}
\end{subfigure}

%---------------- Row 3 ----------------%
\begin{subfigure}{0.32\textwidth}
    \textbf{G}\\[0.5mm]
    \includegraphics[width=\linewidth]{motz_uniform_exact.png}
\end{subfigure}
\hfill
\begin{subfigure}{0.32\textwidth}
    \textbf{H}\\[0.5mm]
    \includegraphics[width=\linewidth]{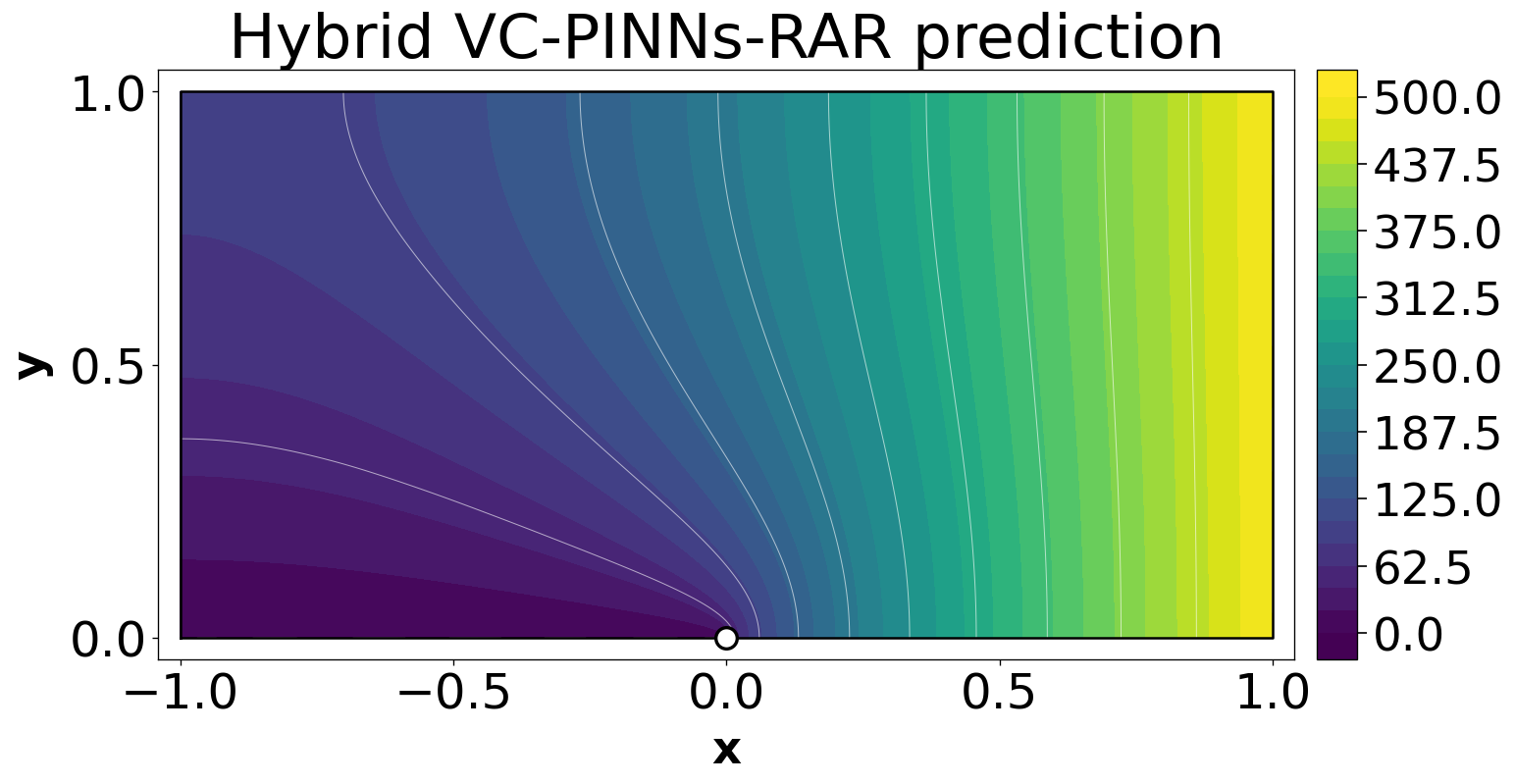}
\end{subfigure}
\hfill
\begin{subfigure}{0.32\textwidth}
    \textbf{I}\\[0.5mm]
    \includegraphics[width=\linewidth]{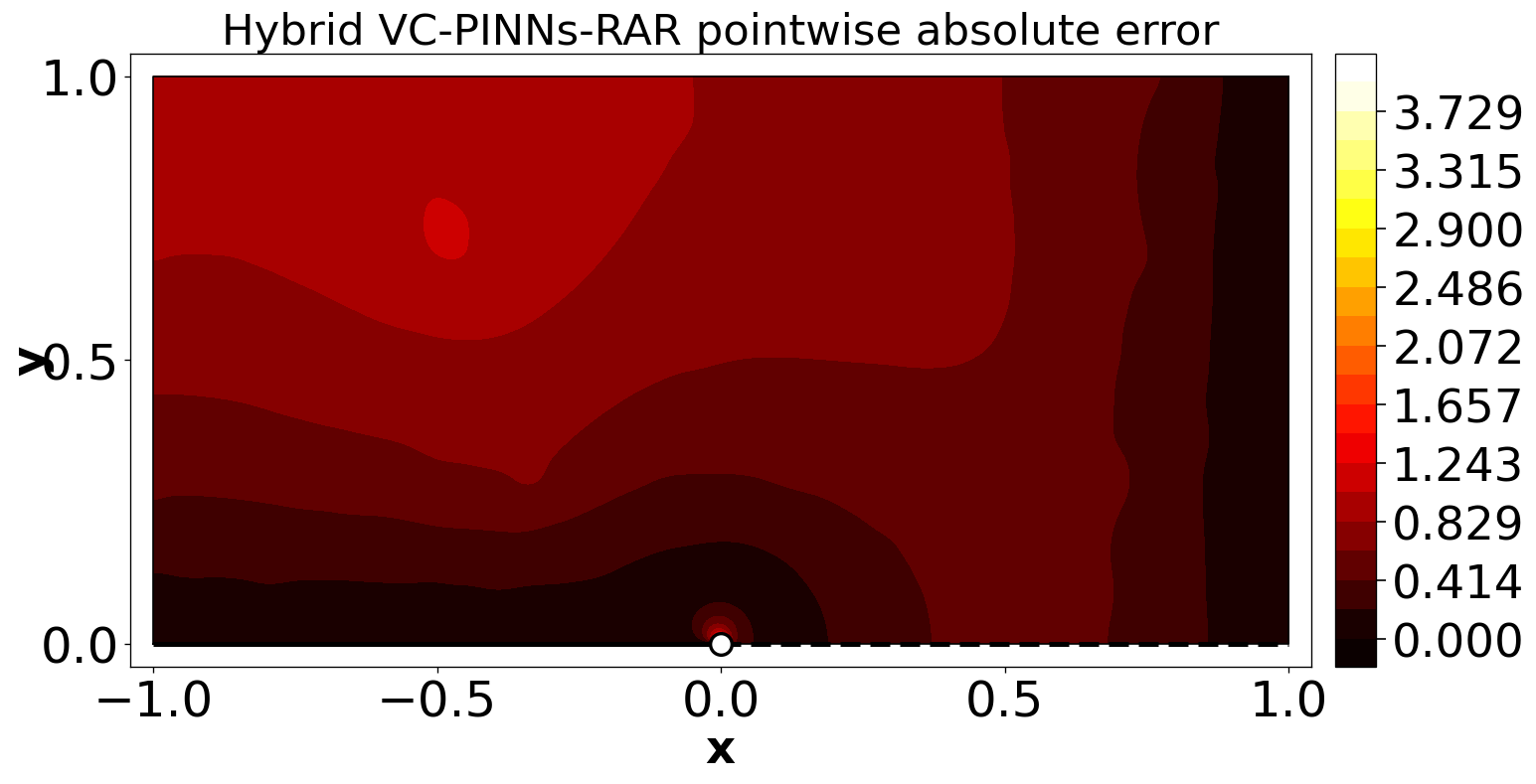}
\end{subfigure}

%---------------- Row 4 ----------------%
\begin{subfigure}{0.44\textwidth}
    \textbf{J}\\[0.5mm]
    \includegraphics[width=\linewidth]{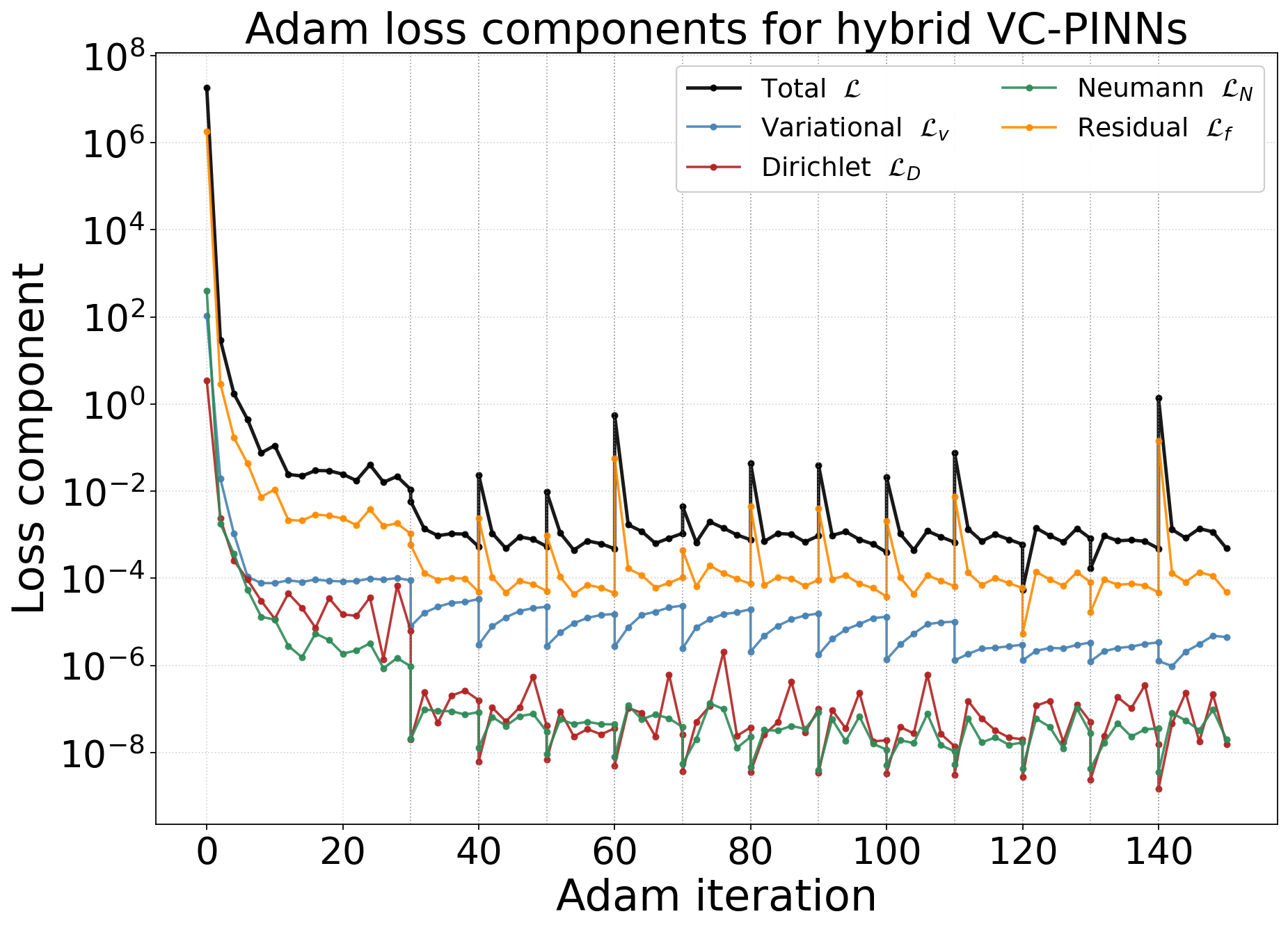}
\end{subfigure}
\hspace{0.04\textwidth}
\begin{subfigure}{0.44\textwidth}
    \textbf{K}\\[0.5mm]
    \includegraphics[width=\linewidth]{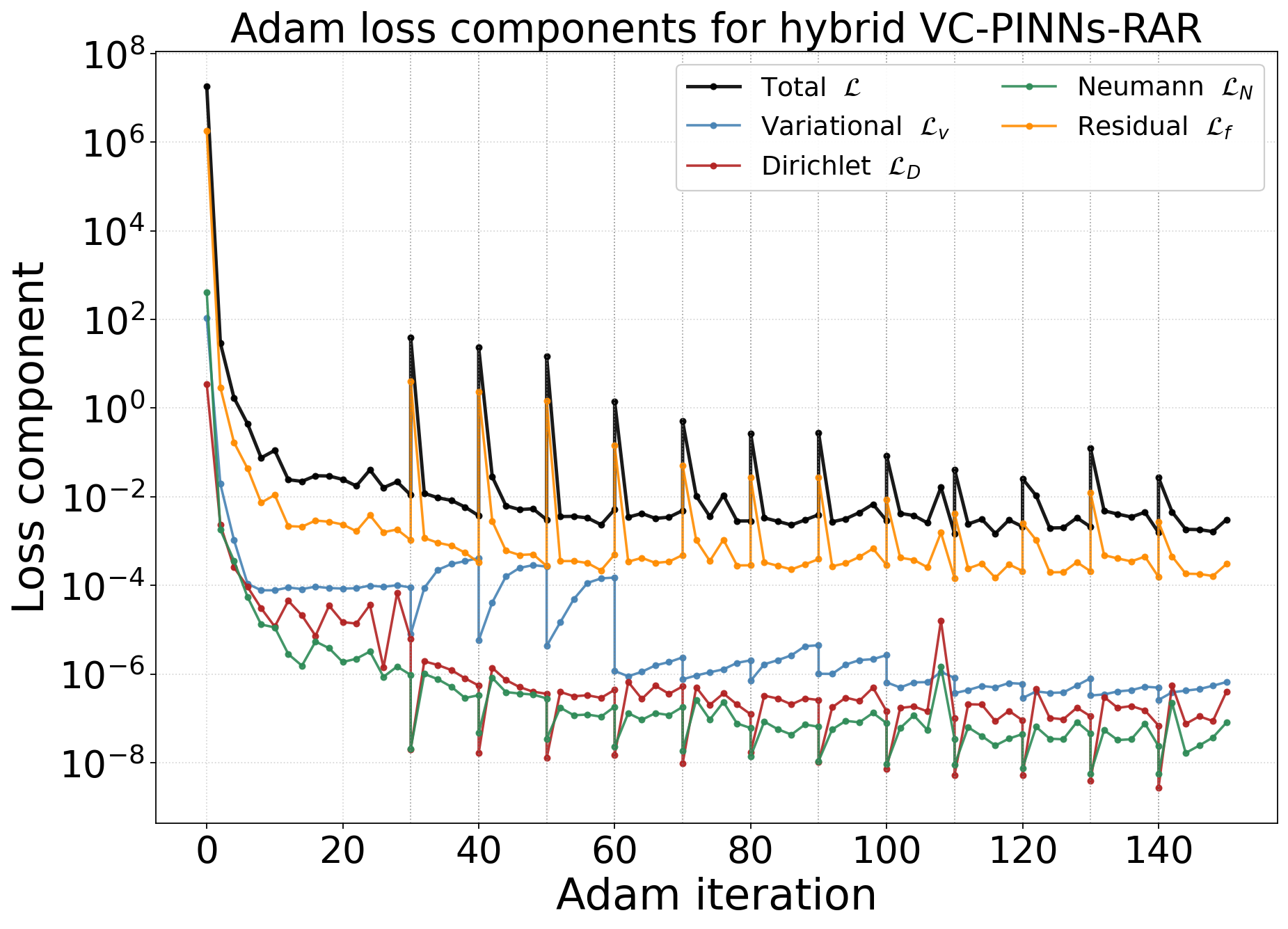}
\end{subfigure}
\caption{The Motz problem. \textbf{(A-C)} Finite element reference solution, VC-PINNs (uniform) prediction and pointwise absolute error against the reference. \textbf{(D-F)} Domain with the mixed boundary conditions and element partition, uniform collocation set and adaptive collocation set after twelve rounds, $|\mathcal{T}_{f}|=1500\to3900$. \textbf{(G-I)} Finite element reference solution, VC-PINNs-RAR prediction and pointwise absolute error. \textbf{(J-K)} Adam loss components for the uniform and the RAR arm.}\label{fig:PDE6_figures}
\end{figure}

\section{Concluding remarks}\label{Sec6_conclusion}
This work studied residual-based adaptive refinement within a hybrid variational-collocation PINNs for second-order elliptic boundary value problems with Dirichlet and Neumann data. The study covers geometric and boundary induced singularities, localized interior features and a boundary layer. The approach rests on a structural observation. A Petrov-Galerkin weak form yields one functional per element and no pointwise residual, so point level adaptivity cannot be driven by the variational term alone. The collocation term, built on the strong-form residual, supplies the missing pointwise quantity. Refinement therefore acts on the collocation set only
and the element partition, the test space and the quadrature rule are left
unchanged.

The scheme was assessed on five benchmarks, each chosen to isolate a distinct difficulty. In every comparison the adaptive and uniform runs share the final collocation count, the number of Adam iterations and the L-BFGS schedule. The two arms therefore differ only in how the collocation set is placed and grown. The comparison is matched in budget. The relative $H^{1}$ semi-norm error decreased on all five problems, by factors between $1.7$ and $16$, which is consistent with refinement concentrating where the gradient is steep or unbounded. The $L^{2}$ error was less uniform. It ranged from a $32$ fold reduction on the cross-shaped domain to a slight increase on the Motz problem. The individual benchmarks
lead to the following conclusions.

\emph{Mode~III crack.} The slit makes the domain non-Lipschitz and the solution lies only in $H^{3/2-\epsilon}$ near the tip. Refinement graded the collocation set toward the crack tip and reduced both the $L^{2}$ and the $H^{1}$ errors. This problem motivated three selection devices, radial normalization of the indicator, density based sampling and a minimum separation filter. Together they prevent successive rounds from collapsing onto a small cluster of candidates. The RAR arm reduced the relative \(H^1\) seminorm error by 1.65 times compared with uniform arm.

\emph{Cross-shaped domain.} Four re-entrant corners generate $r^{2/3}$ singularities of unequal amplitude, which compete for the same collocation budget. The indicator ranked the corners by amplitude from a single scalar per candidate point, with no knowledge of the singular coefficients. This benchmark gave the largest $L^{2}$ reduction, a factor of $32$.

\emph{Reaction-diffusion with localized peaks.} Although the solution is smooth on a convex square, it has four sharp interior peaks with $O(1/\sigma)$ gradients. The problem therefore tests whether the indicator can resolve interior features of finite yet large gradient, rather than geometric singularities. RAR arm reduces the relative \(L^\infty\) error by a factor of 38 times compared with uniform arm.

\emph{Semilinear singularly perturbed problem.} For $-\varepsilon^{2}\Delta u + u^{3} = 1$ with $\varepsilon = 0.05$ on the L-shaped domain, a boundary layer along $\Gamma_{D}$ coexists with the $r^{2/3}$ singularity at the re-entrant corner. Accuracy was measured against a graded-mesh $P_{2}$ finite element reference. The indicator divided refinement between a one-dimensional layer and a point singularity, again from one scalar per candidate. RAR arm reduces the relative $L^\infty$ error by a factor of 16 times compared with uniform arm.

\emph{Motz problem.} Here the singularity arises from the change from Dirichlet to Neumann data along a straight edge, not from the geometry. Under the bare residual ranking, refinement improved every measure weighted toward the singular point, the $H^{1}$ semi-norm included. The domain-averaged $L^{2}$ error, however, increased slightly. The mechanism is identifiable. Each round placed its new points in a single neighborhood instead of grading them toward the singularity, so resolution near the
singular point was gained at the expense of the rest of the domain. RAR arm reduces the relative $L^\infty$  error by a factor of 2 times compared with
uniform arm.

In both the crack and the Motz problems the indicator located the singularity correctly. What differed was how ranked candidates were turned into new collocation
points. The evidence points to the selection rule, at least as much as the indicator, as the component that makes residual driven refinement reliable.
The radial normalization, density sampling and minimum separation filter were designed to prevent exactly the collapse observed on the Motz problem.
That they do so there remains to be confirmed.

Three limitations remain. First, the strong-form residual requires second derivatives of the network. Second, Neumann data are imposed by penalty.
Third, only the collocation set is adapted and no $h$- or $p$-refinement of the test space is performed. Finally, all results are outcome from a single random seed run.

\section*{CRediT authorship contribution statement}
\textbf{Somnath Giri:} Conceptualization, Methodology, Software, Investigation, Writing - original draft.
\textbf{Muslim Malik:} Supervision, Methodology, Writing - review \& editing.
\textbf{Rishikesh Yadav:} Supervision, Methodology, Validation, Writing - review \&
editing.

\section*{Disclosure statement}
No potential conflict of interest was reported by the author(s).

\section*{Declaration of competing interest}
The authors declare that they have no known competing financial interests or personal relationships that could have appeared to influence the work reported in this paper.

\section*{Data availability}
No data was used for the research described in the article.

%%=====================================================================


\begin{thebibliography}{99}
%1
\bibitem{Ciarlet2002}
Ciarlet, P. G. $(2002)$. The finite element method for elliptic problems. Society for Industrial and Applied Mathematics.

%2
\bibitem{BrennerScott2008}
Brenner, S. C., \& Scott, L. R. $(2008)$. The mathematical theory of finite element methods. New York, NY: Springer New York.

%3
\bibitem{LeVeque2007}
LeVeque, R. J. $(2007)$. Finite difference methods for ordinary and partial differential equations: steady-state and time-dependent problems. Society for Industrial and Applied Mathematics.

%4
\bibitem{Grisvard2011}
Grisvard, P. $(2011)$. Elliptic problems in nonsmooth domains. Society for Industrial and Applied Mathematics.

%5
\bibitem{Motz1947}
Motz, H. (1947). The treatment of singularities of partial differential equations by relaxation methods. Quarterly of Applied Mathematics, $4(4), 371-377.$

%6
\bibitem{Kondratev1967}
Kondrat'ev, V. A. $(1967)$. Boundary value problems for elliptic equations in domains with conical or angular points. Trudy Moskovskogo Matematicheskogo Obshchestva, $16, 209-292.$

%7
\bibitem{Verfurth2013}
Verfürth, R. $(2013)$. A posteriori error estimation techniques for finite element methods. OUP Oxford.

%8
\bibitem{BinevDahmenDeVore2004}
Binev, P., Dahmen, W., \& DeVore, R. $(2004)$. Adaptive finite element methods with convergence rates. Numerische Mathematik, $97(2), 219-268.$

%9
\bibitem{Cascon2008}
Cascon, J. M., Kreuzer, C., Nochetto, R. H., \& Siebert, K. G. $(2008)$. Quasi-optimal convergence rate for an adaptive finite element method. SIAM Journal on Numerical Analysis, $46(5), 2524-2550.$

%10
\bibitem{SzaboBabuska2021}
Szabó, B., \& Babuška, I. $(2021)$. Finite Element Analysis: Method. Verification and Validation, $2023(23), 980.$

%11
\bibitem{MitchellMcClain2011}
Mitchell, W. F., \& McClain, M. A. $(2011)$. A survey of hp-adaptive strategies for elliptic partial differential equations. In Recent advances in computational and applied mathematics (pp. $227-258)$. Dordrecht: Springer Netherlands.

%12
\bibitem{Raissi2019}
Raissi, M., Perdikaris, P., \& Karniadakis, G. E. $(2019)$. Physics-informed neural networks: A deep learning framework for solving forward and inverse problems involving nonlinear partial differential equations. Journal of Computational physics, $378, 686-707$.

%13
\bibitem{Lagaris1998}
Lagaris, I. E., Likas, A., \& Fotiadis, D. I. $(1998)$. Artificial neural networks for solving ordinary and partial differential equations. IEEE transactions on neural networks, $9(5), 987-1000.$

%14
\bibitem{Baydin2018}
Baydin, A. G., Pearlmutter, B. A., Radul, A. A., \& Siskind, J. M. $(2018)$. Automatic differentiation in machine learning: a survey. Journal of machine learning research, $18(153), 1-43.$

%15
\bibitem{Karniadakis2021}
Karniadakis, G. E., Kevrekidis, I. G., Lu, L., Perdikaris, P., Wang, S., \& Yang, L. $(2021)$. Physics-informed machine learning. Nature Reviews Physics, $3(6), 422-440.$

%16
\bibitem{NoorMalik2025}
Noor, M., \& Malik, M. $(2025)$. Deep neural network approach to solve coupled differential equations with multiple delays. Physica Scripta, $100(7), 075259.$

%17
\bibitem{Rahaman2019}
Rahaman, N., Baratin, A., Arpit, D., Draxler, F., Lin, M., Hamprecht, F., ... \& Courville, A. $(2019, May)$. On the spectral bias of neural networks. In International conference on machine learning $(pp. 5301-5310). PMLR.$

%18
\bibitem{WangWangPerdikaris2021}
Wang, S., Wang, H., \& Perdikaris, P. $(2021)$. On the eigenvector bias of Fourier feature networks: From regression to solving multi-scale PDEs with physics-informed neural networks. Computer Methods in Applied Mechanics and Engineering, $384, 113938.$

%19
\bibitem{Krishnapriyan2021}
Krishnapriyan, A., Gholami, A., Zhe, S., Kirby, R., \& Mahoney, M. $(2021)$. Characterizing possible failure modes in physics-informed neural networks. Advances in neural information processing systems, $34, 26548-26560.$

%20
\bibitem{Wang2021}
Wang, S., Teng, Y., \& Perdikaris, P. $(2021)$. Understanding and mitigating gradient flow pathologies in physics-informed neural networks. SIAM Journal on Scientific Computing, $43(5), A3055-A3081.$

%21
\bibitem{McFallMahan2009}
McFall, K. S., \& Mahan, J. R. $(2009)$. Artificial neural network method for solution of boundary value problems with exact satisfaction of arbitrary boundary conditions. IEEE transactions on neural networks, $20(8), 1221-1233.$

%22
\bibitem{SukumarSrivastava2022}
Sukumar, N., \& Srivastava, A. $(2022)$. Exact imposition of boundary conditions with distance functions in physics-informed deep neural networks. Computer Methods in Applied Mechanics and Engineering, $389, 114333$.

%23
\bibitem{BeidokhtiMalek2009}
Beidokhti, R. S., \& Malek, A. $(2009)$. Solving initial-boundary value problems for systems of partial differential equations using neural networks and optimization techniques. Journal of the franklin institute, $346(9), 898-913.$

%24
\bibitem{HuJinZhou2024}
Hu, T., Jin, B., \& Zhou, Z. $(2024)$. Solving Poisson problems in polygonal domains with singularity enriched physics informed neural networks. SIAM Journal on Scientific Computing, $46(4), C369-C398.$

%25
\bibitem{Kharazmi2019}
Kharazmi, E., Zhang, Z., \& Karniadakis, G. E. $(2019)$. Variational physics-informed neural networks for solving partial differential equations. arXiv preprint arXiv:$1912.00873.$

%26
\bibitem{Kharazmi2021}
Kharazmi, E., Zhang, Z., \& Karniadakis, G. E. $(2021)$. hp-VPINNs: Variational physics-informed neural networks with domain decomposition. Computer Methods in Applied Mechanics and Engineering, $374, 113547.$

%27
\bibitem{GriffithsLorenz1978}
Griffiths, D. F., \& Lorenz, J. $(1978)$. An analysis of the petrov—galerkin finite element method. Computer Methods in Applied Mechanics and Engineering, $14(1), 39-64.$

%28
\bibitem{BerroneCanutoPintore2022a}
Berrone, S., Canuto, C., \& Pintore, M. $(2022)$. Variational physics informed neural networks: the role of quadratures and test functions. Journal of Scientific Computing, $92(3), 100.$

%29
\bibitem{BerroneCanutoPintore2022b}
Berrone, S., Canuto, C., \& Pintore, M. $(2022)$. Solving PDEs by variational physics-informed neural networks: an a posteriori error analysis: S. Berrone et al. ANNALI DELL'UNIVERSITA'DI FERRARA, $68(2), 575-595$.

%30
\bibitem{Rojas2024}
Rojas, S., Maczuga, P., Muñoz-Matute, J., Pardo, D., \& Paszyński, M. $(2024)$. Robust variational physics-informed neural networks. Computer Methods in Applied Mechanics and Engineering, $425, 116904.$

%31
\bibitem{Taylor2023}
Taylor, J. M., Pardo, D., \& Muga, I. $(2023)$. A deep Fourier residual method for solving PDEs using neural networks. Computer Methods in Applied Mechanics and Engineering, $405, 115850.$

%32
\bibitem{Taylor2024}
Taylor, J. M., Bastidas, M., Calo, V. M., \& Pardo, D. $(2024)$. Adaptive Deep Fourier Residual method via overlapping domain decomposition. Computer Methods in Applied Mechanics and Engineering, $427, 116997.$

%33
\bibitem{BerronePintore2024}
Berrone, S., \& Pintore, M. $(2024)$. Meshfree Variational-Physics-Informed neural networks (MF-VPINN): An adaptive training strategy. Algorithms, $17(9), 415.$

%34
\bibitem{Udomworarat2026}
Udomworarat, T., Brevis, I., van der Zee, K. G., \& Rojas, S. $(2026).$ Neural network approximation in discrete dual norms with adaptive test spaces. arXiv preprint arXiv:$2606.20972.$

%35
\bibitem{Badia2025}
Badia, S., Li, W., \& Martín, A. F. $(2025)$. Adaptive finite element interpolated neural networks. Computer Methods in Applied Mechanics and Engineering, $437, 117806.$

%36
\bibitem{Lu2021}
Lu, L., Meng, X., Mao, Z., \& Karniadakis, G. E. $(2021)$. DeepXDE: A deep learning library for solving differential equations. SIAM review, $63(1), 208-228.$

%37
\bibitem{Wu2023}
Wu, C., Zhu, M., Tan, Q., Kartha, Y., \& Lu, L. $(2023)$. A comprehensive study of non-adaptive and residual-based adaptive sampling for physics-informed neural networks. Computer Methods in Applied Mechanics and Engineering, $403, 115671.$

%38
\bibitem{MaoMeng2023}
Mao, Z., \& Meng, X. $(2023)$. Physics-informed neural networks with residual/gradient-based adaptive sampling methods for solving partial differential equations with sharp solutions. Applied Mathematics and Mechanics, $44(7), 1069-1084.$

%39
\bibitem{Tang2023}
Tang, K., Wan, X., \& Yang, C. $(2023)$. DAS-PINNs: A deep adaptive sampling method for solving high-dimensional partial differential equations. Journal of Computational Physics, $476, 111868.$

%40
\bibitem{Gao2023}
Gao, Z., Yan, L., \& Zhou, T. $(2023)$. Failure-informed adaptive sampling for PINNs. SIAM Journal on Scientific Computing, $45(4), A1971-A1994.$

%41
\bibitem{INIVPINN2026}
Dodge, S., Formisano, A., \& Barmada, S. $(2026)$. INI-VPINN: A variational physics-informed neural network with implicit Neumann and interface handling for multi-material domains with geometric singularities. Journal of Computational Physics, $565, 115328$.

%42
\bibitem{Los2025}
Łoś, M., Służalec, T., Maczuga, P., Vilkha, A., Uriarte, C., \& Paszyński, M. $(2025)$. Collocation-based robust variational physics-informed neural networks (CRVPINNs). Computers \& Structures, $316, 107839$.

%43
\bibitem{Zeidler2013}
Zeidler, E. $(2013)$. Nonlinear functional analysis and its applications: II/B: nonlinear monotone operators. Springer Science \& Business Media.

%44
\bibitem{Tancik2020}
Tancik, M., Srinivasan, P., Mildenhall, B., Fridovich-Keil, S., Raghavan, N., Singhal, U., ... \& Ng, R. $(2020)$. Fourier features let networks learn high frequency functions in low dimensional domains. Advances in neural information processing systems, $33, 7537-7547$.

%45
\bibitem{FixGulatiWakoff1973}
Fix, G. J., Gulati, S., \& Wakoff, G. I. $(1973)$. On the use of singular functions with finite element approximations. Journal of Computational Physics, $13(2), 209-228.$

%46
\bibitem{Moes1999}
Moës, N., Dolbow, J., \& Belytschko, T. $(1999)$. A finite element method for crack growth without remeshing. International journal for numerical methods in engineering, $46(1), 131-150$.

%47
\bibitem{Gu2023}
Gu, Y., Zhang, C., Zhang, P., Golub, M. V., \& Yu, B. $(2023)$. Enriched physics-informed neural networks for $2D$ in-plane crack analysis: Theory and MATLAB code. International Journal of Solids and Structures, $276, 112321$.

%48
\bibitem{Zhu2025}
Zhu, B., Li, H., \& Zhang, Q. $(2025)$. Extended physics-informed extreme learning machine for linear elastic fracture mechanics. Computer Methods in Applied Mechanics and Engineering, $435, 117655$.

%49
\bibitem{Kingma2015}
Kingma, D. P., \& Ba, J. $(2014)$. Adam: A method for stochastic optimization. arXiv preprint arXiv:$1412.6980$.

%50
\bibitem{LiuNocedal1989}
Liu, D. C., \& Nocedal, J. $(1989)$. On the limited memory BFGS method for large scale optimization. Mathematical programming, $45(1), 503-528$.

%51
\bibitem{NocedalWright2006}
Nocedal, J., \& Wright, S. $(2006)$. Numerical Optimization.

%52
\bibitem{Mitchell2013}
Mitchell, W. F. $(2013)$. A collection of $2D$ elliptic problems for testing adaptive grid refinement algorithms. Applied mathematics and computation, $220, 350-364.$

%53
\bibitem{AllenCahn1979}
Allen, S. M., \& Cahn, J. W. $(1979)$. A microscopic theory for antiphase boundary motion and its application to antiphase domain coarsening. Acta metallurgica, $27(6), 1085-1095.$

%54
\bibitem{Paszke2019}
Paszke, A., Gross, S., Massa, F., Lerer, A., Bradbury, J., Chanan, G., ... \& Chintala, S. $(2019)$. Pytorch: An imperative style, high-performance deep learning library. Advances in neural information processing systems, $32$.

%55 
\bibitem{Anderson2017}
Anderson, T. L. (2017). Fracture mechanics: Fundamentals and applications (4th ed.). Boca Raton, FL: CRC Press.

%56
\bibitem{LuHuLi2004}
Lu, T. T., Hu, H. Y., \& Li, Z. C. $(2004)$. Highly accurate solutions of Motz's and the cracked beam problems. Engineering Analysis with Boundary Elements, $28(11), 1387-1403$.

%57
\bibitem{Li2006}
Li, Z. C., Chan, Y. L., Georgiou, G. C., \& Xenophontos, C. $(2006)$. Special boundary approximation methods for Laplace equation problems with boundary singularities—applications to the Motz problem. Computers \& Mathematics with Applications, 51(1), 115-142.

%58
\bibitem{Williams1952}
Williams, M. L. $(1952)$. Stress singularities resulting from various boundary conditions in angular corners of plates in extension.
\end{thebibliography}
\end{document}